\documentclass[11pt, twoside]{amsart}

\title[Dihedral pointed fusion categories: diagrammatics, \(3\)-cocycle representatives, and planar algebras]
{Dihedral pointed fusion categories: diagrammatics, \(3\)-cocycle representatives, and planar algebras}

\author[A. Czenky]{Agustina Czenky}
\address{Czenky: Department of Mathematics, University of Southern California,
3620 S. Vermont Ave., KAP 104. Los Angeles, CA 90089, USA}
\email{czenky@usc.edu}

\author[M. Molander]{Melody Molander}
\address{Molander: Department of Mathematics, The Ohio State University, 100 Math Tower, 231 W 18th Ave, Columbus, OH 43210}
\email{molander.3@osu.edu}

\usepackage{amsmath}
\usepackage{amsthm}
\usepackage{dsfont}
\usepackage{stmaryrd}
\usepackage{MnSymbol}
\usepackage[bb=ams]{mathalpha}
\usepackage{bbm}
\usepackage{appendix}
\usepackage[english]{babel}
\usepackage{url}
\usepackage{fancyhdr}
\usepackage{graphicx}
\usepackage{verbatim}
\usepackage[normalem]{ulem}
\usepackage[shortlabels]{enumitem}
\newcommand{\stkout}[1]{\ifmmode\text{\sout{\ensuremath{#1}}}\else\sout{#1}\fi}
\usepackage[
colorlinks=true,
linkcolor=black, 
anchorcolor=black,
citecolor=black,
urlcolor=black, 
]{hyperref}

\usepackage[all,cmtip]{xy}
\usepackage{geometry}
\usepackage[dvipsnames]{xcolor}
\usepackage{lscape} 
\usepackage{multicol}
\usepackage{shortcuts}
\usepackage[utf8]{inputenc}
\usepackage{tikz}
\usepackage{tikz-cd}
\usepackage{mathrsfs}
\usepackage{adjustbox}
\usepackage{subcaption}

\usetikzlibrary{positioning}
\usetikzlibrary{decorations.pathreplacing}
\allowdisplaybreaks

\usepackage{adjustbox}
\usepackage{overpic}

\usepackage{enumitem}

\usepackage[notcite,notref,final]{showkeys}

\usepackage{calrsfs}
\DeclareMathAlphabet{\cal}{OMS}{zplm}{m}{n}
\DeclareMathAlphabet{\mathsf}{OT1}{cmss}{m}{n} 

\newcommand{\id}{\textnormal{id}}

\newcommand{\Hom}{\textnormal{Hom}}

\newcommand{\KFGR}{\mathcal K(  \mathcal F(\mathsf G)/\langle \mathsf R\rangle )}
\newcommand{\FGR}{ \mathcal F(\mathsf G)/\langle \mathsf R\rangle }

\newcounter{savedequation}

\usepackage{relsize}

\newcommand{\End}{\mathsf{End}}

\newcommand{\Dn}{\mathrm{D}_{2n}}
\newcommand{\VecDn}{\mathrm{Vec}_{\mathrm{D}_{2n}}^{\omega}}
\newcommand{\VecDnx}{\mathrm{Vec}_{\mathrm{D}_{2n}}^{\epsilon_r,\epsilon_b,\xi}}
\newcommand{\Dx}{\mathscr{D}^{\epsilon_r,\epsilon_b,\xi}}

\newcommand{\VecDnpi}{(\VecDn,\pi)}

\newcommand\mtiny[1]{\mbox{\tiny\ensuremath{#1}}}

\numberwithin{equation}{section}

\newtheorem{theorem}{Theorem}[section]

\newtheorem{proposition}[theorem]{Proposition}

\newtheorem{corollary}[theorem]{Corollary}
\newtheorem{lemma}[theorem]{Lemma}

\newtheorem{theorem*}{Theorem}
\newtheorem{algorithm}[theorem]{Algorithm}

\theoremstyle{definition}
\newtheorem{definition}[theorem]{Definition}

\newtheorem{remark}[theorem]{Remark}

\makeatletter              
\let\c@equation\c@theorem  
\makeatother
\numberwithin{equation}{section}

\definecolor{r}{RGB}{208, 2, 27}
\definecolor{b}{RGB}{74, 144, 226}
\definecolor{g}{RGB}{126, 211, 33}
\definecolor{p}{RGB}{144, 19, 254}

\makeatletter
\@namedef{subjclassname@2020}{%
  \textup{2020} Mathematics Subject Classification}
\makeatother



\tikzcdset{scale cd/.style={every label/.append style={scale=#1},
    cells={nodes={scale=#1}}}}

\usetikzlibrary{intersections}

\usetikzlibrary{decorations.markings}
\usetikzlibrary{shapes.geometric}

\begin{document}
	\begin{abstract}
We construct a diagrammatic presentation by generators and relations for all dihedral pointed fusion categories $\VecDn$. Applying a jellyfish algorithm on our diagrammatics, we obtain an explicit closed formula for a complete set of  normalized 3-cocycle representatives of $\Dn$. 
 As applications of our diagrammatics, we prove a universal property and show that in some cases we can recover affine $A$ unshaded subfactor planar algebras. We then extend our presentation to obtain diagrammatics for all pivotal structures on $\VecDn$. 
\end{abstract}
\maketitle
\thispagestyle{plain}
\pagestyle{fancy}

\section{Introduction}
 Just as groups provide a mathematical language for classical symmetries, tensor categories provide one for quantum symmetries, such as those arising in topological phases of matter. 
Fusion categories form a widely studied class of tensor categories which appear naturally in areas such as representation theory \cite{EGNO}, topological quantum field theory \cite{Moore, Turaev},  quantum groups \cite{BK, Kassel}, subfactor theory \cite{EvansKawahigashi} and vertex operator algebras \cite{Huang05}, making their study relevant beyond their intrinsic mathematical interest. A classification of fusion categories is currently out of reach: it would include a classification of finite groups, as their representation categories are fusion categories. Nevertheless, progress can be made by restricting attention to particular classes of fusion categories. 

A fusion category is called \emph{pointed} if every simple object is invertible. Pointed fusion categories admit a complete description in terms of finite group data: they are determined by a finite group $G$ and     $\omega \in H^3(G, \mathsf k^{\times})$; we denote the corresponding fusion category by $\operatorname{Vec}_G^{\omega}$. The group $G$ encodes the fusion rules, while the cocycle $\omega$ describes the associator. Tensor functors between pointed fusion categories can likewise be described in terms of group homomorphisms and cohomological data.  Although this gives a complete classification in terms of groups, it remains abstract and does not always provide a convenient way to carry out explicit calculations.

A presentation by generators and relations describes a tensor category in terms of generating objects and morphisms, together with a complete set of relations. When expressed diagrammatically, such a presentation gives a graphical calculus to  carry out computations explicitly. The first author developed a presentation of this type for cyclic pointed fusion categories in \cite{Czenky}. In the present paper, we extend this program to dihedral groups $\Dn$. Our main result, Theorem~\ref{thm diag category}, proves a diagrammatic presentation for $\VecDn$; the proof uses a jellyfish elimination algorithm. The presentation can be found in Definition \ref{tensor presentaion}. Related diagrammatics for dihedral pointed fusion categories appear in \cite{IMP}. Their construction corresponds to a restricted family, whereas ours applies to all dihedral pointed fusion categories.

One application of our diagrammatics  is the construction of explicit cocycle representatives for the dihedral groups. Cohomology rings of the dihedral groups have been known for a long time: $H^3(\Dn, \mathsf k^{\times})\cong \mathbb Z_{2}\times \mathbb Z_n$ when $n$ is odd, while $H^3(\Dn, \mathsf k^{\times})\cong \mathbb Z_2 \times \mathbb Z_{2}\times \mathbb Z_{n}$ when $n$ is even. However, explicit closed formulas for cocycle representatives have proved harder to obtain. De Wild Propitius gave a closed formula when $n$ is odd and left the even case open \cite{deWildPropitius1995}. Recently, Teo and Hu gave a complete family of representatives in their study of dihedral gauge theories \cite[Equations (212) and (221)]{TeoHu2023}. Their computations treat the odd and even cases separately. 
 In this work, we derive the cocycles directly from our graphical calculus for both odd and even $n$. The formulas can be found in Theorem \ref{thm: cocycle formula}. 

 A second application connects $\VecDn$  with affine $A$ unshaded subfactor planar algebras. The second author developed diagrammatics for affine $A$ unshaded subfactor planar algebras of which she proved that there are two types. One called ``the arrow case" and the other ``the color case". From these planar algebras, one can obtain a fusion category. In \cite{Mol24} she proves that the category corresponding to the arrow case is monoidally equivalent to a cyclic pointed fusion category using the presentation given by \cite{Czenky}. In this paper, Theorem \ref{thm: spa-is-VecD2n} proves, using the constructed diagrammatic presentations, that the categories corresponding to the color case are monoidally equivalent to a dihedral pointed fusion category. Our results together with \cite{Mol24} complete the categorical interpretation of the diagrammatics for affine $A$ unshaded subfactor planar algebras.

 As a further application, Theorem~\ref{thm: univ prop} describes tensor functors from $\VecDn$ to an arbitrary tensor category $\mathcal C$ in terms of objects and morphisms  satisfying explicit relations, and Theorem~\ref{thm: pivotal presentation}  extends our diagrammatic presentation to incorporate pivotal structures. 

 We organize the paper as follows. In Section~\ref{sec: prelim}, we fix our notation.
 In Section~\ref{sec: presentation}, we choose a specific presentation for $\VecDn$. In Section~\ref{sec: cocycles}, we 
 derive the  cocycle formulas. In Section~\ref{sec: equiv}, we prove the main diagrammatic presentation, which is given in Definition~\ref{tensor presentaion}. In Section~\ref{sec: applications}, we look at pivotal structures, the universal property,  and the connection with affine $A$ unshaded subfactor planar algebras. For readers who choose to print this article in black-and-white, we ask that you print using grayscale so that the colors red and blue appear in different shadings of gray.

\section{Acknowledgments}
The authors would like to thank Stephen Bigelow for an initial question that led to this collaboration. The authors would also like to thank David Penneys for numerous helpful conversations. This project began while A.C. and M.M. were participating in the workshop ``Skew Braces, Braids and the Yang-Baxter Equation" at the Banff International Research Station, whose hospitality and support are gratefully acknowledged. A.C. was partially supported by Simons Collaboration Grant No. 999367. M.M. was supported in part by an AMS-Simons Travel Grant. 

Chumeng Di declined to be listed as a coauthor, but conversations with her were incredibly influential on this paper. In particular, Chumeng Di helped construct the pivotal presentation in section \ref{sec: pivotal}, proved Lemma \ref{lemma: twooppositefacingjellysreduce}, verified calculations, created a few diagrams, and provided feedback on an earlier draft of this manuscript. We thank her tremendously for her help.  

\subsection{AI Declaration} AI was used to help create a select few diagrams, which were verified to be correct by the authors.

	\section{Preliminaries}\label{sec: prelim}

	We work over an algebraically closed field $\mathsf k$ of characteristic zero. 
We assume familiarity with tensor categories, and refer to \cite{ENO1,EGNO} for the background. We fix our conventions below.

\subsection{Monoidal notation}

In a monoidal category $\mathcal C$, we use the associator convention
$a_{X,Y,Z}:(X\otimes Y)\otimes Z
\xrightarrow{\sim}
X\otimes(Y\otimes Z)$ for all $X,Y,Z\in \mathcal C$. For a monoidal functor $F:\mathcal C \to \mathcal D$, we write 
$J_{X,Y}:F(X)\otimes F(Y)\xrightarrow{\sim}F(X\otimes Y)$ 
for its monoidal structure. 
Thus
\begin{align}\label{eq:monoidal-functor-coherence}
F(a_{X,Y,Z})\circ J_{X\otimes Y,Z}
\circ(J_{X,Y}\otimes\operatorname{id}_{F(Z)})
={}&
J_{X,Y\otimes Z}
\circ(\operatorname{id}_{F(X)}\otimes J_{Y,Z})
\circ a_{F(X),F(Y),F(Z)}
\end{align}
for all $X,Y,Z\in \mathcal C$.
We identify the
	tensor units; with this convention,
	$F(1_{\mathcal C})= 1_{\mathcal D}$ and
	$J_{1_{\mathcal C},X}=J_{X,1_{\mathcal C}}=\operatorname{id}_{F(X)}$; see
	\cite[Remark~2.4.6]{EGNO}. 

When it exists, we denote the left dual for an object $X\in \mathcal C$ by $X^{\vee}\in \mathcal C$, with evaluation and coevaluation conventions
$\operatorname{ev}_X\colon X^\vee\otimes X\to 1_{\mathcal C}$ and $\operatorname{coev}_X\colon 1_{\mathcal C}\to X\otimes X^\vee.$

	\subsection{Pointed fusion categories}
	
	Let $G$ be a finite group. 
	Given $\omega\in H^3(G,\mathsf k^\times)$, we denote by $\operatorname{Vec}_G^\omega$ the
	category of finite-dimensional $G$-graded vector spaces. 
	For $g\in G$, let $\delta_g$ denote the one-dimensional vector space
	concentrated in degree $g$. The associator is determined by
	\begin{equation}\label{eq:associator-vec-g-omega}
		a_{\delta_g,\delta_h,\delta_k}
		=\omega(g,h,k)\operatorname{id}_{\delta_{ghk}}.
	\end{equation}
	Every
	pointed fusion category is tensor equivalent to
	$\operatorname{Vec}_G^\omega$ for a finite group $G$ and a normalized
	$3$-cocycle $\omega$. Cohomologous cocycles yield equivalent pointed  fusion categories; we will assume without loss of generality that cocycles $\omega$ are normalized. 
	We will often identify the object $\delta_g$ with the group element $g\in G$ and	use the skeletal realization of
	$\operatorname{Vec}_G^\omega$.
	
	\subsection{Presentations by generators and relations}\label{sec: prelim presentation}
	
	Let $\mathcal A$ be a $\mathsf k$-linear category.
 We write $\mathcal K(\mathcal A)
	:=\operatorname{Kar}(\operatorname{Add}(\mathcal A))$ 
	for the additive Karoubi envelope, also called the \textit{Cauchy completion}.
	
	For a monoidal category $\mathcal C$, let $\mathsf G$ be a collection of generating objects and generating
	morphisms between words in those objects. We denote by
	$\mathcal F(\mathsf G)$ the free strict $\mathsf k$-linear monoidal category generated
	by $\mathsf G$. Given a set $\mathsf R$ of relations, we
	write $\langle\mathsf R\rangle$ for the  tensor ideal generated by the 
	differences of the two sides of the relations in $\mathsf R$.
A  \emph{presentation} of a fusion category $\mathcal C$ is a monoidal equivalence
$ 		\mathcal C\cong\mathcal K(  \mathcal F(\mathsf G)/\langle \mathsf R\rangle ), $ 
see \cite[Definition XII.1.3]{Kassel}. Every strict fusion category admits a presentation. 
	
	\subsection{Dihedral notation}
	
	Let $\Dn=\langle r,b\mid r^2=b^2=(rb)^n=e\rangle$
	be the dihedral group of order $2n$. Every element of $\Dn$
	has a unique \emph{standard form}
	\[
	g=(rb)^m r^\ell,
	\qquad 0\leq m<n,\quad \ell\in\{0,1\}. 
	\]
	We write $	|g|:=2m+\ell$
	for the length of $g$ in standard form. 

\subsection{Pivotal categories}\label{sec: pivotal prelim}
Since $\VecDn$ is pointed, it admits a canonical pseudounitary structure, we denote it by $\varphi_0$. The set of pivotal structures on $\VecDn$ is identified with
$\operatorname{Hom}_{\mathrm{Grp}}(\Dn,\mathsf k^\times)$.
 For every group homomorphism 
$\pi:\Dn\to \mathsf k^\times$, we denote the
resulting pivotal category by $\VecDnpi$.

A group homomorphism $\pi: \Dn \to \mathsf k^\times$ is determined by $\pi(r), \pi(b),$ and each can take values $\pm 1.$ When $n$ is odd, $(rb)^n \mapsto 1$, thus $\pi(r) = \pi(b).$  Thus there are 2 possible pivotal structures on $\VecDn$ when $n$ is odd and 4 when $n$ is even.

\subsection{Planar Algebras}  For more detailed definitions and discussion, see \cite{MPS10, Pet10}. A \textit{planar tangle}, $T$, consists of an outside box, along with some finite collection of interior boxes and nonintersecting strands connecting those boxes. Each box is required to have the same number of points on its top and bottom boundary. 

If $\mathcal{P}=\{\mathcal{P}_\ell\}_{\ell\geq 0}$ is a collection of vector spaces over the field $\mathsf k$, then each box in each tangle can be assigned a vector space of $\mathcal{P}$ indexed by the number of strands on its top boundary. When we coherently assign to every tangle $T$ a multilinear map $Z_T$ from the tensor product of the vector spaces of the interior squares to the vector space of the outside square, then $\mathcal{P}$ along with all maps $Z_T$ form a \textit{planar algebra}. 

Let $\mathsf{G}$ be a set of diagrams and $\mathsf{R}$ be a set of relations on those diagrams. Let $\mathcal{P}_\ell$ be the vector space spanned by diagrams with $\ell$ boundary points on top and bottom that are obtained by combining diagrams in $\mathsf{G}$  and strands in a planarly fashion, modulo the relations. Then $\mathcal{P}=\{\mathcal{P}_\ell\}_{\ell \geq 0}$ is the planar algebra \textit{generated} by $\mathsf{G}$ with relations $\mathsf{R}$.

If a planar algebra is checkerboard shaded, evaluable, spherical, and positive definite it is a \textit{subfactor planar algebra}. If we can drop the shading, we obtain an \textit{unshaded subfactor planar algebra}. We commonly refer to both shaded and unshaded subfactor planar algebras as `subfactor planar algebras'. In a subfactor planar algebra there is an anti-linear \textit{adjoint}, $*:\mathcal{P}_\ell \to \mathcal{P}_\ell$, for every $\ell\geq 0$. This defines an inner product on each vector space by $\langle f, g\rangle :=\text{tr}(f^*g)$, where $\text{tr}$ is the map corresponding to the trace tangle, i.e., the tangle of a singular interior box with all of the strands of that box connected to itself on the right. Any subfactor planar algebra has an invariant called the \textit{index}. When the index is finite and greater than or equal to 4, the subfactor planar algebra contains the Temperley-Lieb planar algebra $\mathcal{TL}(\delta)$, where $\delta^2$ is the index.

\section{\texorpdfstring{Choice of presentation for $\VecDn$}{Choice of presentation for $\VecDn$}}\label{sec: presentation}

We start by fixing a presentation for $\VecDn$
\begin{equation}\label{eq: monoidal equiv}
	(F,J) :\VecDn
	\longrightarrow
	\KFGR.
\end{equation}
 In the rest of this section, we will choose convenient object and morphism generators, explicit representatives for the monoidal structure J, and record relations they must satisfy.

\subsection{Generating objects} We begin by choosing the object generators.

\begin{lemma}
    We may assume  $\mathsf  G$ has object generators
$ 	\reddot$ and
	 $\bluedot$,
that \(F(e)\) is the empty word and that \(F(g)\) is the sequence of red and blue dots corresponding to the standard form of \(g\) for  \(g \ne e\).
\end{lemma}

\begin{proof}
Since $F$ is a monoidal equivalence and simple objects in $\VecDn$ are generated by $r$ and $b$, the objects $F(r)$ and $F(b)$ generate the target category.
 Hence, setting $F(r)=\reddot$ and $F(b)=\bluedot$, up to relabeling we may assume that $\mathsf G$ has object generators $\{\reddot \ , \bluedot\}.$

Since $F$ is monoidal, it follows that for any $g \in \Dn$, the object $F(g)$ is isomorphic to the tensor product corresponding to any expression of $g$ in terms of $r$ and $b$.
Hence, up to replacing \(F\) with an isomorphic functor, we may assume that $F$ is as in the statement.
\end{proof}

Thus from now on we assume that objects in $\mathcal F(\mathsf  G)$ are horizontally aligned sequences $\reddot $ and $\bluedot$. The words corresponding to elements in standard form are alternating and begin in red.

\subsection{Generating morphisms.}
We will now choose the generating morphisms for our graphical calculus. We read diagrams from bottom to top. Tensor product is horizontal juxtaposition, composition is vertical stacking, and the unit is the empty boundary. A straight strand denotes the identity morphism of the corresponding object. 

In $\VecDn$, $r^2=e,$ and so 
 $F(r) \otimes F(r) \simeq F(r^2) = F(e)$ 
in $\mathcal K(  \mathcal F(\mathsf G)/\langle \mathsf R\rangle )$.
Choose such an isomorphism and denote it by $\redcap$, and let $\redcup$ denote its inverse. 
We also have that $r(rbr)^2r=e$, and so there exists an isomorphism 
  $F(r) \otimes F(rbr)\otimes F(rbr) \otimes F(r) \to F(e).$ 
 Precomposing with
 $\redcup \bluestrand \redcup \bluestrand \redcup$ 
 we obtain an isomorphism
$\bluedot \bluedot \to \text{empty word}.$ We choose $\bluecap$ to represent this isomorphism and $\bluecup$ for its inverse.
Lastly, from $(rb)^n=e$, we get $F(rb) \otimes F((rb)^{n-1}) \simeq F(e);$ we choose such an isomorphism and represent it by $\rbjellydown$ and its inverse by $\rbjellyup$.\\

The diagrams defined below will be used throughout. 
\begin{definition}\label{def: blue jelly}
    We define 
$     \brjellydown \coloneq \defnofbrjellydown, \brjellyup \coloneq \defnofbrjellyup.$
\end{definition}

\begin{remark}\label{rk:normalization} Up to rescaling, we normalize $\rbjellydown$ so that 
$     \mujelly = \text{id}_1.$ 
\end{remark}

Let $\mathsf S$ denote the collection of morphisms 
\[
	\redstrand,\;
	\bluestrand,\;
	\redcap,\;
	\bluecap,\;
	\redcup,\;
	\bluecup,\;
	\rbjellyup,\;
	\rbjellydown.
	\]

\begin{lemma}\label{lemma: red blue sequence}
	For any finite word $\vec{x}$ in $\{\reddot \ , \bluedot  \ \}$, let $\overline{\vec{x}}$ denote the standard-form word representing the same element,  i.e. $\overline{\vec{x}}$ is   of the form
$ 	(\reddot\bluedot \ )^m \reddot^{\,\ell}$, for $0\le m<n,\ \ell\in\{0,1\}$.
    Then there exists an isomorphism
	$u_{\vec x}:\vec x \xrightarrow{\sim} \overline{\vec x}$ built from morphisms in $\mathsf S$.
\end{lemma}

\begin{proof}
	We first construct an isomorphism
$	\beta:\bluedot\reddot \xrightarrow{\sim} (\reddot\bluedot \ )^{n-1}$ using only morphisms in $\mathsf{S}$.
	Indeed, tensor $\rbjellyup$ on the left by $\bluestrand\redstrand$, and then apply the  maps as shown below:
	\[
	\bluedot\reddot
	\xrightarrow{\ \id\otimes \rbjellyup\ }
	\bluedot\reddot \ (\reddot\bluedot \ )^n
	\xrightarrow{\ \id\otimes \redcap\otimes \id\ }
	\bluedot\bluedot \ (\reddot\bluedot \ )^{n-1}
	\xrightarrow{\ \bluecap\otimes \id\ }
	(\reddot\bluedot \ )^{n-1}.
	\]

    Now let \(\vec x\) be any word  in $\{\reddot \ , \bluedot  \ \}$. 
	If \(\vec x\) contains an adjacent pair \(\reddot\reddot\) or \(\bluedot\bluedot\), apply \(\redcap\) or \(\bluecap\) to remove that pair. 
	Thus we may assume \(\vec x\) is alternating. If it begins with \(\reddot\), then it has the form \((\reddot\bluedot \ )^m\reddot^\ell\). If \(m\ge n\), apply \(\rbjellydown\) to remove one copy of $(\reddot\bluedot \ )^n$; repeating this finitely many times produces the standard form. 
    
     If \(\vec x\) begins with \(\bluedot\), then either \(\vec x=\bluedot\), or \(\vec x\) begins with \(\bluedot\reddot\). In the first case, insert \(\reddot\reddot\) using \(\redcup\) and then apply $\beta$, obtaining
	\[
	\bluedot \xrightarrow{\ \id\otimes \redcup\ } \bluedot\reddot\reddot\xrightarrow{\ \beta\otimes \id\ } (\reddot\bluedot \ )^{n-1}\reddot,
	\]
	which is in standard form.  In the second case, apply \(\beta\) to replace the initial \(\bluedot\reddot\) by \((\reddot\bluedot \ )^{n-1}\); any adjacent equal pair created this way can then be removed by \(\redcap\) or \(\bluecap\). Repeating this procedure  yields the standard form.
    \end{proof}

\begin{lemma}
    We may assume  $\mathsf  G$ has morphism-generators $\mathsf S.$
\end{lemma}

\begin{proof}

	Let $\vec{x}$ be any sequence of red and blue dots in $\KFGR$.  By the lemma above, 
	there is an isomorphism $u_{\vec x}:\vec x \xrightarrow{\sim} \overline{\vec x}$
 built from morphisms in $\mathsf S$.
	Also note that $\overline{\vec{x}}=F(g)$ for some $g\in \Dn$. 

Let $f:\vec{x}\to\vec{y}$ be any morphism, where $\vec{x}, \vec{y}$ are sequences in $\{\reddot \ ,\bluedot \ \}$. Conjugating, we get a morphism 
$u_{\vec y}\circ f\circ u_{\vec x}^{-1}:\overline{\vec x}\to \overline{\vec y}$. 
Since $\VecDn \cong \KFGR$ is pointed and $\overline{\vec x}, \overline{\vec y}$ are images of invertibles, this Hom space is $0$ unless $\overline{\vec x}=\overline{\vec y}$, and in the nonzero case it is one-dimensional. Therefore $f$ is either zero or a scalar multiple of $u_{\vec y}^{-1}\circ u_{\vec x}$. It follows that every morphism between word-objects in $\KFGR$ lies in the $\mathsf k$-linear tensor subcategory generated by $\mathsf S$. Hence we may replace the original generating set of morphisms by $\mathsf S$ without changing the resulting Karoubian envelope. 
\end{proof}

\subsection{\texorpdfstring{Choice of monoidal structure $J$}{Choice of monoidal structure $J$}}\label{Choices for J}

In the following diagrams, purple and green dots represent a dot that could either be red or blue, depending on the parity of the element being considered, and these dots are always opposites in color.

For every $g,h \in \Dn$, we now choose an isomorphism 
 $f_{g,h}: F(g)\otimes F(h) \xrightarrow{\simeq} F(gh)$ 
as follows. 

\noindent	$\diamond$ Case 1: If $| g|$ is even and $|g|+| h|\geq 2n$: 
		$  \rbjellydown\redstrand  \bluestrand \  \mtiny{\ldots}\greenstrand$,
   with  $|g|+|h|-2n$ strands on the right.
   
\noindent	$\diamond$ Case 2: If $|g|$ is even and $| g|+| h| <2n$:
		$  \redstrand \bluestrand \  \mtiny{\ldots}\greenstrand$,
	a total of $|g|+|h|$ strands.
	
\noindent	$\diamond$ Case 3:  If $|g|$ is odd and $| g|\geq | h|$: 
	$	\redstrand \bluestrand \  \mtiny{\ldots} \ \greenstrand$  $ \pmanycaps,$  
	with $|g|-|h|$ strands on the left. 
    
\noindent	$\diamond$ Case 4: If $| g|$ is odd and $| g|<|h|$:
      $  \rmanycaps$ \raisebox{0.2cm}{$\brjellyswitchg$},$ $ 
    with  $2n+|g|-|h|$ dots on the top.

    Since Hom spaces are either 0 or 1-dimensional, the tensor isomorphisms
	\[
	J_{g,h}:F(g)\otimes F(h)\xrightarrow{\ \simeq\ } F(gh), \qquad g,h\in \Dn 
	\]
are scalar multiples of the $f_{g,h}$. 
We show in the following lemmas that $J$ recovers the cocycle $\omega$, and that without loss of generality, we may assume $	J_{g,h}:=f_{g,h}$.

\begin{lemma}\label{lemma: omega} 
 For every triple \(g,h,k\in \Dn\), we have that 
	\[
\omega(g,h,k)	(J_{gh,k}\circ \big(J_{g,h}\otimes\mathrm{id}_{F(k)})\big)
	\;=\;
	\;J_{g,hk}\circ(\mathrm{id}_{F(g)}\otimes J_{h,k}).
	\] 
\end{lemma}

\begin{proof}
	Since $(F,J)$ 
	is monoidal, this follows directly from the coherence condition.
\end{proof}

	\begin{lemma}\label{choice of J}
	Without loss of generality, we may assume that
$ 		J_{g,h}:=f_{g,h}$ 
		 for all $g,h\in \Dn$.
	\end{lemma}

\begin{proof}
   Write \(f_{g,h}:=b(g,h)\,J_{g,h}\) for a normalized 2-cochain $b$. 

  Let $\omega_f$ denote the cocycle recovered from the coherence
condition for the family $\{f_{g,h}\}$. The coherence equation gives
		\[\omega_f(g,h,k)=
		\omega(g,h,k)\,
		\frac{b(h,k)\,b(g,hk)}
		{b(gh,k)\,b(g,h)}= \omega(g,h,k)\cdot (db)(g,h,k).
		\]
Thus $\omega_f$ is cohomologous to $\omega$, and  replacing $\omega$ by $\omega_f$ we may take $J_{g,h}=f_{g,h}$ for all $g,h\in \Dn$.
\end{proof}
    
\subsection{Relations} Having fixed our generating objects and morphisms, and the monoidal structure $J$, we now record relations among these morphisms that follow from their construction.

\begin{proposition}\label{prop:rels}
  The following relations hold in $\mathcal K(  \mathcal F(\mathsf G)/\langle \mathsf R\rangle ):$ 
\begingroup
\renewcommand{\theequation}{(\textnormal{N}\arabic{equation})}
\setcounter{equation}{0}

\noindent
\refstepcounter{equation}\label{rel: redstrandcap}%
\begin{minipage}[t]{0.48\linewidth}
\centering
$\theequation\quad \redstrand \redcap = \omega(r,r,r)\redcap \redstrand$
\end{minipage}%
\hfill
\refstepcounter{equation}\label{rel: bluestrandcap}%
\begin{minipage}[t]{0.48\linewidth}
\centering
$\theequation\quad \bluestrand \bluecap = \omega(rbr,rbr,rbr)\bluecap \bluestrand$
\end{minipage}

\vspace{0.8em}

\noindent
\refstepcounter{equation}\label{rel: redbluejelly}%
\begin{center}
$\theequation\quad \redstrand \bluestrand \rbjellydown
=
\omega(rb,(rb)^{n-1},rb)\rbjellydown \redstrand \bluestrand$
\end{center}

\vspace{0.8em}

\noindent
\refstepcounter{equation}\label{rel: redcircle}%
\begin{minipage}[t]{0.31\linewidth}
\centering
$\theequation\quad \redcircle = \mathrm{id}_{\mathbbm 1}$
\end{minipage}%
\hfill
\refstepcounter{equation}\label{rel: bluecircle}%
\begin{minipage}[t]{0.31\linewidth}
\centering
$\theequation\quad \bluecircle = \mathrm{id}_{\mathbbm 1}$
\end{minipage}%
\hfill
\refstepcounter{equation}\label{rel: twojellys}%
\begin{minipage}[t]{0.31\linewidth}
\centering
$\theequation\quad \rbtwojellys = \mathrm{id}_{\mathbbm 1}$
\end{minipage}

\vspace{0.8em}

\noindent
\refstepcounter{equation}\label{rel: rcupcap}%
\begin{minipage}[t]{0.31\linewidth}
\centering
$\theequation\quad \rcupcap = \redstrand\;\redstrand$
\end{minipage}%
\hfill
\refstepcounter{equation}\label{rel: bcupcap}%
\begin{minipage}[t]{0.31\linewidth}
\centering
$\theequation\quad \bcupcap = \bluestrand\;\bluestrand$
\end{minipage}%
\hfill
\refstepcounter{equation}\label{rel: jellysmultiplied}%
\begin{minipage}[t]{0.31\linewidth}
\centering
$\theequation\quad \rbtwojellysout \ = \redstrand\;\bluestrand\ \cdots\redstrand\;\bluestrand$
\end{minipage}

\endgroup
\end{proposition}

\begin{proof}
    Relations \ref{rel: redcircle}--\ref{rel: jellysmultiplied} follow immediately by construction. 
Relations \ref{rel: redstrandcap} and \ref{rel: redbluejelly} follow from monoidality of $(F,J)$ applied to the triples $(r,r,r)$ and $(rb, (rb)^{n-1}, rb)$, respectively.

To prove \ref{rel: bluestrandcap}, since Hom spaces between simple objects in $\KFGR$ are either zero or one dimensional, there exists a constant $\lambda$ such that $(\widetilde{N2}) \bluestrand \bluecap=\lambda\bluecap \bluestrand$. The monoidal coherence equation  applied to $(rbr,rbr,rbr)$ yields
$\begin{tikzpicture}
    \node at (-0.5, -0.2) {$\omega(rbr,rbr,rbr) $};
      \node at (2,0.2) {$\biggredcap$};
    \node at (2, 0.13) {$\bigbluecap$};
    \node at (1.9, 0.01) {$\redcap$};
     \node at (3.2, 0.01) {$\redstrand$};
       \node at (3.45, 0.01) {$\bluestrand$};
         \node at (3.7, 0.01) {$\redstrand$};
         \node at (4.3, -0.3) {$=$};
    \end{tikzpicture}
\begin{tikzpicture}
    \node at (0, 0.03) {$ \redstrand$};
    \node at (0.25, 0.03) {$\bluestrand$};
     \node at (0.5, 0.03) {$\redstrand$};
      \node at (2,0.2) {$\biggredcap$};
    \node at (2, 0.13) {$\bigbluecap$};
    \node at (1.9, 0.01) {$\redcap$};

    \end{tikzpicture}.$ 
Since we can move from the right-hand side to the left-hand side by applying the sequence of relations \ref{rel: redstrandcap}, $(\widetilde{N2})$ and \ref{rel: redstrandcap}, we obtain that $\lambda=\omega(r,r,r)^{-2}\omega(rbr,rbr,rbr)$. Relation \ref{rel: bluestrandcap} then follows from this and the fact that $\omega(r,r,r)\in \{\pm 1\}$ (see \cite[Lemmas 3.1 and 3.2]{Czenky}).
\end{proof}

We emphasize that at this stage we do not claim that
these relations exhaust the tensor ideal $\langle \mathsf R\rangle$.
Section \ref{sec: equiv} is devoted to proving that
the relations listed above generate $\langle \mathsf R\rangle$. 
From now on, to ease notation we 
let $\epsilon_r:=\omega(r,r,r), \epsilon_b:=\omega(rbr,rbr,rbr), \text{ and } \xi:=\omega(rb, (rb)^{n-1},rb).$ 
	By \cite[Lemmas 3.1 and 3.2]{Czenky}, $\epsilon_r, \epsilon_b\in \{1,-1\}$ and $\xi$ is an $n$-th root of unity.

    \section{Cocycle computation using diagrammatics}\label{sec: cocycles}

We use our diagrammatics to obtain an explicit formula for  normalized $3$-cocycles on $\Dn$. 

    \subsection{Cocycle formula} 
Throughout this section, \(\omega\) denotes the normalized \(3\)-cocycle
determining the associator of \(\VecDn\).

\begin{theorem}\label{thm: cocycle formula}
    For all $g,h,k\in \Dn$, we have  $\omega(g,h,k)=\epsilon_r^{\lceil x \rceil} \epsilon_b^{\lfloor x \rfloor} \xi^{(-1)^{|g|+|h|}\lfloor y \rfloor},$  where
   	{\small 
	\[
x=\begin{cases}
	\dfrac{|g|}{2}, &
\begin{aligned}[t]
		&|g|\text{ even},\ |h|\text{ odd},\ |gh|<|k|,\ |h|<|k|
	\end{aligned}
	\\[1.0em]
	
	\dfrac{|ghk|}{2}, &
	\begin{aligned}[t]
		&|g|\text{ even},\ |h|\text{ odd},\ |g|+|h|\ge 2n>|hk|+|g|,\ |gh|<|k|\le |h|\\
		\text{or }&|g|\text{ odd},\ |h|\text{ even},\ |g|\ge |h|,\ |gh|<|k|,\ |h|+|k|\ge 2n,\ |g|\ge |hk|\\
		\text{or }&|g|,|h|\text{ odd},\ |g|<|h|,\ |gh|+|k|\ge 2n,\ |h|\ge |k|,\ |g|\ge |hk|
	\end{aligned}
	\\[1.0em]
	
\dfrac{|hk|}{2}, &
\begin{aligned}[t]
	&|g|\text{ even},\ |h|\text{ odd},\ |g|+|h|<2n\le |hk|+|g|,\ |gh|\ge |k|>|h|\\
	\text{or }&|g|,|h|\text{ odd},\ |g|<|h|,\ |gh|+|k|\ge 2n,\ |h|<|k|,\ |g|<|hk|
\end{aligned}
	\\[1.0em]
	
\dfrac{|gh|}{2}, &
	\begin{aligned}[t]
		|g|\text{ odd},\ |h|\text{ even},\ |g|<|h|,\ |h|+|k|\ge 2n
	\end{aligned}
	\\[1.0em]
	
\dfrac{|k|}{2}, &
\begin{aligned}[t]
	&|g|\text{ odd},\ |h|\text{ even},\ |g|<|h|,\ |gh|\ge |k|,\ |h|+|k|<2n,\ |g|<|hk|\\
	\text{or }&|g|,|h|\text{ odd},\ |g|\ge |h|,\ |gh|+|k|<2n,\ |h|\ge |k|,\ |g|\ge |hk|
\end{aligned}
	\\[1.0em]
	
	\dfrac{|h|}{2}, &
	\begin{aligned}[t]	|g|,|h|\text{ odd},\ |g|\ge |h|,\ |gh|+|k|\ge 2n,\ |h|<|k|,\ |g|\ge |hk|
	\end{aligned}
	\\[1.0em]
	
	\dfrac{|kg|}{2}, &
		\begin{aligned}[t]
	|g|,|h|\text{ odd},\ |g|\ge |h|,\ |gh|+|k|<2n,\ |h|<|k|,\ |g|<|hk| 
	\end{aligned}
	\\[1.0em]
	
	0, & \text{otherwise}.
\end{cases}
\]}
    and 
   	{\small  \begin{align*}
        y = \begin{cases}
                0 , &\text{ if } \begin{cases}
                |h| \text{ even and } |h|+|k|<2n, \text{ or }\\
                |h| \text { odd and } |h|\geq |k|\end{cases}\\
                \frac{|g|}{2}, &\text{otherwise}.
        \end{cases}
    \end{align*}}
\end{theorem}

\begin{remark} When $\epsilon_r=\epsilon_b=:\epsilon$ (which is automatic when $n$ is odd), the cocycle formula becomes:
\[
\omega(g,h,k)=\epsilon^{2x}\xi^{(-1)^{|g|+|h|}\lfloor y \rfloor}.
\]
Note that for the case $\epsilon=1,$ this formula matches the one in \cite{Dave}.
\end{remark}

For a fixed $n$,
	\(
	\mu_n:=\{\zeta\in \mathsf k^\times\mid \zeta^n=1\}
	\)
	denotes the set of $n$-th roots of unity in $\mathsf k$.
 We call a triple 
 $(\epsilon_r,\epsilon_b,\xi)$
\emph{admissible} when 
\[
\epsilon_r,\epsilon_b\in\{\pm1\},
\qquad
\xi\in\mu_n,
\]
and, when \(n\) is odd, \(\epsilon_r=\epsilon_b\). We denote the set of admissible triples by $\mathcal A_n$.

For each admissible triple $(\epsilon_r,\epsilon_b,\xi)$, let
$\omega_{\epsilon_r,\epsilon_b,\xi}$ 
denote the function given in
Theorem~\ref{thm: cocycle formula}.

\begin{proposition}\label{prop: cocycle parameters invariant}
	Suppose that
	\(\omega_{\epsilon_r,\epsilon_b,\xi}\) and
	\(\omega_{\epsilon'_r,\epsilon'_b,\xi'}\) are normalized
	\(3\)-cocycles. If they are cohomologous, then
$ 	(\epsilon_r,\epsilon_b,\xi)
	=
	(\epsilon'_r,\epsilon'_b,\xi').$ 
\end{proposition}

\begin{proof}
	 Suppose that $\omega' = \omega_{\epsilon'_r,\epsilon'_b,\xi'}$ and $\omega = \omega_{\epsilon_r,\epsilon_b,\xi} $ are cohomologous. Since both cocycles are normalized, we may write  $\omega'=\omega\cdot d\beta$ for a normalized \(2\)-cochain \(\beta\).
	If \(a\in\Dn\) has order two, then \[ (d\beta)(a,a,a) = \frac{\beta(a,a)\beta(a,e)} {\beta(e,a)\beta(a,a)} =1. \] Applying this first to \(a=r\) and then to \(a=rbr\), we obtain 
	$\epsilon'_r = \epsilon_r $ and $ \epsilon'_b = \epsilon_b. $  To  recover the remaining parameter, note that from Theorem~\ref{thm: cocycle formula} we know that
	$\omega(rb,(rb)^j,rb)=1=\omega'(rb,(rb)^j,rb)$ for $0\leq j\leq n-2$,
	$\omega(rb,(rb)^{n-1},rb)=\xi,$ and $\omega'(rb,(rb)^{n-1},rb)=\xi'$.
	Since $\omega'=\omega\cdot d\beta$, it follows that
	$(d\beta)(rb,(rb)^j,rb)=1$ for $0\leq j\leq n-2$.
	It can then be shown by induction that 
	\[
	\beta(rb,(rb)^j)=\beta((rb)^j,rb)
	\qquad
	\text{for }0\leq j\leq n-1.
	\]

Therefore,
	$\xi'= \omega'(rb,(rb)^{n-1},rb) = \omega(rb,(rb)^{n-1},rb)(d\beta)(rb,(rb)^{n-1},rb)=\xi, $ as desired. 
\end{proof}

\begin{proposition}\label{prop: admissible triples give cocycles}
	For every admissible triple
	\((\epsilon_r,\epsilon_b,\xi)\), 
	\(\omega_{\epsilon_r,\epsilon_b,\xi}\) is a normalized
	\(3\)-cocycle.
\end{proposition}

\begin{proof}
	By Lemma~\ref{choice of J}, Proposition~\ref{prop:e_r=e_b}, and
	Theorem~\ref{thm: cocycle formula}, every class in
	\(H^3(\Dn,\mathsf k^{\times})\) has a normalized representative of the
	form \(\omega_{\epsilon_r,\epsilon_b,\xi}\) for some admissible triple. 
 We define a
	map
$ 	\Phi\colon H^3(\Dn,\mathsf{k}^{\times})\longrightarrow\mathcal A_n$ 
	sending each cohomology class to the parameter triple of such a
	representative.
		Proposition~\ref{prop: cocycle parameters invariant} shows that \(\Phi\) is
	well-defined. If two classes have the same parameter triple,
 Theorem~\ref{thm: cocycle formula} shows that their chosen
	representatives are equal, so \(\Phi\) is injective.
	
	Since $\mathsf k$ is algebraically closed of characteristic zero,
	\(\mu_n\) has \(n\) elements. Therefore,
	$ |\mathcal A_n|= 2n$ when $n$ is odd, and $4n$ when $n$ is even. 
	On the other hand, the known computation
	\[
	H^3(\Dn,\mathsf k^{\times})
	\cong
	\begin{cases}
		\mathbb Z_{2n}, & n\text{ odd},\\
		\mathbb Z_2\times\mathbb Z_2\times\mathbb Z_n,
		& n\text{ even},
	\end{cases}
	\]
	shows that \(H^3(\Dn,\mathsf k^{\times})\) has the same cardinality.
	Thus \(\Phi\) is bijective.
		Consequently, every admissible triple occurs as the parameter triple of
	a normalized \(3\)-cocycle. 
\end{proof}

\begin{theorem}\label{thm: cocycle reps}The family
\(
\bigl\{
\omega_{\epsilon_r,\epsilon_b,\xi}
\mid
(\epsilon_r,\epsilon_b,\xi)\text{ is admissible}
\bigr\}
\)
is a complete set of normalized \(3\)-cocycle representatives for
\(H^3(\Dn,\mathsf{k}^{\times})\).
\end{theorem}

	\begin{proof}
    By Proposition~\ref{prop: admissible triples give cocycles}, the family in the statement 
consists of normalized \(3\)-cocycles; its proof also shows that every cohomology class occurs. 
		Uniqueness follows by Proposition~\ref{prop: cocycle parameters invariant}.
\end{proof}

\subsection{Auxiliary diagrammatic relations}
We derive several
consequences of relations \ref{rel: redstrandcap}—\ref{rel: jellysmultiplied}. These
auxiliary relations will be used repeatedly.

\begin{lemma}\label{lemma: extra rels}
	The following relations hold in $\KFGR$
    \setcounter{savedequation}{\value{equation}}
\begingroup
\renewcommand{\theequation}{(\textnormal{N}\arabic{equation})}
\renewcommand{\theHequation}{N\arabic{equation}}
\setcounter{equation}{9}

\noindent
\refstepcounter{equation}\label{rel: redstrandcup}%
\begin{minipage}[t]{0.48\linewidth}
\makebox[3.5em][l]{$\theequation$}%
$\redstrand \redcup
=
\epsilon_r\redcup \redstrand$
\end{minipage}%
\hfill
\refstepcounter{equation}\label{rel: bluestrandcup}%
\begin{minipage}[t]{0.48\linewidth}
\makebox[3.5em][l]{$\theequation$}%
$\bluestrand \bluecup
=
\epsilon_b\bluecup \bluestrand$
\end{minipage}

\vspace{0.8em}

\noindent
\refstepcounter{equation}\label{rel: zigzag}%
\begin{minipage}[t]{0.48\linewidth}
\makebox[3.5em][l]{$\theequation$}%
$\gzigzagup
=
\epsilon_x\greenstrand
=
\gzigzagdown$
\end{minipage}%
\hfill
\refstepcounter{equation}\label{rel: brtwojellys}%
\begin{minipage}[t]{0.48\linewidth}
\makebox[3.5em][l]{$\theequation$}%
$\brtwojellys
=
\mathrm{id}_{\mathbbm 1}$
\end{minipage}

\vspace{0.8em}

\noindent
\refstepcounter{equation}\label{rel: blueredjelly}%
\begin{minipage}[t]{0.48\linewidth}
\makebox[3.5em][l]{$\theequation$}%
$\bluestrand \redstrand \brjellydown
=
\xi\brjellydown \bluestrand \redstrand$
\end{minipage}%
\hfill
\refstepcounter{equation}\label{rel: brtwojellysout}%
\begin{minipage}[t]{0.48\linewidth}
\makebox[3.5em][l]{$\theequation$}%
$\brtwojellysout
=
\bluestrand \redstrand \ \cdots
\bluestrand \redstrand$
\end{minipage}
\begin{minipage}[t]{0.48\linewidth}
\refstepcounter{equation}\label{eq: blue up}%
\makebox[\linewidth][l]{%
    \(\theequation\quad
    \rwithbup=\epsilon_b\xi\bwithbup\)
} 
\end{minipage}
\hfill
\begin{minipage}[t]{0.48\linewidth}
    \refstepcounter{equation}\label{eq: red up}%
\makebox[\linewidth][l]{%
    \(\theequation\quad
    \bwithrup=\epsilon_r\rwithrup\)
}
\end{minipage}
\endgroup
where \(x=r\) when $\greendot \ $ represents $\reddot\ $ and \(x=b\) when
it represents $\bluedot$ .
\end{lemma}

\begin{proof}
Relations \ref{rel: redstrandcup} and \ref{rel: bluestrandcup} follow from \ref{rel: redstrandcap} and  \ref{rel: rcupcap}, and 
\ref{rel: bluestrandcap} and  \ref{rel: bcupcap}, respectively. 
We prove relation \ref{rel: zigzag} for red dots; the blue dots case is analogous. First compose \ref{rel: redstrandcap} below by 
$
    \redcup \redstrand \text{ or } \redstrand \redcup,
$ 
then reduce the red bubble using \ref{rel: redcircle}. 
For \ref{rel: brtwojellys}, first use \ref{rel: rcupcap} to recognize the left-hand side of the equality as 
$     \rbjellysinredcircle$.
Then using \ref{rel: twojellys} followed by \ref{rel: redcircle} gives the relation. 

 For \ref{rel: blueredjelly}, we may add a red strand to the left of diagrams on both sides, then show the equality. By Definition \ref{def: blue jelly} and \ref{rel: redstrandcup}, \ref{rel: zigzag}, the left diagram with the additional strand is
\[
 \vcenter{\hbox{\begin{tikzpicture}[scale=1]
    \draw[thick] (0,0) -- (0,.6);
    \draw[thick] (0.3,0) -- (0.3,.6);
    \draw[thick] (.6,0) -- (.6,.6);
    \draw[thick,name path=arch] (.9,0) arc (180:0:.55cm);
    \path[name path=probe1] (1.2,0) -- (1.2,.6);
    \path[name intersections={of=arch and probe1, by=Q1}];
    \draw[thick] (1.2,0) -- (Q1); 
    \path[name path=probe2] (1.7,0) -- (1.7,.6);
    \path[name intersections={of=arch and probe2, by=Q2}];
    \draw[thick] (1.7,0) -- (Q2); 
    \foreach \x in {0, .6,1.2,2}
        \filldraw[r] (\x,0) circle (1.5pt);
    \foreach \x in {0.3,.9,1.7}
        \filldraw[b] (\x,0) circle (1.5pt);
    \foreach \x in {0, .6}
        \filldraw[r] (\x,0.6) circle (1.5pt);
    \foreach \x in {0.3}
        \filldraw[b] (\x,0.6) circle (1.5pt);
    \node at (1.45,.1) {$\scriptstyle\cdots$};
\end{tikzpicture}
}}
=
\vcenter{\hbox{\begin{tikzpicture}[scale=1]
    \draw[thick] (-.3,-.15) -- (-.3,.8);
    \draw[thick] (0,-.15) -- (0,.8);
    \draw[thick] (0.3,-.15) -- (0.3,.8);
    \draw[thick,name path=arch] (.9,0) arc (180:0:.55cm) -- (2,-.15);
    \draw[thick] (.9,0) arc (0:-180:.15cm);
    \draw[thick] (.6,0) arc (180:0:.85cm) -- (2.3,-.15);
    \path[name path=probe1] (1.2,-.15) -- (1.2,.6);
    \path[name intersections={of=arch and probe1, by=Q1}];
    \draw[thick] (1.2,-.15) -- (Q1); 
    \path[name path=probe2] (1.7,-.15) -- (1.7,.6);
    \path[name intersections={of=arch and probe2, by=Q2}];
    \draw[thick] (1.7,-.15) -- (Q2); 
    \foreach \x in {-.3,.3,1.7,2.3}
        \filldraw[r] (\x,-.15) circle (1.5pt);
    \foreach \x in {0,1.2,2}
        \filldraw[b] (\x,-.15) circle (1.5pt);
    \foreach \x in {-.3,.3}
        \filldraw[r] (\x,0.8) circle (1.5pt);
    \foreach \x in {.9}
        \filldraw[r] (\x,0) circle (1.5pt);
    \foreach \x in {0}
        \filldraw[b] (\x,0.8) circle (1.5pt);
    \node at (1.45,.1) {$\scriptstyle\cdots$};
\end{tikzpicture}
}}
=
\vcenter{\hbox{\begin{tikzpicture}[scale=1]
    \draw[thick] (2.3,0) -- (2.3,.6);
    \draw[thick] (0.3,0) -- (0.3,.6);
    \draw[thick] (.6,0) -- (.6,.6);
    \draw[thick,name path=arch] (.9,0) arc (180:0:.55cm);
    \path[name path=probe1] (1.2,0) -- (1.2,.6);
    \path[name intersections={of=arch and probe1, by=Q1}];
    \draw[thick] (1.2,0) -- (Q1); 
    \path[name path=probe2] (1.7,0) -- (1.7,.6);
    \path[name intersections={of=arch and probe2, by=Q2}];
    \draw[thick] (1.7,0) -- (Q2); 
    \foreach \x in {0.3,.9,1.7,2.3}
        \filldraw[r] (\x,0) circle (1.5pt);
    \foreach \x in {0.6,1.2,2}
        \filldraw[b] (\x,0) circle (1.5pt);
    \foreach \x in {0.3,2.3}
        \filldraw[r] (\x,0.6) circle (1.5pt);
    \foreach \x in {0.6}
        \filldraw[b] (\x,0.6) circle (1.5pt);
    \node at (1.45,.1) {$\scriptstyle\cdots$};
\end{tikzpicture}
}},
\]
which is equal to
\(\xi \hspace{.2cm}
\vcenter{\hbox{\begin{tikzpicture}[scale=1]
    \draw[thick] (2.3,0) -- (2.3,.6);
    \draw[thick] (2.6,0) -- (2.6,.6);
    \draw[thick] (2.9,0) -- (2.9,.6);
    \draw[thick,name path=arch] (.9,0) arc (180:0:.55cm);
    \path[name path=probe1] (1.2,0) -- (1.2,.6);
    \path[name intersections={of=arch and probe1, by=Q1}];
    \draw[thick] (1.2,0) -- (Q1); 
    \path[name path=probe2] (1.7,0) -- (1.7,.6);
    \path[name intersections={of=arch and probe2, by=Q2}];
    \draw[thick] (1.7,0) -- (Q2); 
    \foreach \x in {0.9,1.7,2.3,2.9}
        \filldraw[r] (\x,0) circle (1.5pt);
    \foreach \x in {1.2,2,2.6}
        \filldraw[b] (\x,0) circle (1.5pt);
    \foreach \x in {2.3,2.9}
        \filldraw[r] (\x,0.6) circle (1.5pt);
    \foreach \x in {2.6}
        \filldraw[b] (\x,0.6) circle (1.5pt);
    \node at (1.45,.1) {$\scriptstyle\cdots$};
\end{tikzpicture}
}}
\)
by \ref{rel: redbluejelly}.
The right diagram with the additional strand is
\[
\vcenter{\hbox{\begin{tikzpicture}[scale=1]
    \draw[thick] (.6,0) -- (.6,.6);
    \draw[thick] (2.6,0) -- (2.6,.6);
    \draw[thick] (2.3,0) -- (2.3,.6);
    \draw[thick,name path=arch] (.9,0) arc (180:0:.55cm);
    \path[name path=probe1] (1.2,0) -- (1.2,.6);
    \path[name intersections={of=arch and probe1, by=Q1}];
    \draw[thick] (1.2,0) -- (Q1); 
    \path[name path=probe2] (1.7,0) -- (1.7,.6);
    \path[name intersections={of=arch and probe2, by=Q2}];
    \draw[thick] (1.7,0) -- (Q2); 
    \foreach \x in {0.6,1.2,2,2.6}
        \filldraw[r] (\x,0) circle (1.5pt);
    \foreach \x in {.9,1.7,2.3}
        \filldraw[b] (\x,0) circle (1.5pt);
    \foreach \x in {.6,2.6}
        \filldraw[r] (\x,0.6) circle (1.5pt);
    \foreach \x in {2.3}
        \filldraw[b] (\x,0.6) circle (1.5pt);
    \node at (1.45,.1) {$\scriptstyle\cdots$};
\end{tikzpicture}
}}
=
\vcenter{\hbox{\begin{tikzpicture}[scale=1]
    \draw[thick] (.3,-.15) -- (.3,.8);
    \draw[thick] (2.6,-.15) -- (2.6,.8);
    \draw[thick] (2.9,-.15) -- (2.9,.8);
    \draw[thick,name path=arch] (.9,0) arc (180:0:.55cm) -- (2,-.15);
    \draw[thick] (.9,0) arc (0:-180:.15cm);
    \draw[thick] (.6,0) arc (180:0:.85cm) -- (2.3,-.15);
    \path[name path=probe1] (1.2,-.15) -- (1.2,.6);
    \path[name intersections={of=arch and probe1, by=Q1}];
    \draw[thick] (1.2,-.15) -- (Q1); 
    \path[name path=probe2] (1.7,-.15) -- (1.7,.6);
    \path[name intersections={of=arch and probe2, by=Q2}];
    \draw[thick] (1.7,-.15) -- (Q2); 
    \foreach \x in {.3,1.7,2.3,2.9}
        \filldraw[r] (\x,-.15) circle (1.5pt);
    \foreach \x in {1.2,2,2.6}
        \filldraw[b] (\x,-.15) circle (1.5pt);
    \foreach \x in {.3,2.9}
        \filldraw[r] (\x,0.8) circle (1.5pt);
    \foreach \x in {.9}
        \filldraw[r] (\x,0) circle (1.5pt);
    \foreach \x in {2.6}
        \filldraw[b] (\x,0.8) circle (1.5pt);
    \node at (1.45,.1) {$\scriptstyle\cdots$};
\end{tikzpicture}
}}
=
\vcenter{\hbox{\begin{tikzpicture}[scale=1]
    \draw[thick] (2.3,0) -- (2.3,.6);
    \draw[thick] (2.6,0) -- (2.6,.6);
    \draw[thick] (2.9,0) -- (2.9,.6);
    \draw[thick,name path=arch] (.9,0) arc (180:0:.55cm);
    \path[name path=probe1] (1.2,0) -- (1.2,.6);
    \path[name intersections={of=arch and probe1, by=Q1}];
    \draw[thick] (1.2,0) -- (Q1); 
    \path[name path=probe2] (1.7,0) -- (1.7,.6);
    \path[name intersections={of=arch and probe2, by=Q2}];
    \draw[thick] (1.7,0) -- (Q2); 
    \foreach \x in {0.9,1.7,2.3,2.9}
        \filldraw[r] (\x,0) circle (1.5pt);
    \foreach \x in {1.2,2,2.6}
        \filldraw[b] (\x,0) circle (1.5pt);
    \foreach \x in {2.3,2.9}
        \filldraw[r] (\x,0.6) circle (1.5pt);
    \foreach \x in {2.6}
        \filldraw[b] (\x,0.6) circle (1.5pt);
    \node at (1.45,.1) {$\scriptstyle\cdots$};
\end{tikzpicture}
}}
\]
by the same relations. This proves \ref{rel: blueredjelly}.

For \ref{rel: brtwojellysout}, first use \ref{rel: rcupcap} on the left side of \ref{rel: brtwojellysout} to get 
$     \brjellysoutRseven \longredstrand$.  Then use \ref{rel: jellysmultiplied} to obtain
$ 
    \redcircle \bluestrand \redstrand \  \cdots \bluestrand \redstrand.
$
Finally, popping the red circle with \ref{rel: redcircle} gives the result.
 For \ref{eq: red up}, we compute
$    \bwithrup=\updownupred= \epsilon_r \rwithrup,$ 
    where in the first equality we are using Definition \ref{def: blue jelly}, and in the second \ref{rel: zigzag}.

 For $\ref{eq: blue up}$, 
it suffices to show that after tensoring on the right by $\text{id}_b$, and composing both sides with $\bluecap$ followed by $\rbjellyup$, $c=\epsilon_b\xi$ satisfies
$\vcenter{\hbox{\scalebox{0.5}{
   \tikzset{every picture/.style={line width=0.75pt}} 
\begin{tikzpicture}[x=0.75pt,y=0.75pt,yscale=-1,xscale=1]

\draw [line width=1.5]    (372.17,617) -- (372.33,636.5) ;
\draw [line width=1.5]    (278.33,635.83) .. controls (277.33,582.83) and (357.33,583.83) .. (357.33,635.83) ;
\draw [line width=1.5]    (297.33,601.83) -- (297.33,634.83) ;
\draw [line width=1.5]    (338.33,602.83) -- (338.33,635.83) ;
\draw [line width=1.5]    (357.33,635.83) .. controls (357.33,646.5) and (372.33,647.5) .. (372.33,636.5) ;
\draw [line width=1.5]    (372.17,617) .. controls (372,596.5) and (390.33,596.17) .. (390,615.83) ;
\draw [line width=1.5]    (390,615.83) -- (390.01,633) ;
\draw [line width=1.5]    (278.67,632.82) .. controls (277.67,685.82) and (371,688.5) .. (389.67,635.17) ;
\draw [line width=1.5]    (297.67,666.82) -- (297.67,633.82) ;
\draw [line width=1.5]    (338.67,672.83) -- (338.67,632.82) ;
\draw  [draw opacity=0][fill={rgb, 255:red, 208; green, 2; blue, 27 }  ,fill opacity=1 ] (273.32,633.82) .. controls (273.33,631.06) and (275.57,628.83) .. (278.33,628.83) .. controls (281.09,628.84) and (283.32,631.08) .. (283.32,633.84) .. controls (283.31,636.6) and (281.07,638.83) .. (278.31,638.82) .. controls (275.55,638.82) and (273.32,636.58) .. (273.32,633.82) -- cycle ;
\draw  [draw opacity=0][fill={rgb, 255:red, 74; green, 144; blue, 226 }  ,fill opacity=1 ] (292.34,633.82) .. controls (292.34,631.06) and (294.59,628.83) .. (297.34,628.84) .. controls (300.1,628.84) and (302.34,631.09) .. (302.33,633.85) .. controls (302.32,636.6) and (300.08,638.84) .. (297.32,638.83) .. controls (294.56,638.82) and (292.33,636.58) .. (292.34,633.82) -- cycle ;
\draw  [draw opacity=0][fill={rgb, 255:red, 208; green, 2; blue, 27 }  ,fill opacity=1 ] (333.32,633.82) .. controls (333.33,631.06) and (335.57,628.83) .. (338.33,628.83) .. controls (341.09,628.84) and (343.32,631.08) .. (343.32,633.84) .. controls (343.31,636.6) and (341.07,638.83) .. (338.31,638.82) .. controls (335.55,638.82) and (333.32,636.58) .. (333.32,633.82) -- cycle ;
\draw  [draw opacity=0][fill={rgb, 255:red, 74; green, 144; blue, 226 }  ,fill opacity=1 ] (384.83,635.15) .. controls (384.84,632.49) and (387.01,630.33) .. (389.68,630.33) .. controls (392.35,630.34) and (394.51,632.51) .. (394.5,635.18) .. controls (394.49,637.85) and (392.32,640.01) .. (389.65,640) .. controls (386.99,639.99) and (384.83,637.82) .. (384.83,635.15) -- cycle ;
\draw [line width=1.5]    (460.67,636.83) .. controls (459.67,583.83) and (539.67,584.83) .. (539.67,636.83) ;
\draw [line width=1.5]    (479.67,602.83) -- (479.67,635.83) ;
\draw [line width=1.5]    (520.67,603.83) -- (520.67,636.83) ;
\draw [line width=1.5]    (445.65,634.16) .. controls (445.65,644.83) and (460.65,645.83) .. (460.65,634.83) ;
\draw  [draw opacity=0][fill={rgb, 255:red, 208; green, 2; blue, 27 }  ,fill opacity=1 ] (534.67,636.82) .. controls (534.68,634.06) and (536.92,631.83) .. (539.68,631.84) .. controls (542.44,631.84) and (544.67,634.09) .. (544.66,636.85) .. controls (544.66,639.6) and (542.41,641.84) .. (539.65,641.83) .. controls (536.9,641.82) and (534.66,639.58) .. (534.67,636.82) -- cycle ;
\draw [line width=1.5]    (445.65,634.16) .. controls (445.65,564.16) and (558.33,566.67) .. (558.33,636.67) ;
\draw [line width=1.5]    (479.33,636.67) .. controls (478.33,689.67) and (558.33,688.67) .. (558.33,636.67) ;
\draw [line width=1.5]    (520.65,674.83) -- (520.67,638.67) ;
\draw [line width=1.5]    (539.33,669.67) -- (539.33,636.67) ;
\draw  [draw opacity=0][fill={rgb, 255:red, 74; green, 144; blue, 226 }  ,fill opacity=1 ] (553.34,636.65) .. controls (553.34,633.9) and (555.59,631.66) .. (558.35,631.67) .. controls (561.1,631.68) and (563.34,633.92) .. (563.33,636.68) .. controls (563.32,639.44) and (561.08,641.67) .. (558.32,641.66) .. controls (555.56,641.66) and (553.33,639.41) .. (553.34,636.65) -- cycle ;
\draw  [draw opacity=0][fill={rgb, 255:red, 208; green, 2; blue, 27 }  ,fill opacity=1 ] (534.67,636.82) .. controls (534.68,634.06) and (536.92,631.83) .. (539.68,631.84) .. controls (542.44,631.84) and (544.67,634.09) .. (544.66,636.85) .. controls (544.66,639.6) and (542.41,641.84) .. (539.65,641.83) .. controls (536.9,641.82) and (534.66,639.58) .. (534.67,636.82) -- cycle ;
\draw  [draw opacity=0][fill={rgb, 255:red, 208; green, 2; blue, 27 }  ,fill opacity=1 ] (474.67,635.82) .. controls (474.68,633.06) and (476.92,630.83) .. (479.68,630.84) .. controls (482.44,630.84) and (484.67,633.09) .. (484.66,635.85) .. controls (484.66,638.6) and (482.41,640.84) .. (479.65,640.83) .. controls (476.9,640.82) and (474.66,638.58) .. (474.67,635.82) -- cycle ;
\draw  [draw opacity=0][fill={rgb, 255:red, 74; green, 144; blue, 226 }  ,fill opacity=1 ] (515.67,636.82) .. controls (515.68,634.06) and (517.92,631.83) .. (520.68,631.84) .. controls (523.44,631.84) and (525.67,634.09) .. (525.66,636.85) .. controls (525.66,639.6) and (523.41,641.84) .. (520.65,641.83) .. controls (517.9,641.82) and (515.66,639.58) .. (515.67,636.82) -- cycle ;

\draw (395,621.67) node [anchor=north west][inner sep=0.75pt]  [font=\huge]  {$=c$};
\draw (305,628) node [anchor=north west][inner sep=0.75pt]    {$\dotsc \ \ \ \ \ \ \ \ \ \ \ \ \ \ \ \ \ \ \ \ \ \ \ \ \ \ \ \ \ \ \ \ \ \ \dotsc $};
\end{tikzpicture}}}}.$ 
Applying  \ref{rel: zigzag} and Remark \ref{rk:normalization},  the left hand side equals $\epsilon_b.$ For the right hand side, we compute
\begin{equation*}
    \relation
\end{equation*}
   Therefore $c=\epsilon_b\xi$, as desired. 
\end{proof}

\setcounter{equation}{\value{savedequation}}

\begin{proposition}\label{prop:e_r=e_b}
    When $n$ is odd, $\epsilon_r=\epsilon_b$.
\end{proposition}

\begin{proof}
    Assume $n$ is odd. Note first that 
\begin{align}\label{eq:er=eb}
    \adjustedoctopus = \epsilon_r \epsilon_b \rbjellydown, 
\end{align}
    where we are using relation \ref{rel: zigzag} $n$ times for each color, and that $\epsilon_r^n\epsilon_b^n=\epsilon_r\epsilon_b$ since $n$ is odd. 
    
  Since Hom spaces are 1-dimensional, there exist $\lambda,\tilde\lambda\in k^{\times}$ such that $
        \octopusr  = \lambda \rbjellyup \qquad$  and $\octopusb = \tilde \lambda \rbjellydown.$ 
    Composing the left-hand sides of these two equations, we obtain 
    \begin{equation*}
        \octopuscomp = \circles =\text{id}_1
    \end{equation*}
    where we are using \ref{rel: twojellys} on the first equality and \ref{rel: redcircle} and \ref{rel: bluecircle} on the second. Composing the right-hand sides yields $\lambda \tilde \lambda \ \text{id}_1$ by \ref{rel: twojellys}. We then conclude $\lambda=\tilde\lambda^{-1}.$

Finally, note that 
\begin{equation*}
    \adjustedoctopus = \lambda \octopusb = \lambda \tilde \lambda \rbjellydown = \rbjellydown.
\end{equation*}
    Combining this with Equation \eqref{eq:er=eb}, we conclude $\epsilon_r=\epsilon_b$, as desired. 
\end{proof}

\begin{lemma}\label{lem: brtwojellys} 
    In $\KFGR$,
$\mujellybr=\operatorname{id}_1$.
\end{lemma}

\begin{proof}
    First, use \ref{rel: zigzag} on the innermost connected cup. Then use \ref{eq: red up} twice to obtain $\epsilon_r$ multiplied by 
  $  
        \twojellyspfpicones.
$
Using relation \ref{rel: rcupcap} on the innermost red strands we then obtain $\epsilon_r$ multiplied by
 $        \twojellyspfpictwos$.
The result then follows from \ref{rk:normalization}, \ref{rel: zigzag}, and \ref{rel: redcircle}. 
\end{proof}

\subsection{Diagrammatic computations}\label{sec:diag comp}
For $g,h,k\in \Dn$, call $J_{gh,k}\circ (J_{g,h}\otimes \text{id}_{F(k)})$ the \textit{left tree} and $J_{g,hk}\circ (\text{id}_{F(g)}\otimes J_{h,k})$ the \textit{right tree} (of $g,h,k$).
By Lemma \ref{lemma: omega}, the monoidal coherence equation says the right three is $\omega(g,h,k)$ times the left tree.  We apply the first steps of the \textit{jellyfish algorithm} of \cite{BMPS12} to transform the left tree into the right tree (or vice versa) using our relations. 

 For each choice of the parities of \(|g|\) and \(|h|\), we treat all possible cases. 
In spirit of the named jellyfish algorithm, we call $\rbjellydown $ a \textit{jellyfish}. We will also refer to $\redcap$ and $\bluecap$ as \textit{caps}.

\subsubsection{The even--even case}Suppose $|g|$ and $|h|$ are both even, and let  $k\in \Dn$. We label the possible left and right trees below.

\noindent
\begin{minipage}[t]{0.485\linewidth}
\vspace{0pt}
\begin{enumerate}[
    label=(L\arabic*$_{ee}$), start=1,
    leftmargin=*, topsep=0pt, partopsep=0pt,
    itemsep=0pt, parsep=0pt
]
    \item\label{L1ee}
    If $\begin{cases}
        |g|+|h|\geq 2n, \text{ and }\\
        |gh|+|k|\geq 2n
    \end{cases}$
    \par\nobreak\vspace{1pt}\nointerlineskip
    \noindent
    \begin{minipage}[c][2.1cm][t]{\linewidth}
    \begin{adjustbox}{max width=\linewidth}
    \begin{tikzpicture}
        \node at (0,0) {$\rbjellyspreadbr$};
        \node at (2.5,0) {$\rbjellyspreadbr$};
        \node at (3.8,0.03) {$\redstrand$};
        \node at (4.5,0.03) {$\greenstrand$};
        \node[scale=0.5] at (4.25,-0.3) {$\cdots$};

        \draw[decorate, decoration={brace,mirror}]
            (-1.2,-0.5) -- (-0.1,-0.5)
            node[midway, below=3pt] {$g$};
        \draw[decorate, decoration={brace,mirror}]
            (0,-0.5) -- (2.4,-0.5)
            node[midway, below=3pt] {$h$};
        \draw[decorate, decoration={brace,mirror}]
            (2.5,-0.5) -- (4.7,-0.5)
            node[midway, below=3pt] {$k$};
    \end{tikzpicture}
    \end{adjustbox}
    \end{minipage}
\end{enumerate}
\end{minipage}%
\hfill
\begin{minipage}[t]{0.485\linewidth}
\vspace{0pt}
\begin{enumerate}[
    label=(R\arabic*$_{ee}$), start=1,
    leftmargin=*, topsep=0pt, partopsep=0pt,
    itemsep=0pt, parsep=0pt
]
    \item\label{R1ee}
    If $\begin{cases}
        |h|+|k|\geq 2n, \text{ and }\\
        |hk|+|g|\geq 2n
    \end{cases}$
    \par\nobreak\vspace{1pt}\nointerlineskip
    \noindent
    \begin{minipage}[c][2.1cm][t]{\linewidth}
    \begin{adjustbox}{max width=\linewidth}
    \begin{tikzpicture}
        \node at (0,0) {$\rbbbigjellyspreadbr$};
        \node at (0,-0.18) {$\rbjellyspreadbr$};
        \node at (2.05,-0.15) {$\redstrand$};
        \node at (2.7,-0.15) {$\greenstrand$};
        \node[scale=0.5] at (2.5,-0.48) {$\cdots$};

        \draw[decorate, decoration={brace,mirror}]
            (-2,-0.7) -- (-1.25,-0.7)
            node[midway, below=3pt] {$g$};
        \draw[decorate, decoration={brace,mirror}]
            (-1.15,-0.7) -- (-0.1,-0.7)
            node[midway, below=3pt] {$h$};
        \draw[decorate, decoration={brace,mirror}]
            (0,-0.7) -- (2.9,-0.7)
            node[midway, below=3pt] {$k$};
    \end{tikzpicture}
    \end{adjustbox}
    \end{minipage}
\end{enumerate}
\end{minipage}
\par
\penalty0
\vspace{0.6em}

\noindent
\begin{minipage}[t]{0.485\linewidth}
\vspace{0pt}
\begin{enumerate}[
    label=(L\arabic*$_{ee}$), start=2,
    leftmargin=*, topsep=0pt, partopsep=0pt,
    itemsep=0pt, parsep=0pt
]
    \item\label{L2ee}
    If $\begin{cases}
        |g|+|h|\geq 2n, \text{ and }\\
        |gh|+|k|<2n
    \end{cases}$
    \par\nobreak\vspace{1pt}\nointerlineskip
    \noindent
    \begin{minipage}[c][2.1cm][t]{\linewidth}
    \begin{adjustbox}{max width=\linewidth}
    \begin{tikzpicture}
        \node at (0,0) {$\rbjellyspreadbr$};
        \node at (1.25,0.03) {$\redstrand$};
        \node at (1.95,0.03) {$\bluestrand$};
        \node at (2.2,0.03) {$\redstrand$};
        \node at (2.85,0.03) {$\greenstrand$};
        \node[scale=0.5] at (1.7,-0.3) {$\cdots$};
        \node[scale=0.5] at (2.65,-0.3) {$\cdots$};

        \draw[decorate, decoration={brace,mirror}]
            (-1.2,-0.5) -- (-0.1,-0.5)
            node[midway, below=3pt] {$g$};
        \draw[decorate, decoration={brace,mirror}]
            (0,-0.5) -- (2.2,-0.5)
            node[midway, below=3pt] {$h$};
        \draw[decorate, decoration={brace,mirror}]
            (2.3,-0.5) -- (3,-0.5)
            node[midway, below=3pt] {$k$};
    \end{tikzpicture}
    \end{adjustbox}
    \end{minipage}
\end{enumerate}
\end{minipage}%
\hfill
\begin{minipage}[t]{0.485\linewidth}
\vspace{0pt}
\begin{enumerate}[
    label=(R\arabic*$_{ee}$), start=2,
    leftmargin=*, topsep=0pt, partopsep=0pt,
    itemsep=0pt, parsep=0pt
]
    \item\label{R2ee}
    If $\begin{cases}
        |h|+|k|\geq 2n, \text{ and }\\
        |hk|+|g|<2n
    \end{cases}$
    \par\nobreak\vspace{1pt}\nointerlineskip
    \noindent
    \begin{minipage}[c][2.1cm][t]{\linewidth}
    \begin{adjustbox}{max width=\linewidth}
    \begin{tikzpicture}
        \node at (0,0.03) {$\redstrand$};
        \node at (0.75,0.03) {$\bluestrand$};
        \node at (2.25,0) {$\rbjellyspreadbr$};
        \node at (3.5,0.03) {$\redstrand$};
        \node at (4.15,0.03) {$\greenstrand$};
        \node[scale=0.5] at (0.5,-0.3) {$\cdots$};
        \node[scale=0.5] at (3.95,-0.3) {$\cdots$};

        \draw[decorate, decoration={brace,mirror}]
            (0,-0.5) -- (0.9,-0.5)
            node[midway, below=3pt] {$g$};
        \draw[decorate, decoration={brace,mirror}]
            (1,-0.5) -- (2.2,-0.5)
            node[midway, below=3pt] {$h$};
        \draw[decorate, decoration={brace,mirror}]
            (2.3,-0.5) -- (4.4,-0.5)
            node[midway, below=3pt] {$k$};
    \end{tikzpicture}
    \end{adjustbox}
    \end{minipage}
\end{enumerate}
\end{minipage}
\par
\penalty0
\vspace{0.6em}

\noindent
\begin{minipage}[t]{0.485\linewidth}
\vspace{0pt}
\begin{enumerate}[
    label=(L\arabic*$_{ee}$), start=3,
    leftmargin=*, topsep=0pt, partopsep=0pt,
    itemsep=0pt, parsep=0pt
]
    \item\label{L3ee}
    If $\begin{cases}
        |g|+|h|<2n, \text{ and }\\
        |gh|+|k|\geq 2n
    \end{cases}$
    \par\nobreak\vspace{1pt}\nointerlineskip
    \noindent
    \begin{minipage}[c][2.1cm][t]{\linewidth}
    \begin{adjustbox}{max width=\linewidth}
    \begin{tikzpicture}
        \node at (0,0) {$\rbjellyspreadbrbr$};
        \node at (1.3,0.03) {$\redstrand$};
        \node at (2,0.03) {$\greenstrand$};
        \node[scale=0.5] at (1.75,-0.3) {$\cdots$};

        \draw[decorate, decoration={brace,mirror}]
            (-1.2,-0.5) -- (-0.35,-0.5)
            node[midway, below=3pt] {$g$};
        \draw[decorate, decoration={brace,mirror}]
            (-0.25,-0.5) -- (0.3,-0.5)
            node[midway, below=3pt] {$h$};
        \draw[decorate, decoration={brace,mirror}]
            (0.4,-0.5) -- (2.2,-0.5)
            node[midway, below=3pt] {$k$};
    \end{tikzpicture}
    \end{adjustbox}
    \end{minipage}
\end{enumerate}
\end{minipage}%
\hfill
\begin{minipage}[t]{0.485\linewidth}
\vspace{0pt}
\begin{enumerate}[
    label=(R\arabic*$_{ee}$), start=3,
    leftmargin=*, topsep=0pt, partopsep=0pt,
    itemsep=0pt, parsep=0pt
]
    \item\label{R3ee}
    If $\begin{cases}
        |h|+|k|<2n, \text{ and }\\
        |hk|+|g|\geq 2n
    \end{cases}$
    \par\nobreak\vspace{1pt}\nointerlineskip
    \noindent
    \begin{minipage}[c][2.1cm][t]{\linewidth}
    \begin{adjustbox}{max width=\linewidth}
    \begin{tikzpicture}
        \node at (0,0) {$\rbjellyspreadbr$};
        \node at (1.25,0.03) {$\redstrand$};
        \node at (1.95,0.03) {$\bluestrand$};
        \node at (2.2,0.03) {$\redstrand$};
        \node at (2.85,0.03) {$\greenstrand$};
        \node[scale=0.5] at (1.7,-0.3) {$\cdots$};
        \node[scale=0.5] at (2.65,-0.3) {$\cdots$};

        \draw[decorate, decoration={brace,mirror}]
            (-1.2,-0.5) -- (-0.1,-0.5)
            node[midway, below=3pt] {$g$};
        \draw[decorate, decoration={brace,mirror}]
            (0,-0.5) -- (2.2,-0.5)
            node[midway, below=3pt] {$h$};
        \draw[decorate, decoration={brace,mirror}]
            (2.3,-0.5) -- (3,-0.5)
            node[midway, below=3pt] {$k$};
    \end{tikzpicture}
    \end{adjustbox}
    \end{minipage}
\end{enumerate}
\end{minipage}
\par
\penalty0
\vspace{0.6em}

\noindent
\begin{minipage}[t]{0.485\linewidth}
\vspace{0pt}
\begin{enumerate}[
    label=(L\arabic*$_{ee}$), start=4,
    leftmargin=*, topsep=0pt, partopsep=0pt,
    itemsep=0pt, parsep=0pt
]
    \item\label{L4ee}
    If $\begin{cases}
        |g|+|h|<2n, \text{ and }\\
        |gh|+|k|<2n
    \end{cases}$
    \par\nobreak\vspace{1pt}\nointerlineskip
    \noindent
    \begin{minipage}[c][2.1cm][t]{\linewidth}
    \begin{adjustbox}{max width=\linewidth}
    \begin{tikzpicture}
        \node at (0,0.03) {$\redstrand$};
        \node at (0.75,0.03) {$\bluestrand$};
        \node at (1.05,0.03) {$\redstrand$};
        \node at (1.7,0.03) {$\bluestrand$};
        \node at (2.05,0.03) {$\redstrand$};
        \node at (2.7,0.03) {$\greenstrand$};
        \node[scale=0.5] at (0.5,-0.3) {$\cdots$};
        \node[scale=0.5] at (1.5,-0.3) {$\cdots$};
        \node[scale=0.5] at (2.5,-0.3) {$\cdots$};

        \draw[decorate, decoration={brace,mirror}]
            (0,-0.5) -- (0.9,-0.5)
            node[midway, below=3pt] {$g$};
        \draw[decorate, decoration={brace,mirror}]
            (1,-0.5) -- (1.9,-0.5)
            node[midway, below=3pt] {$h$};
        \draw[decorate, decoration={brace,mirror}]
            (2,-0.5) -- (2.9,-0.5)
            node[midway, below=3pt] {$k$};
    \end{tikzpicture}
    \end{adjustbox}
    \end{minipage}
\end{enumerate}
\end{minipage}%
\hfill
\begin{minipage}[t]{0.485\linewidth}
\vspace{0pt}
\begin{enumerate}[
    label=(R\arabic*$_{ee}$), start=4,
    leftmargin=*, topsep=0pt, partopsep=0pt,
    itemsep=0pt, parsep=0pt
]
    \item\label{R4ee}
    If $\begin{cases}
        |h|+|k|<2n, \text{ and }\\
        |hk|+|g|<2n
    \end{cases}$
    \par\nobreak\vspace{1pt}\nointerlineskip
    \noindent
    \begin{minipage}[c][2.1cm][t]{\linewidth}
    \begin{adjustbox}{max width=\linewidth}
    \begin{tikzpicture}
        \node at (0,0.03) {$\redstrand$};
        \node at (0.75,0.03) {$\bluestrand$};
        \node at (1.05,0.03) {$\redstrand$};
        \node at (1.7,0.03) {$\bluestrand$};
        \node at (2.05,0.03) {$\redstrand$};
        \node at (2.7,0.03) {$\greenstrand$};
        \node[scale=0.5] at (0.5,-0.3) {$\cdots$};
        \node[scale=0.5] at (1.5,-0.3) {$\cdots$};
        \node[scale=0.5] at (2.5,-0.3) {$\cdots$};

        \draw[decorate, decoration={brace,mirror}]
            (0,-0.5) -- (0.9,-0.5)
            node[midway, below=3pt] {$g$};
        \draw[decorate, decoration={brace,mirror}]
            (1,-0.5) -- (1.9,-0.5)
            node[midway, below=3pt] {$h$};
        \draw[decorate, decoration={brace,mirror}]
            (2,-0.5) -- (2.9,-0.5)
            node[midway, below=3pt] {$k$};
    \end{tikzpicture}
    \end{adjustbox}
    \end{minipage}
\end{enumerate}
\end{minipage}
\par

\begin{lemma}
    When $|g|$ and $|h|$ are both even, then
    \begin{align*}
        \omega(g,h,k) &= \begin{cases}
        \xi^{\frac{|g|}{2}}, \text{ if } |h|+|k| \geq 2n,\\
        1, \text{ otherwise }.
    \end{cases}
    \end{align*}
\end{lemma}

\begin{proof}
Using the standard-form length conditions, one checks that only the
following six pairs of trees are compatible:

\begin{center}
\renewcommand{\arraystretch}{1.2}
\begin{tabular}{lll}
(1) \ref{L1ee} and \ref{R1ee}
&
(2) \ref{L2ee} and \ref{R2ee}
&
(3) \ref{L2ee} and \ref{R3ee}
\\
(4) \ref{L3ee} and \ref{R2ee}
&
(5) \ref{L3ee} and \ref{R3ee}
&
(6) \ref{L4ee} and \ref{R4ee}
\end{tabular}
\end{center}

\noindent \underline{Case (1)}. To transform the right tree \ref{R1ee} into the left tree \ref{L1ee}, we use relation \ref{rel: redbluejelly} to move over the last $2n-|g|$ strands of the outer jellyfish to the left of the inner jellyfish. Each pair of strands moved to the left results in a multiplication of a factor of $\xi^{-1}$. After moving over the $2n-|g|$ strands, \ref{R1ee} has been transformed into \ref{L1ee} multiplied by $\xi^{\frac{-(2n-|g|)}{2}}=\xi^{\frac{|g|}{2}}$. In this case, $\omega(g,h,k)=\xi^{\frac{|g|}{2}}$.

\noindent\underline{Case (2)}. 
To transform the right tree \ref{R2ee} into the left tree \ref{L2ee} we move the leftmost $|g|$ strands to the right of the jellyfish using \ref{rel: redbluejelly}. For each pair of strands, we multiply the diagram by $\xi$. After moving all of the $|g|$ strands, the diagram is $\xi^{\frac{|g|}{2}}$ times the left tree \ref{L2ee}. In this case, $\omega(g,h,k)=\xi^{\frac{|g|}{2}}$.

\noindent\underline{Case (3)}. 
The two trees are identical so $\omega(g,h,k)=1$.

\noindent\underline{Case (4)}. 
This is the same process as Case 2. Thus $\omega(g,h,k)=\xi^{\frac{|g|}{2}}$.

\noindent\underline{Case (5)}. 
These are the same diagrams so $\omega(g,h,k)=1$.

\noindent\underline{Case (6)}.
These are the same diagrams so $\omega(g,h,k)=1$.
\end{proof}

\subsubsection{The even--odd case} Suppose $|g|$ is even and $|h|$ is odd, and let $k\in \Dn$. 

\noindent
\begin{minipage}[t]{0.485\linewidth}
	\vspace{0pt}
	\begin{enumerate}[
		label=(L\arabic*$_{eo}$), start=1,
		leftmargin=*, topsep=0pt, partopsep=0pt,
		itemsep=0pt, parsep=0pt
		]
		\item\label{L1eo}
		If $\begin{cases}
			|g|+|h|\geq 2n, \text{ and }\\
			|gh|\geq |k|
		\end{cases}$
		\par\nobreak\vspace{1pt}\nointerlineskip
		\noindent
		\begin{adjustbox}{max width=\linewidth}
			\begin{tikzpicture}
				\node at (0,0) {$\rbjellyspreadbr$};
				\node at (1.25,0.03) {$\redstrand$};
				\node at (1.95,0.03) {$\purplestrand$};
				\node at (3.06,0.15) {$\biggreencap$};
				\node at (3,0.03) {$\redcap$};
				
				\node[scale=0.5] at (1.7,-0.3) {$\cdots$};
				\node[scale=0.5] at (2.6,-0.3) {$\cdots$};
				\node[scale=0.5] at (3.65,-0.3) {$\cdots$};
				
				\draw[decorate, decoration={brace,mirror}]
				(-1.2,-0.5) -- (-0.1,-0.5)
				node[midway, below=3pt] {$g$};
				\draw[decorate, decoration={brace,mirror}]
				(0,-0.5) -- (3,-0.5)
				node[midway, below=3pt] {$h$};
				\draw[decorate, decoration={brace,mirror}]
				(3.2,-0.5) -- (4,-0.5)
				node[midway, below=3pt] {$k$};
			\end{tikzpicture}
		\end{adjustbox}
	\end{enumerate}
\end{minipage}%
\hfill
\begin{minipage}[t]{0.485\linewidth}
	\vspace{0pt}
	\begin{enumerate}[
		label=(R\arabic*$_{eo}$), start=1,
		leftmargin=*, topsep=0pt, partopsep=0pt,
		itemsep=0pt, parsep=0pt
		]
		\item\label{R1eo}
		If $\begin{cases}
			|h|\geq |k|, \text{ and }\\
			|hk|+|g|\geq 2n
		\end{cases}$
		\par\nobreak\vspace{1pt}\nointerlineskip
		\noindent
		\begin{adjustbox}{max width=\linewidth}
			\begin{tikzpicture}
				\node at (0,0) {$\rbjellyspreadbr$};
				\node at (1.25,0.03) {$\redstrand$};
				\node at (1.95,0.03) {$\purplestrand$};
				\node at (3.06,0.15) {$\biggreencap$};
				\node at (3,0.03) {$\redcap$};
				
				\node[scale=0.5] at (1.7,-0.3) {$\cdots$};
				\node[scale=0.5] at (2.6,-0.3) {$\cdots$};
				\node[scale=0.5] at (3.65,-0.3) {$\cdots$};
				
				\draw[decorate, decoration={brace,mirror}]
				(-1.2,-0.5) -- (-0.1,-0.5)
				node[midway, below=3pt] {$g$};
				\draw[decorate, decoration={brace,mirror}]
				(0,-0.5) -- (3,-0.5)
				node[midway, below=3pt] {$h$};
				\draw[decorate, decoration={brace,mirror}]
				(3.2,-0.5) -- (4,-0.5)
				node[midway, below=3pt] {$k$};
			\end{tikzpicture}
		\end{adjustbox}
	\end{enumerate}
\end{minipage}
\par
\penalty0
\vspace{0.6em}

\noindent
\begin{minipage}[t]{0.485\linewidth}
	\vspace{0pt}
	\begin{enumerate}[
		label=(L\arabic*$_{eo}$), start=2,
		leftmargin=*, topsep=0pt, partopsep=0pt,
		itemsep=0pt, parsep=0pt
		]
		\item\label{L2eo}
		If $\begin{cases}
			|g|+|h|\geq 2n, \text{ and }\\
			|gh|<|k|
		\end{cases}$
		\par\nobreak\vspace{1pt}\nointerlineskip
		\noindent
		\begin{adjustbox}{max width=\linewidth}
			\begin{tikzpicture}
				\node at (0,0) {$\rbjellyspreadbr$};
				\node at (2.2,0.15) {$\bigredcap$};
				\node at (2.14,0.03) {$\redcap$};
				\node at (4.2,0.26) {$\brjellyswitchg$};
				
				\node[scale=0.5] at (1.73,-0.3) {$\cdots$};
				\node[scale=0.5] at (2.77,-0.3) {$\cdots$};
				
				\draw[decorate, decoration={brace,mirror}]
				(-1.2,-0.5) -- (-0.1,-0.5)
				node[midway, below=3pt] {$g$};
				\draw[decorate, decoration={brace,mirror}]
				(0,-0.5) -- (2.1,-0.5)
				node[midway, below=3pt] {$h$};
				\draw[decorate, decoration={brace,mirror}]
				(2.2,-0.5) -- (4,-0.5)
				node[midway, below=3pt] {$k$};
			\end{tikzpicture}
		\end{adjustbox}
	\end{enumerate}
\end{minipage}%
\hfill
\begin{minipage}[t]{0.485\linewidth}
	\vspace{0pt}
	\begin{enumerate}[
		label=(R\arabic*$_{eo}$), start=2,
		leftmargin=*, topsep=0pt, partopsep=0pt,
		itemsep=0pt, parsep=0pt
		]
		\item\label{R2eo}
		If $\begin{cases}
			|h|\geq |k|, \text{ and }\\
			|hk|+|g|<2n
		\end{cases}$
		\par\nobreak\vspace{1pt}\nointerlineskip
		\noindent
		\begin{adjustbox}{max width=\linewidth}
			\begin{tikzpicture}
				\node at (0,0.03) {$\redstrand$};
				\node at (0.7,0.03) {$\bluestrand$};
				\node at (0.95,0.03) {$\redstrand$};
				\node at (1.65,0.03) {$\purplestrand$};
				\node at (2.76,0.15) {$\biggreencap$};
				\node at (2.7,0.03) {$\redcap$};
				
				\node[scale=0.5] at (0.47,-0.3) {$\cdots$};
				\node[scale=0.5] at (1.42,-0.3) {$\cdots$};
				\node[scale=0.5] at (2.3,-0.3) {$\cdots$};
				\node[scale=0.5] at (3.35,-0.3) {$\cdots$};
				
				\draw[decorate, decoration={brace,mirror}]
				(0,-0.5) -- (0.9,-0.5)
				node[midway, below=3pt] {$g$};
				\draw[decorate, decoration={brace,mirror}]
				(1,-0.5) -- (2.7,-0.5)
				node[midway, below=3pt] {$h$};
				\draw[decorate, decoration={brace,mirror}]
				(2.9,-0.5) -- (3.75,-0.5)
				node[midway, below=3pt] {$k$};
			\end{tikzpicture}
		\end{adjustbox}
	\end{enumerate}
\end{minipage}
\par
\penalty0
\vspace{0.6em}

\begingroup
\renewcommand{\arraystretch}{0}
\noindent
\begin{tabular}{@{}l@{\hspace{0.03\linewidth}}l@{}}

\begin{minipage}[t]{0.485\linewidth}
\vspace{0pt}
\begin{enumerate}[
    label=(L\arabic*$_{eo}$), start=3,
    leftmargin=*, topsep=0pt, partopsep=0pt,
    itemsep=0pt, parsep=0pt
]
    \item\label{L3eo}
    If $\begin{cases}
        |g|+|h|<2n, \text{ and }\\
        |gh|\geq |k|
    \end{cases}$
\end{enumerate}
\end{minipage}
&
\begin{minipage}[t]{0.485\linewidth}
\vspace{0pt}
\begin{enumerate}[
    label=(R\arabic*$_{eo}$), start=3,
    leftmargin=*, topsep=0pt, partopsep=0pt,
    itemsep=0pt, parsep=0pt
]
    \item\label{R3eo}
    If $\begin{cases}
        |h|<|k|, \text{ and }\\
        |hk|+|g|\geq 2n
    \end{cases}$
\end{enumerate}
\end{minipage}
\\[-0.6cm]

\begin{adjustbox}{max width=0.485\linewidth}
\begin{tikzpicture}
    \node at (0.1,0.23) {$\mediumredstrand$};
    \node at (0.85,0.23) {$\mediumpurplestrand$};
    \node at (2.7,0.33) {$\biggggreencap$};
    \node at (2.7,0.2) {$\biggbluecap$};
    \node at (2.7,0.13) {$\bigredcap$};
    \node at (2.65,0.01) {$\redcap$};

    \node[scale=0.5] at (0.5,-0.3) {$\cdots$};
    \node[scale=0.5] at (1.4,-0.3) {$\cdots$};
    \node[scale=0.5] at (2.24,-0.3) {$\cdots$};
    \node[scale=0.5] at (3.27,-0.3) {$\cdots$};
    \node[scale=0.5] at (4.04,-0.3) {$\cdots$};

    \draw[decorate, decoration={brace,mirror}]
        (0,-0.5) -- (1.75,-0.5)
        node[midway, below=3pt] {$g$};
    \draw[decorate, decoration={brace,mirror}]
        (1.85,-0.5) -- (2.6,-0.5)
        node[midway, below=3pt] {$h$};
    \draw[decorate, decoration={brace,mirror}]
        (2.9,-0.5) -- (4.4,-0.5)
        node[midway, below=3pt] (EOLeftBase) {$k$};

    \pgfsetbaselinepointlater{%
        \pgfpointanchor{EOLeftBase}{base}%
    }
\end{tikzpicture}
\end{adjustbox}
&
\begin{adjustbox}{max width=0.485\linewidth}
\begin{tikzpicture}
    \node at (0,0) {$\twojellys$};
    \node at (-1.43,-0.53) {$\biggredcap$};
    \node at (-1.5,-0.7) {$\redcap$};

    \node[scale=0.5] at (-2.05,-1.04) {$\cdots$};
    \node[scale=0.5] at (-0.75,-1.04) {$\cdots$};

    \draw[decorate, decoration={brace,mirror}]
        (-3.4,-1.23) -- (-2.6,-1.23)
        node[midway, below=3pt] {$g$};
    \draw[decorate, decoration={brace,mirror}]
        (-2.55,-1.23) -- (-1.55,-1.23)
        node[midway, below=3pt] {$h$};
    \draw[decorate, decoration={brace,mirror}]
        (-1.3,-1.23) -- (1.5,-1.23)
        node[midway, below=3pt] (EORightBase) {$k$};

    \pgfsetbaselinepointlater{%
        \pgfpointanchor{EORightBase}{base}%
    }
\end{tikzpicture}
\end{adjustbox}

\end{tabular}
\par
\endgroup
\penalty0

\noindent
\begin{minipage}[t]{0.485\linewidth}
	\vspace{0pt}
	\begin{enumerate}[
		label=(L\arabic*$_{eo}$), start=4,
		leftmargin=*, topsep=0pt, partopsep=0pt,
		itemsep=0pt, parsep=0pt
		]
		\item\label{L4eo}
		If $\begin{cases}
			|g|+|h|<2n, \text{ and }\\
			|gh|<|k|
		\end{cases}$
		\par\nobreak\vspace{1pt}\nointerlineskip
		\noindent
		\begin{adjustbox}{max width=\linewidth}
			\begin{tikzpicture}
				\node at (0,0.33) {$\bigggredcap$};
				\node at (0,0.2) {$\biggbluecap$};
				\node at (0,0.13) {$\bigredcap$};
				\node at (-0.05,0.01) {$\redcap$};
				\node at (2.6,0.26) {$\brjellyswitchg$};
				
				\node[scale=0.5] at (-1.3,-0.3) {$\cdots$};
				\node[scale=0.5] at (-0.46,-0.3) {$\cdots$};
				\node[scale=0.5] at (0.57,-0.3) {$\cdots$};
				\node[scale=0.5] at (1.34,-0.3) {$\cdots$};
				
				\draw[decorate, decoration={brace,mirror}]
				(-1.7,-0.5) -- (-0.95,-0.5)
				node[midway, below=3pt] {$g$};
				\draw[decorate, decoration={brace,mirror}]
				(-0.85,-0.5) -- (-0.1,-0.5)
				node[midway, below=3pt] {$h$};
				\draw[decorate, decoration={brace,mirror}]
				(0.2,-0.5) -- (2.3,-0.5)
				node[midway, below=3pt] {$k$};
			\end{tikzpicture}
		\end{adjustbox}
	\end{enumerate}
\end{minipage}%
\hfill
\begin{minipage}[t]{0.485\linewidth}
	\vspace{0pt}
	\begin{enumerate}[
		label=(R\arabic*$_{eo}$), start=4,
		leftmargin=*, topsep=0pt, partopsep=0pt,
		itemsep=0pt, parsep=0pt
		]
		\item\label{R4eo}
		If $\begin{cases}
			|h|<|k|, \text{ and }\\
			|hk|+|g|<2n
		\end{cases}$
		\par\nobreak\vspace{1pt}\nointerlineskip
		\noindent
		\begin{adjustbox}{max width=\linewidth}
			\begin{tikzpicture}
				\node at (0,0.28) {$\mediumredstrand$};
				\node at (0.7,0.28) {$\mediumbluestrand$};
				\node at (1.77,0.15) {$\bigredcap$};
				\node at (1.71,0.03) {$\redcap$};
				\node at (3.7,0.26) {$\brjellyswitchg$};
				
				\node[scale=0.5] at (0.45,-0.3) {$\cdots$};
				\node[scale=0.5] at (1.3,-0.3) {$\cdots$};
				\node[scale=0.5] at (2.35,-0.3) {$\cdots$};
				
				\draw[decorate, decoration={brace,mirror}]
				(0,-0.5) -- (0.85,-0.5)
				node[midway, below=3pt] {$g$};
				\draw[decorate, decoration={brace,mirror}]
				(0.95,-0.5) -- (1.7,-0.5)
				node[midway, below=3pt] {$h$};
				\draw[decorate, decoration={brace,mirror}]
				(1.9,-0.5) -- (3.5,-0.5)
				node[midway, below=3pt] {$k$};
			\end{tikzpicture}
		\end{adjustbox}
	\end{enumerate}
\end{minipage}
\par

\begin{lemma}
    When $|g|$ is even and $|h|$ is odd, we have that
    \begin{align*}
        \omega(g,h,k) &= \begin{cases}
(\epsilon_r\epsilon_b\xi)^{-|g|/2}, &\text{ if }   |gh|, |h|<|k|,\\
\epsilon_r^{\lceil\frac{|ghk|}{2} \rceil}\epsilon_b^{\lfloor\frac{|ghk|}{2}\rfloor}, &\text{ if }  |g|+|h|\geq 2n> |hk|+|g| \text{ and } |gh|<|k|\leq |h|,\\
\epsilon_r^{\lceil\frac{|hk|}{2} \rceil}\epsilon_b^{\lfloor\frac{|hk|}{2}\rfloor}\xi^{-|g|/2}, &\text{ if }  |g|+|h|<2n \leq |hk|+|g|\text{ and } |gh|\geq |k| >|h|,\\
        1, &\text{ otherwise }.
    \end{cases}
    \end{align*}
\end{lemma}

\begin{proof}
Here are the 6 possible combinations of left trees with right trees to solve for: 
\begin{center}
\renewcommand{\arraystretch}{1.2}
\begin{tabular}{lll}
(1) \ref{L1eo} and \ref{R1eo}
&
(2) \ref{L2eo} and \ref{R2eo}
&
(3) \ref{L2eo} and \ref{R3eo}
\\
(4) \ref{L3eo} and \ref{R2eo}
&
(5) \ref{L3eo} and \ref{R3eo}
&
(6) \ref{L4eo} and \ref{R4eo}
\end{tabular}
\end{center}


\noindent \underline{Case (1)}. 
 The two trees are identical so  $\omega(g,h,k)=1$.

\noindent\underline{Case (2)}. 
Apply relations \ref{rel: rcupcap} and \ref{rel: bcupcap} alternately, for a total of $2n-|g|-|h|+|k|$  times. Then  \ref{L2eo} equals 
\begin{equation*}
    \case\lowperiod
\end{equation*}
 Now, starting with the outermost purple cup and innermost purple strand and ending with the innermost red cup and outermost red strand, we alternately apply relations \ref{rel: redstrandcup} and \ref{rel: bluestrandcup} to the diagram above a total of $|ghk|$ times to obtain that it equals
\begin{center}
\ \ \ \ \ \ \ \ \ \ \ \ \ \ \raisebox{-1cm}{  \begin{tikzpicture}
    \node at (0,0.03) {$\mujelly$};
    \node at (1.35, 0.2) {$\longredstrand$};
        \node at (1.57, 0.2) {$\longbluestrand$};
         \node at (2.2, 0.2) {$\longpurplestrand$};
    \node at (3.3, -0.5) {$\biggreencap$};
    \node at (3.2, -0.6) {$\redcap$};
         \node[scale=1] at (-1.6, 0) {$\lambda$};
     \node[scale=0.5] at (0.45, -0.3) {$\cdots$};
     \node[scale=0.5] at (2.8, -1) {$\cdots$};
     \node[scale=0.5] at (3.9, -1) {$\cdots$};
   \node[scale=0.5] at (2, -0.5) {$\cdots$};
\end{tikzpicture}}= $\lambda$
\raisebox{-0.5cm}{\begin{tikzpicture}
    \node at (1.35, 0.2) {$\redstrand$};
        \node at (1.60, 0.2) {$\bluestrand$};
         \node at (2.2, 0.2) {$\purplestrand$};
    \node at (3.3, 0.3) {$\biggreencap$};
    \node at (3.2, 0.20) {$\redcap$};
     \node[scale=0.5] at (2.8, -0.15) {$\cdots$};
     \node[scale=0.5] at (3.9, -0.15) {$\cdots$};
   \node[scale=0.5] at (2, -0.15) {$\cdots$};
\end{tikzpicture}}$\lowcomma$
\end{center}
where there are $|ghk|$ strands to the left and $|k|$ cups to the right, and 
$\lambda=
\epsilon_r^{\left\lceil |ghk|/2\right\rceil}
\epsilon_b^{\left\lfloor |ghk|/2\right\rfloor}$.
Since the last diagram is exactly \ref{R2eo}, it follows that in this case $\omega(g,h,k)= \epsilon_r^{\lceil\frac{|ghk|}{2} \rceil}\epsilon_b^{\lfloor\frac{|ghk|}{2}\rfloor}$.

\noindent\underline{Case (3)}. 
Starting with \ref{R3eo} and applying relations \ref{eq: blue up} and \ref{eq: red up} alternately to the $2n-|g|$ strands connecting the right jellyfish to the left jellyfish, we obtain that \ref{R3eo} equals 
 $(\epsilon_r\epsilon_b)^{n-|g|/2}\xi^{-|g|/2}$ times the diagram below:
\begin{equation*}
    \case\lowcomma
\end{equation*}
where there are $2n-|g|$ strands connecting the two jellyfish and $|h|\geq 2n-|g|$ caps on the bottom. We can then use relations \ref{rel: rcupcap} and \ref{rel: bcupcap} $2n-|g|$ times to get \ref{L2eo}. Thus we conclude in this case that $\omega(g,h,k)=(\epsilon_r\epsilon_b)^{n-|g|/2}\xi^{-|g|/2}=(\epsilon_r\epsilon_b\xi)^{-|g|/2}$, where in the last equality we are using that $\epsilon_r,\epsilon_b$ are 2-roots of unity, and that $\epsilon_r=\epsilon_b$ when $n$ is odd.

\noindent\underline{Case (4)}. 
The two trees are identical so $\omega(g,h,k)=1$.

\noindent\underline{Case (5)}. 
Following Case (3), we apply relations \ref{eq: blue up} and \ref{eq: red up}  to \ref{R3eo} followed by relations \ref{rel: rcupcap} and \ref{rel: bcupcap}, to get it equals $(\epsilon_r\epsilon_b)^{n-|g|/2}\xi^{-|g|/2}$ times the diagram
\begin{center}
 \raisebox{-1cm}{\begin{tikzpicture}
    \node at (0,0) {$\rbjellyspreadbr$};
    \node at (2.2, 0.15) {$\bigredcap$};
    \node at (2.14, 0.03) {$\redcap$};
    \node at (4.1, 0.26) {$\brjellyswitchg$};
     \node[scale=0.5] at (1.73, -0.3) {$\cdots$};
     \node[scale=0.5] at (2.77, -0.3) {$\cdots$};
   \node[scale=1] at (7, -0.3) {.};
  
\end{tikzpicture}}
\end{center}
That is, we get diagram \ref{L2eo}. So we can then proceed as in Case (2) applying relations \ref{rel: rcupcap} and \ref{rel: bcupcap} followed by   \ref{rel: redstrandcup} and \ref{rel: bluestrandcup}  to get that the diagram above equals $\epsilon_r^{\lceil\frac{|ghk|}{2} \rceil}\epsilon_b^{\lfloor\frac{|ghk|}{2}\rfloor}$ times \ref{R2eo}=\ref{L3eo}. Putting these two steps together, we conclude that in this case $\omega(g,h,k)=(\epsilon_r\epsilon_b)^{n-|g|/2}\xi^{-|g|/2}\epsilon_r^{\lceil\frac{|ghk|}{2} \rceil}\epsilon_b^{\lfloor\frac{|ghk|}{2}\rfloor}=\epsilon_r^{\lceil\frac{|hk|}{2} \rceil}\epsilon_b^{\lfloor\frac{|hk|}{2}\rfloor}\xi^{-|g|/2}.$

\noindent\underline{Case (6)}. 
 Starting with the outermost red cap to the left and the innermost red string to the right on \ref{L4eo} we use alternatingly relations 
\ref{rel: redstrandcap} and \ref{rel: bluestrandcap} a total of $|g|$ times to turn \ref{L4eo} into 
\begin{equation*}
    \lastcase \lowperiod
\end{equation*}
Now, using relations \ref{eq: blue up} and \ref{eq: red up} alternatingly, followed each time by \ref{rel: zigzag}, we obtain that the diagram above equals $(\epsilon_r\epsilon_b\xi)^{|g|/2}$ times \ref{R4eo}. Hence, in this case we conclude $\omega(g,h,k)=(\epsilon_r\epsilon_b\xi)^{-|g|/2}$.
\end{proof}

\subsubsection{The odd--even case}
Suppose $|g|$ is odd and $|h|$ is even, and let $k\in \Dn$. 

\noindent
\begin{minipage}[t]{0.58\linewidth}
	\vspace{0pt}
	\begin{enumerate}[
		label=(L\arabic*$_{oe}$), start=1,
		leftmargin=*, topsep=0pt, partopsep=0pt,
		itemsep=0pt, parsep=0pt
		]
		\item\label{L1oe}
		If $\begin{cases}
			|g|\geq |h|, \text{ and }\\
			|gh|\geq |k|
		\end{cases}$
		\par\nobreak\vspace{1pt}\nointerlineskip
\noindent\hspace*{-\leftmargin}%
\begin{adjustbox}{max width=\dimexpr\linewidth+\leftmargin\relax}
			\begin{tikzpicture}
				\node at (0.1,0.23) {$\mediumredstrand$};
				\node at (0.85,0.23) {$\mediumpurplestrand$};
				\node at (2.7,0.33) {$\biggggreencap$};
				\node at (2.7,0.2) {$\biggredcap$};
				\node at (2.7,0.13) {$\bigbluecap$};
				\node at (2.65,0.01) {$\redcap$};
				
				\node[scale=0.5] at (0.5,-0.3) {$\cdots$};
				\node[scale=0.5] at (1.4,-0.3) {$\cdots$};
				\node[scale=0.5] at (2.24,-0.3) {$\cdots$};
				\node[scale=0.5] at (3.27,-0.3) {$\cdots$};
				\node[scale=0.5] at (4.04,-0.3) {$\cdots$};
				
				\draw[decorate, decoration={brace,mirror}]
				(0,-0.5) -- (2.6,-0.5)
				node[midway, below=3pt] {$g$};
				\draw[decorate, decoration={brace,mirror}]
				(2.9,-0.5) -- (3.65,-0.5)
				node[midway, below=3pt] {$h$};
				\draw[decorate, decoration={brace,mirror}]
				(3.7,-0.5) -- (4.4,-0.5)
				node[midway, below=3pt] {$k$};
			\end{tikzpicture}
		\end{adjustbox}
	\end{enumerate}
\end{minipage}%
\hfill
\begin{minipage}[t]{0.39\linewidth}
	\vspace{0pt}
	\begin{enumerate}[
		label=(R\arabic*$_{oe}$), start=1,
		leftmargin=*, topsep=0pt, partopsep=0pt,
		itemsep=0pt, parsep=0pt
		]
		\item\label{R1oe}
		If $\begin{cases}
			|h|+|k|\geq 2n, \text{ and }\\
			|g|\geq |hk|
		\end{cases}$
		\par\nobreak\vspace{1pt}\nointerlineskip
\noindent\hspace*{-\leftmargin}%
\begin{adjustbox}{max width=\dimexpr\linewidth+\leftmargin\relax}
			\begin{tikzpicture}
				\node at (0.1,0.23) {$\mediumredstrand$};
				\node at (0.85,0.23) {$\mediumpurplestrand$};
				\node at (3,0.33) {$\widdegcap$};
				\node at (3,0.2) {$\widercap$};
				\node at (3,0) {$\rbjellyspreadbr$};
				
				\node[scale=0.5] at (0.5,-0.3) {$\cdots$};
				\node[scale=0.5] at (1.42,-0.3) {$\cdots$};
				\node[scale=0.5] at (4.66,-0.3) {$\cdots$};
				
				\draw[decorate, decoration={brace,mirror}]
				(0,-0.5) -- (1.75,-0.5)
				node[midway, below=3pt] {$g$};
				\draw[decorate, decoration={brace,mirror}]
				(1.85,-0.5) -- (2.9,-0.5)
				node[midway, below=3pt] {$h$};
				\draw[decorate, decoration={brace,mirror}]
				(3,-0.5) -- (5,-0.5)
				node[midway, below=3pt] {$k$};
			\end{tikzpicture}
		\end{adjustbox}
	\end{enumerate}
\end{minipage}
\par
\penalty0
\vspace{0.6em}

\noindent
\begin{minipage}[t]{0.58\linewidth}
	\vspace{0pt}
	\begin{enumerate}[
		label=(L\arabic*$_{oe}$), start=2,
		leftmargin=*, topsep=0pt, partopsep=0pt,
		itemsep=0pt, parsep=0pt
		]
		\item\label{L2oe}
		If $\begin{cases}
			|g|\geq |h|, \text{ and }\\
			|gh|<|k|
		\end{cases}$
		\par\nobreak\vspace{1pt}\nointerlineskip
\noindent\hspace*{-\leftmargin}%
\begin{adjustbox}{max width=\dimexpr\linewidth+\leftmargin\relax}
			\begin{tikzpicture}
				\node at (0,0.33) {$\bigggredcap$};
				\node at (0,0.2) {$\biggredcap$};
				\node at (0,0.13) {$\bigbluecap$};
				\node at (-0.05,0.01) {$\redcap$};
				\node at (2.7,0.26) {$\brjellyswitchg$};
				
				\node[scale=0.5] at (-1.3,-0.3) {$\cdots$};
				\node[scale=0.5] at (-0.46,-0.3) {$\cdots$};
				\node[scale=0.5] at (0.57,-0.3) {$\cdots$};
				\node[scale=0.5] at (1.34,-0.3) {$\cdots$};
				
				\draw[decorate, decoration={brace,mirror}]
				(-1.7,-0.5) -- (-0.1,-0.5)
				node[midway, below=3pt] {$g$};
				\draw[decorate, decoration={brace,mirror}]
				(0.2,-0.5) -- (0.9,-0.5)
				node[midway, below=3pt] {$h$};
				\draw[decorate, decoration={brace,mirror}]
				(1,-0.5) -- (2.5,-0.5)
				node[midway, below=3pt] {$k$};
			\end{tikzpicture}
		\end{adjustbox}
	\end{enumerate}
\end{minipage}%
\hfill
\begin{minipage}[t]{0.39\linewidth}
	\vspace{0pt}
	\begin{enumerate}[
		label=(R\arabic*$_{oe}$), start=2,
		leftmargin=*, topsep=0pt, partopsep=0pt,
		itemsep=0pt, parsep=0pt
		]
		\item\label{R2oe}
		If $\begin{cases}
			|h|+|k|\geq 2n, \text{ and }\\
			|g|<|hk|
		\end{cases}$
		\par\nobreak\vspace{1pt}\nointerlineskip
\noindent\hspace*{-\leftmargin}%
\begin{adjustbox}{max width=\dimexpr\linewidth+\leftmargin\relax}
			\begin{tikzpicture}
				\node at (0,0.33) {$\widdercap$};
				\node at (0,0.2) {$\widercap$};
				\node at (0,0) {$\rbjellyspreadbr$};
				\node at (3,0.25) {$\brjellyswitchg$};
				
				\node[scale=0.5] at (-1.58,-0.3) {$\cdots$};
				\node[scale=0.5] at (1.66,-0.3) {$\cdots$};
				
				\draw[decorate, decoration={brace,mirror}]
				(-1.9,-0.5) -- (-1.25,-0.5)
				node[midway, below=3pt] {$g$};
				\draw[decorate, decoration={brace,mirror}]
				(-1.1,-0.5) -- (-0.1,-0.5)
				node[midway, below=3pt] {$h$};
				\draw[decorate, decoration={brace,mirror}]
				(0,-0.5) -- (2.8,-0.5)
				node[midway, below=3pt] {$k$};
			\end{tikzpicture}
		\end{adjustbox}
	\end{enumerate}
\end{minipage}
\par
\penalty0
\vspace{0.6em}

\noindent
\begin{minipage}[t]{0.58\linewidth}
	\vspace{0pt}
	\begin{enumerate}[
		label=(L\arabic*$_{oe}$), start=3,
		leftmargin=*, topsep=0pt, partopsep=0pt,
		itemsep=0pt, parsep=0pt
		]
		\item\label{L3oe}
		If $\begin{cases}
			|g|<|h|, \text{ and }\\
			|gh|\geq |k|
		\end{cases}$
		\par\nobreak\vspace{1pt}\nointerlineskip
\noindent\hspace*{-\leftmargin}%
\begin{adjustbox}{max width=\dimexpr\linewidth+\leftmargin\relax}
			\begin{tikzpicture}
				\node at (0,0.15) {$\bigredcap$};
				\node at (-0.06,0.03) {$\redcap$};
				\node at (3.35,0.6) {$\brjellyflip$};
				
				\node[scale=0.5] at (-0.47,-0.3) {$\cdots$};
				\node[scale=0.5] at (0.58,-0.3) {$\cdots$};
				
				\draw[decorate, decoration={brace,mirror}]
				(-0.9,-0.5) -- (-0.1,-0.5)
				node[midway, below=3pt] {$g$};
				\draw[decorate, decoration={brace,mirror}]
				(0.2,-0.5) -- (2.7,-0.5)
				node[midway, below=3pt] {$h$};
				\draw[decorate, decoration={brace,mirror}]
				(5,-0.5) -- (5.8,-0.5)
				node[midway, below=3pt] {$k$};
			\end{tikzpicture}
		\end{adjustbox}
	\end{enumerate}
\end{minipage}%
\hfill
\begin{minipage}[t]{0.39\linewidth}
	\vspace{0pt}
	\begin{enumerate}[
		label=(R\arabic*$_{oe}$), start=3,
		leftmargin=*, topsep=0pt, partopsep=0pt,
		itemsep=0pt, parsep=0pt
		]
		\item\label{R3oe}
		If $\begin{cases}
			|h|+|k|<2n, \text{ and }\\
			|g|\geq |hk|
		\end{cases}$
		\par\nobreak\vspace{1pt}\nointerlineskip
\noindent\hspace*{-\leftmargin}%
\begin{adjustbox}{max width=\dimexpr\linewidth+\leftmargin\relax}
			\begin{tikzpicture}
				\node at (0.1,0.23) {$\mediumredstrand$};
				\node at (0.85,0.23) {$\mediumpurplestrand$};
				\node at (2.7,0.33) {$\biggggreencap$};
				\node at (2.7,0.2) {$\biggredcap$};
				\node at (2.7,0.13) {$\bigbluecap$};
				\node at (2.65,0.01) {$\redcap$};
				
				\node[scale=0.5] at (0.5,-0.3) {$\cdots$};
				\node[scale=0.5] at (1.4,-0.3) {$\cdots$};
				\node[scale=0.5] at (2.24,-0.3) {$\cdots$};
				\node[scale=0.5] at (3.27,-0.3) {$\cdots$};
				\node[scale=0.5] at (4.04,-0.3) {$\cdots$};
				
				\draw[decorate, decoration={brace,mirror}]
				(0,-0.5) -- (2.6,-0.5)
				node[midway, below=3pt] {$g$};
				\draw[decorate, decoration={brace,mirror}]
				(2.9,-0.5) -- (3.65,-0.5)
				node[midway, below=3pt] {$h$};
				\draw[decorate, decoration={brace,mirror}]
				(3.7,-0.5) -- (4.4,-0.5)
				node[midway, below=3pt] {$k$};
			\end{tikzpicture}
		\end{adjustbox}
	\end{enumerate}
\end{minipage}
\par
\penalty0
\vspace{0.6em}

\noindent
\begin{minipage}[t]{0.58\linewidth}
	\vspace{0pt}
	\begin{enumerate}[
		label=(L\arabic*$_{oe}$), start=4,
		leftmargin=*, topsep=0pt, partopsep=0pt,
		itemsep=0pt, parsep=0pt
		]
		\item\label{L4oe}
		If $\begin{cases}
			|g|<|h|, \text{ and }\\
			|gh|<|k|
		\end{cases}$
	\par\nobreak\vspace{1pt}\nointerlineskip
\noindent\hspace*{-\leftmargin}%
\begin{adjustbox}{max width=\dimexpr\linewidth+\leftmargin\relax}
			\begin{tikzpicture}
				\node at (0,0.15) {$\bigredcap$};
				\node at (-0.06,0.03) {$\redcap$};
				\node at (3.15,0.72) {$\brjellyfliprr$};
				\node at (7.1,0.54) {$\brjellyswitchgx$};
				
				\node[scale=0.5] at (-0.47,-0.3) {$\cdots$};
				\node[scale=0.5] at (0.58,-0.3) {$\cdots$};
				
				\draw[decorate, decoration={brace,mirror}]
				(-0.9,-0.5) -- (-0.1,-0.5)
				node[midway, below=3pt] {$g$};
				\draw[decorate, decoration={brace,mirror}]
				(0.2,-0.5) -- (2.7,-0.5)
				node[midway, below=3pt] {$h$};
				\draw[decorate, decoration={brace,mirror}]
				(4.6,-0.5) -- (7.2,-0.5)
				node[midway, below=3pt] {$k$};
			\end{tikzpicture}
		\end{adjustbox}
	\end{enumerate}
\end{minipage}%
\hfill
\begin{minipage}[t]{0.39\linewidth}
	
	\vspace{0pt}
	\begin{enumerate}[
		label=(R\arabic*$_{oe}$), start=4,
		leftmargin=*, topsep=0pt, partopsep=0pt,
		itemsep=0pt, parsep=0pt
		]
		\item\label{R4oe}
		If $\begin{cases}
			|h|+|k|<2n, \text{ and }\\
			|g|<|hk|
		\end{cases}$
	\par\nobreak\vspace{1pt}\nointerlineskip
\noindent\hspace*{-\leftmargin}%
\begin{adjustbox}{max width=\dimexpr\linewidth+\leftmargin\relax}
			\begin{tikzpicture}
				\node at (0,0.33) {$\bigggredcap$};
				\node at (0,0.2) {$\biggredcap$};
				\node at (0,0.13) {$\bigbluecap$};
				\node at (-0.05,0.01) {$\redcap$};
				\node at (2.8,0.26) {$\brjellyswitchg$};
				
				\node[scale=0.5] at (-1.3,-0.3) {$\cdots$};
				\node[scale=0.5] at (-0.46,-0.3) {$\cdots$};
				\node[scale=0.5] at (0.57,-0.3) {$\cdots$};
				\node[scale=0.5] at (1.34,-0.3) {$\cdots$};
				
				\draw[decorate, decoration={brace,mirror}]
				(-1.7,-0.5) -- (-0.1,-0.5)
				node[midway, below=3pt] {$g$};
				\draw[decorate, decoration={brace,mirror}]
				(0.2,-0.5) -- (0.9,-0.5)
				node[midway, below=3pt] {$h$};
				\draw[decorate, decoration={brace,mirror}]
				(1,-0.5) -- (2.6,-0.5)
				node[midway, below=3pt] {$k$};
			\end{tikzpicture}
		\end{adjustbox}
	\end{enumerate}
\end{minipage}
\par

\begin{lemma}
    When $|g|$ is odd and $|h|$ is even, then for all $k$
    \begin{align*}
        \omega(g,h,k) = \begin{cases}
            \epsilon_r^{\lceil |ghk|/2\rceil}\epsilon_b^{\lfloor |ghk|/2\rfloor} \xi^{-\lfloor|g|/2\rfloor}, \text{ if } |g| \geq |h|, |gh|<|k|, |h|+|k|\geq 2n, \text{ and }|g|\geq |hk|,\\
            \epsilon_r^{\lceil |gh|/2\rceil} \epsilon_b^{\lfloor |gh|/2\rfloor} \xi^{-\lfloor |g|/2\rfloor}, \text{ if } |g|<|h| \text{ and } |h|+|k|\geq 2n,\\
            \epsilon_r^{\lceil |k|/2\rceil} \epsilon_b^{\lfloor |k|/2 \rfloor}, \text{ if } |g|<|h|, |gh|\geq |k|, |h|+|k|<2n, \text{ and } |g|<|hk|,\\
            1 \text{ otherwise}.
        \end{cases}
    \end{align*}
\end{lemma}

\begin{proof}
Here are the 6 possible combinations of left trees and right trees.
\begin{center}
\renewcommand{\arraystretch}{1.2}
\begin{tabular}{lll}
(1) \ref{L1oe} and \ref{R3oe}
&
(2) \ref{L2oe} and \ref{R1oe}
&
(3) \ref{L2oe} and \ref{R4oe}
\\
(4) \ref{L3oe} and \ref{R1oe}
&
(5) \ref{L3oe} and \ref{R4oe}
&
(6) \ref{L4oe} and \ref{R2oe}
\end{tabular}
\end{center}

\noindent\underline{Case (1)}. 
The trees are identical so $\omega(g,h,k) =1$. 

\noindent\underline{Case (2)}. 
We start with the left tree. Notice that in this case: $|g|\geq |ghk|=|g|-|h|-|k|+2n$. There are $|g|$ caps to the left of the jellyfish and $|ghk|$ dots at the top of the diagram. For each of the $|ghk|$ dots, we use relations \ref{rel: redstrandcap} and \ref{rel: bluestrandcap}. The resulting diagram is $\epsilon_r^{\lceil |ghk|/2 \rceil} \epsilon_b^{\lfloor |ghk|/2\rfloor}$ multiplied by
\begin{center}
   \raisebox{-1cm}{\begin{tikzpicture}
    \node at (0.1, 0.83) {$\longredstrand$};
    \node at (0.75, 0.83) {$\longpurplestrand$};
    \node at (2.7, 0.33) {$\biggggreencap$};
    \node at (2.7,0.2) {$\biggredcap$};
    \node at (2.7, 0.13) {$\bigbluecap$};
    \node at (2.65, 0.01) {$\redcap$};
    \node at (6.5, 0.81) {$\brjellyflippedover$};

    \node[scale=0.5] at (0.5, -0.3) {$\cdots$};
    \node[scale=0.5] at (1.4, -0.3) {$\cdots$};
    \node[scale=0.5] at (2.24, -0.3) {$\cdots$};
    \node[scale=0.5] at (3.27, -0.3) {$\cdots$};
   \node[scale=0.5] at (4.04, -0.3) {$\cdots$};

    \draw[decorate, decoration={brace,mirror}] (0,-0.5) -- (2.6,-0.5) node[midway, below=3pt] {$g$};
    \draw[decorate, decoration={brace,mirror}] (2.9,-0.5) -- (3.65,-0.5) node[midway, below=3pt] {$h$};
    \draw[decorate, decoration={brace,mirror}] (3.7, -0.5) -- (7, -0.5) node[midway, below=3pt] {$k$};    
\end{tikzpicture}$\lowperiod$}
\end{center}

Next, we move over the rightmost $|g|-1$ strands directly to the left of the jellyfish using relation \ref{rel: blueredjelly}. This results in a factor of $\xi^{\frac{|g|-1}{2}}=\xi^{\lfloor |g|/2\rfloor}$ multiplied by the previous factor of $\epsilon_r^{\lceil |ghk|/2 \rceil} \epsilon_b^{\lfloor |ghk|/2\rfloor}$ times
\begin{center}
    \raisebox{-1cm}{\begin{tikzpicture}
    \node at (0.1, 0.3) {$\mediumxredstrand$};
    \node at (0.85, 0.3) {$\mediumxpurplestrand$};
    \node at (3.7, 0.6) {$\widdddegcap$};
    \node at (3.7, 0.42) {$\widddercap$};
    \node at (3.7, 0.33) {$\widdebcap$};
    \node at (2.3,0.02) {$\redcap$};
    \node at (4.05, -0.01) {$\brjellyspreadbr$};

    \node[scale=0.5] at (0.5, -0.3) {$\cdots$};
    \node[scale=0.5] at (1.38, -0.3) {$\cdots$};
   \node[scale=0.5] at (6.12, -0.3) {$\cdots$};

    \draw[decorate, decoration={brace,mirror}] (0,-0.5) -- (2.25,-0.5) node[midway, below=3pt] {$g$};
    \draw[decorate, decoration={brace,mirror}] (2.5,-0.5) -- (3.7,-0.5) node[midway, below=3pt] {$h$};
    \draw[decorate, decoration={brace,mirror}] (3.8, -0.5) -- (6.5, -0.5) node[midway, below=3pt] {$k$};    
\end{tikzpicture}}.
\end{center}
Rewriting the blue-red jellyfish as a red-blue jellyfish using Definition \ref{def: blue jelly}, then using \ref{rel: rcupcap} with the red cap to the left of the jellyfish, we get $\epsilon_r^{\lceil |ghk|/2 \rceil} \epsilon_b^{\lfloor |ghk|/2\rfloor}\xi^{\lfloor |g|/2\rfloor}$ multiplied by the right tree. Therefore, $\omega(g,h,k)=\epsilon_r^{\lceil |ghk|/2 \rceil} \epsilon_b^{\lfloor |ghk|/2\rfloor}\xi^{-\lfloor |g|/2\rfloor}$.

\noindent\underline{Case (3)}. 
The trees are identical so $\omega(g,h,k) =1$. 

\noindent\underline{Case (4)}. 
We start with the left tree. First use \ref{rel: zigzag} on the leftmost $|k|$ strands to straighten them. This will result in the below diagram multiplied by $\epsilon_r^{\lceil |k|/2\rceil}\epsilon_b^{\lfloor |k|/2\rfloor}$:
\begin{center}
    \raisebox{-1cm}{\begin{tikzpicture}
    \node at (0, 0.15) {$\bigredcap$};
    \node at (-0.06, 0.03) {$\redcap$};
    \node at (1.9, 0.26) {$\brjellyswitchg$};
    
     \node[scale=0.5] at (-0.47, -0.3) {$\cdots$};
     \node[scale=0.5] at (0.58, -0.3) {$\cdots$};   
\end{tikzpicture}$\lowperiod$}
\end{center}
Since $|g|\geq |ghk|=|g|-|h|-|k|+2n$, we can follow a nearly identical argument to Case (2). As a result, we get that the left tree is  $\epsilon_r^{\lceil |k|/2\rceil}\epsilon_b^{\lfloor |k|/2\rfloor}\epsilon_r^{\lceil |ghk|/2\rceil}\epsilon_b^{\lfloor |ghk|/2\rfloor}\xi^{\lfloor|g|/2\rfloor}=\epsilon_r^{\lceil |gh|/2\rceil}\epsilon_b^{\lfloor |gh|/2 \rfloor}\xi^{\lfloor |g|/2\rfloor}$ times the right tree. Therefore, $\omega(g,h,k)=\epsilon_r^{\lceil |gh|/2\rceil}\epsilon_b^{\lfloor |gh|/2 \rfloor}\xi^{-\lfloor |g|/2\rfloor}$.

\noindent\underline{Case (5)}. 
We start with the left tree. Using \ref{rel: zigzag} to straighten the leftmost $|k|$ strands of the jellyfish. This results in $\epsilon_r^{\lceil |k|/2\rceil}\epsilon_b^{\lfloor |k|/2 \rfloor}$ times the right tree. Therefore $\omega(g,h,k) = \epsilon_r^{\lceil |k|/2\rceil}\epsilon_b^{\lfloor |k|/2 \rfloor}$. 

\noindent\underline{Case (6)}.
 We start with the left tree. For the leftmost jellyfish, use relation \ref{rel: zigzag} to straighten the last $|gh|$ strands. This results in $\epsilon_r^{\lceil |gh|/2 \rceil} \epsilon_b^{\lfloor |gh|/2\rfloor}$ multiplied by 
\begin{center}
    \raisebox{-1cm}{\begin{tikzpicture}
    \node at (0, 0.15) {$\bigredcap$};
    \node at (-0.06, 0.03) {$\redcap$};
    \node at (1.5, 0) {$\brjellydown$};
    \node at (3.2, 0.26) {$\brjellyswitchg$};
     \node[scale=0.5] at (-0.47, -0.3) {$\cdots$};
     \node[scale=0.5] at (0.58, -0.3) {$\cdots$};   
\end{tikzpicture}$\lowperiod$}
\end{center}
Then use relation \ref{rel: blueredjelly} to move over the $|g|-1$ strands on the left side of the leftmost jellyfish. We then obtain $\epsilon_r^{\lceil |gh|/2 \rceil} \epsilon_b^{\lfloor |gh|/2\rfloor}\xi^{\lfloor |g|/2\rfloor}$ multiplied by

\begin{center}
    \raisebox{-1cm}{\begin{tikzpicture}
    \node at (3.7, 0.6) {$\widdddercap$};
    \node at (3.7, 0.42) {$\widddercap$};
    \node at (3.7, 0.33) {$\widdebcap$};
    \node at (2.3,0.02) {$\redcap$};
    \node at (4.05, -0.01) {$\brjellyspreadbr$};
    \node at (7.6, 0.26) {$\brjellyswitchg$};

    \node[scale=0.5] at (1.38, -0.3) {$\cdots$};
   \node[scale=0.5] at (6.12, -0.3) {$\cdots$};
      
\end{tikzpicture}$\lowperiod$}
\end{center}
Rewriting the blue-red jellyfish using definition \ref{def: blue jelly}, then using \ref{rel: rcupcap} with the red cap to the left of the jellyfish, we get $\epsilon_r^{\lceil |gh|/2 \rceil} \epsilon_b^{\lfloor |gh|/2\rfloor}\xi^{\lfloor |g|/2\rfloor}$ times the right tree. Therefore, $\omega(g,h,k)=\epsilon_r^{\lceil |gh|/2 \rceil} \epsilon_b^{\lfloor |gh|/2\rfloor}\xi^{-\lfloor |g|/2\rfloor}$.
\end{proof}

\subsubsection{The odd--odd case}    Suppose $|g|$ and $|h|$ are both odd, and let $k\in \Dn$.  

\noindent
\begin{minipage}[t]{0.58\linewidth}
\vspace{0pt}
\begin{enumerate}[
    label=(L\arabic*$_{oo}$), start=1,
    leftmargin=*, topsep=0pt, partopsep=0pt,
    itemsep=0pt, parsep=0pt
]
    \item\label{L1oo}
    If $\begin{cases}
        |g|\geq |h|, \text{ and }\\
        |gh|+|k|\geq 2n
    \end{cases}$
    \par\nobreak\vspace{1pt}\nointerlineskip
    \noindent\hspace*{-\leftmargin}%
    \begin{adjustbox}{max width=\dimexpr\linewidth+\leftmargin\relax}
    \begin{tikzpicture}
        \node at (0,0.16) {$\rbbbigjellyspreadbr$};
        \node at (0,0.13) {$\bigredcap$};
        \node at (-0.05,0.01) {$\redcap$};
        \node at (2,0) {$\redstrand$};
        \node at (2.4,0) {$\greenstrand$};

        \node[scale=0.5] at (-0.46,-0.3) {$\cdots$};
        \node[scale=0.5] at (0.57,-0.3) {$\cdots$};
        \node[scale=0.5] at (2.31,-0.3) {$\cdots$};

        \draw[decorate, decoration={brace,mirror}]
            (-2,-0.5) -- (-0.1,-0.5)
            node[midway, below=3pt] {$g$};
        \draw[decorate, decoration={brace,mirror}]
            (0.2,-0.5) -- (0.91,-0.5)
            node[midway, below=3pt] {$h$};
        \draw[decorate, decoration={brace,mirror}]
            (1.3,-0.5) -- (2.6,-0.5)
            node[midway, below=3pt] {$k$};
    \end{tikzpicture}
    \end{adjustbox}
\end{enumerate}
\end{minipage}%
\hfill
\begin{minipage}[t]{0.39\linewidth}
\vspace{0pt}
\begin{enumerate}[
    label=(R\arabic*$_{oo}$), start=1,
    leftmargin=*, topsep=0pt, partopsep=0pt,
    itemsep=0pt, parsep=0pt
]
    \item\label{R1oo}
    If $\begin{cases}
        |h|\geq |k|, \text{ and }\\
        |g|\geq |hk|
    \end{cases}$
    \par\nobreak\vspace{1pt}\nointerlineskip
    \noindent\hspace*{-\leftmargin}%
    \begin{adjustbox}{max width=\dimexpr\linewidth+\leftmargin\relax}
    \begin{tikzpicture}
        \node at (0,0.03) {$\redstrand$};
        \node at (0.75,0.03) {$\greenstrand$};
        \node at (1.85,0.03) {$\redcap$};
        \node at (1.9,0.15) {$\bigpurplecap$};
        \node at (1.85,0.03) {$\redcap$};
        \node at (3.7,0.03) {$\redcap$};
        \node at (3.75,0.15) {$\biggreencap$};

        \node[scale=0.5] at (0.5,-0.3) {$\cdots$};
        \node[scale=0.5] at (1.45,-0.3) {$\cdots$};
        \node[scale=0.5] at (2.47,-0.3) {$\cdots$};
        \node[scale=0.5] at (3.26,-0.3) {$\cdots$};
        \node[scale=0.5] at (4.35,-0.3) {$\cdots$};

        \draw[decorate, decoration={brace,mirror}]
            (0,-0.5) -- (1.8,-0.5)
            node[midway, below=3pt] {$g$};
        \draw[decorate, decoration={brace,mirror}]
            (2.1,-0.5) -- (3.6,-0.5)
            node[midway, below=3pt] {$h$};
        \draw[decorate, decoration={brace,mirror}]
            (3.9,-0.5) -- (4.8,-0.5)
            node[midway, below=3pt] {$k$};
    \end{tikzpicture}
    \end{adjustbox}
\end{enumerate}
\end{minipage}
\par
\penalty0
\vspace{0.6em}

\noindent
\begin{minipage}[t]{0.58\linewidth}
\vspace{0pt}
\begin{enumerate}[
    label=(L\arabic*$_{oo}$), start=2,
    leftmargin=*, topsep=0pt, partopsep=0pt,
    itemsep=0pt, parsep=0pt
]
    \item\label{L2oo}
    If $\begin{cases}
        |g|\geq |h|, \text{ and }\\
        |gh|+|k|<2n
    \end{cases}$
    \par\nobreak\vspace{1pt}\nointerlineskip
    \noindent\hspace*{-\leftmargin}%
    \begin{adjustbox}{max width=\dimexpr\linewidth+\leftmargin\relax}
    \begin{tikzpicture}
        \node at (0,0.03) {$\redstrand$};
        \node at (0.75,0.03) {$\bluestrand$};
        \node at (1.85,0.03) {$\redcap$};
        \node at (1.9,0.15) {$\bigredcap$};
        \node at (2.9,0.04) {$\redstrand$};
        \node at (3.4,0.04) {$\greenstrand$};

        \node[scale=0.5] at (0.5,-0.3) {$\cdots$};
        \node[scale=0.5] at (1.45,-0.3) {$\cdots$};
        \node[scale=0.5] at (2.47,-0.3) {$\cdots$};
        \node[scale=0.5] at (3.26,-0.3) {$\cdots$};

        \draw[decorate, decoration={brace,mirror}]
            (0,-0.5) -- (1.8,-0.5)
            node[midway, below=3pt] {$g$};
        \draw[decorate, decoration={brace,mirror}]
            (2.1,-0.5) -- (2.85,-0.5)
            node[midway, below=3pt] {$h$};
        \draw[decorate, decoration={brace,mirror}]
            (2.95,-0.5) -- (3.6,-0.5)
            node[midway, below=3pt] {$k$};
    \end{tikzpicture}
    \end{adjustbox}
\end{enumerate}
\end{minipage}%
\hfill
\begin{minipage}[t]{0.39\linewidth}
\vspace{0pt}
\begin{enumerate}[
    label=(R\arabic*$_{oo}$), start=2,
    leftmargin=*, topsep=0pt, partopsep=0pt,
    itemsep=0pt, parsep=0pt
]
    \item\label{R2oo}
    If $\begin{cases}
        |h|\geq |k|, \text{ and }\\
        |g|<|hk|
    \end{cases}$
    \par\nobreak\vspace{1pt}\nointerlineskip
    \noindent\hspace*{-\leftmargin}%
    \begin{adjustbox}{max width=\dimexpr\linewidth+\leftmargin\relax}
    \begin{tikzpicture}
        \node at (0,0.15) {$\bigredcap$};
        \node at (-0.06,0.03) {$\redcap$};
        \node at (2.7,0.3) {$\brjellyswitchp$};
        \node at (5,0.1) {$\biggreencap$};
        \node at (4.9,-0.01) {$\redcap$};

        \node[scale=0.5] at (-0.47,-0.3) {$\cdots$};
        \node[scale=0.5] at (0.58,-0.3) {$\cdots$};
        \node[scale=0.5] at (4.51,-0.3) {$\cdots$};
        \node[scale=0.5] at (5.55,-0.3) {$\cdots$};

        \draw[decorate, decoration={brace,mirror}]
            (-0.9,-0.5) -- (-0.1,-0.5)
            node[midway, below=3pt] {$g$};
        \draw[decorate, decoration={brace,mirror}]
            (0.2,-0.5) -- (4.8,-0.5)
            node[midway, below=3pt] {$h$};
        \draw[decorate, decoration={brace,mirror}]
            (5.2,-0.5) -- (6,-0.5)
            node[midway, below=3pt] {$k$};
    \end{tikzpicture}
    \end{adjustbox}
\end{enumerate}
\end{minipage}
\par
\penalty0
\vspace{0.6em}

\noindent
\begin{minipage}[t]{0.58\linewidth}
\vspace{0pt}
\begin{enumerate}[
    label=(L\arabic*$_{oo}$), start=3,
    leftmargin=*, topsep=0pt, partopsep=0pt,
    itemsep=0pt, parsep=0pt
]
    \item\label{L3oo}
    If $\begin{cases}
        |g|<|h|, \text{ and }\\
        |gh|+|k|\geq 2n
    \end{cases}$
    \par\nobreak\vspace{1pt}\nointerlineskip
    \noindent\hspace*{-\leftmargin}%
    \begin{adjustbox}{max width=\dimexpr\linewidth+\leftmargin\relax}
    \begin{tikzpicture}
        \node at (0,0.15) {$\bigredcap$};
        \node at (-0.06,0.03) {$\redcap$};
        \node at (3.27,0.7) {$\jellyontopofjelly$};
        \node at (5.7,0.84) {$\longredstrand$};
        \node at (6.4,0.84) {$\longgreenstrand$};

        \node[scale=0.5] at (-0.47,-0.3) {$\cdots$};
        \node[scale=0.5] at (0.58,-0.3) {$\cdots$};
        \node[scale=0.5] at (6.06,-0.3) {$\cdots$};

        \draw[decorate, decoration={brace,mirror}]
            (-0.9,-0.5) -- (-0.1,-0.5)
            node[midway, below=3pt] {$g$};
        \draw[decorate, decoration={brace,mirror}]
            (0.2,-0.5) -- (2.7,-0.5)
            node[midway, below=3pt] {$h$};
        \draw[decorate, decoration={brace,mirror}]
            (4.6,-0.5) -- (6.5,-0.5)
            node[midway, below=3pt] {$k$};
    \end{tikzpicture}
    \end{adjustbox}
\end{enumerate}
\end{minipage}%
\hfill
\begin{minipage}[t]{0.39\linewidth}
\vspace{0pt}
\begin{enumerate}[
    label=(R\arabic*$_{oo}$), start=3,
    leftmargin=*, topsep=0pt, partopsep=0pt,
    itemsep=0pt, parsep=0pt
]
    \item\label{R3oo}
    If $\begin{cases}
        |h|<|k|, \text{ and }\\
        |g|\geq |hk|
    \end{cases}$
    \par\nobreak\vspace{1pt}\nointerlineskip
    \noindent\hspace*{-\leftmargin}%
    \begin{adjustbox}{max width=\dimexpr\linewidth+\leftmargin\relax}
    \begin{tikzpicture}[scale=.8]
        \draw[thick] (0,-.3) -- (0,2);
        \draw[thick] (0.5,-.3) -- (0.5,2);
        \draw[thick] (.8,-.3) -- (.8,.5);
        \draw[thick] (1.3,-.3) -- (1.3,.3);
        \draw[thick, bend left=90] (1.8,-.3) to (3.3,-.3);
        \draw[thick, bend left=90] (2.3,-.3) to (2.8,-.3);

        \draw[thick, name path=arch]
            (3.6,-.3) -- (3.6,.3)
            to[out=90,in=90] (4.9,.3);

        \path[name path=probe1] (4.1,0) -- (4.1,1.2);
        \path[name intersections={of=arch and probe1,by=Q1}];
        \draw[thick] (4.1,-.3) -- (Q1);

        \path[name path=probe2] (4.4,0) -- (4.4,1.2);
        \path[name intersections={of=arch and probe2,by=Q2}];
        \draw[thick] (4.4,0) -- (Q2);

        \draw[thick] (4.9,.3) arc (-180:0:.2cm);
        \draw[thick, bend left=90] (1.3,.3) to (5.3,.3);
        \draw[thick, out=-90,in=-90] (4.4,0) to (5.8,0);
        \draw[thick] (5.8,0) -- (5.8,.5);
        \draw[thick, out=90,in=90] (.8,.5) to (5.8,.5);

        \foreach \x in {0,1.3,1.8,2.3,2.8,3.3}
            \filldraw[r] (\x,-.3) circle (1.5pt);
        \foreach \x in {0.5,4.1}
            \filldraw[g] (\x,-.3) circle (1.5pt);

        \filldraw[p] (.8,-.3) circle (1.5pt);
        \filldraw[b] (3.6,-.3) circle (1.5pt);
        \filldraw[r] (0,2) circle (1.5pt);
        \filldraw[g] (.5,2) circle (1.5pt);

        \foreach \x in {0.25,1.05,2.05,3.85}
            \node at (\x,-.35) {\scriptsize{$\cdots$}};
        \node at (4.65,.2) {\scriptsize{$\cdots$}};

        \draw[decorate, decoration={brace,mirror}]
            (-.1,-.5) -- (1.4,-.5)
            node[midway, below=3pt] {$g$};
        \draw[decorate, decoration={brace,mirror}]
            (1.7,-.5) -- (2.4,-.5)
            node[midway, below=3pt] {$h$};
        \draw[decorate, decoration={brace,mirror}]
            (2.7,-.5) -- (4.2,-.5)
            node[midway, below=3pt] {$k$};
    \end{tikzpicture}
    \end{adjustbox}
\end{enumerate}
\end{minipage}
\par
\penalty0
\vspace{0.6em}

\noindent
\begin{minipage}[t]{0.58\linewidth}
\vspace{0pt}
\begin{enumerate}[
    label=(L\arabic*$_{oo}$), start=4,
    leftmargin=*, topsep=0pt, partopsep=0pt,
    itemsep=0pt, parsep=0pt
]
    \item\label{L4oo}
    If $\begin{cases}
        |g|<|h|, \text{ and }\\
        |gh|+|k|<2n
    \end{cases}$
    \par\nobreak\vspace{1pt}\nointerlineskip
    \noindent\hspace*{-\leftmargin}%
    \begin{adjustbox}{max width=\dimexpr\linewidth+\leftmargin\relax}
    \begin{tikzpicture}
        \node at (0,0.15) {$\bigredcap$};
        \node at (-0.06,0.03) {$\redcap$};
        \node at (2.7,.3) {$\brjellyswitchr$};
        \node at (4.6,.3) {$\mediumxredstrand$};
        \node at (5.1,.3) {$\mediumxgreenstrand$};

        \node[scale=0.5] at (-0.47,-.3) {$\cdots$};
        \node[scale=0.5] at (0.58,-.3) {$\cdots$};
        \node[scale=0.5] at (4.85,-.3) {$\cdots$};

        \draw[decorate, decoration={brace,mirror}]
            (-.9,-.5) -- (-.1,-.5)
            node[midway, below=3pt] {$g$};
        \draw[decorate, decoration={brace,mirror}]
            (.2,-.5) -- (2.7,-.5)
            node[midway, below=3pt] {$h$};
        \draw[decorate, decoration={brace,mirror}]
            (4.45,-.5) -- (5.2,-.5)
            node[midway, below=3pt] {$k$};
    \end{tikzpicture}
    \end{adjustbox}
\end{enumerate}
\end{minipage}%
\hfill
\begin{minipage}[t]{0.39\linewidth}
\vspace{0pt}
\begin{enumerate}[
    label=(R\arabic*$_{oo}$), start=4,
    leftmargin=*, topsep=0pt, partopsep=0pt,
    itemsep=0pt, parsep=0pt
]
    \item\label{R4oo}
    If $\begin{cases}
        |h|<|k|, \text{ and }\\
        |g|<|hk|
    \end{cases}$
    \par\nobreak\vspace{1pt}\nointerlineskip
    \noindent\hspace*{-\leftmargin}%
    \begin{adjustbox}{max width=\dimexpr\linewidth+\leftmargin\relax}
    \begin{tikzpicture}[scale=.8]
        \draw[thick] (0,.1) -- (0,1);
        \draw[thick] (.5,.1) -- (.5,1);
        \draw[thick, bend left=90] (.8,.1) to (2.2,.1);
        \draw[thick, bend left=90] (1.3,.1) to (1.7,.1);
        \draw[thick, bend left=90] (0,1) to (5.5,1);
        \draw[thick, bend left=90] (.5,1) to (5,1);

        \draw[thick, name path=arch1]
            (2.5,.1) -- (2.5,1)
            to[out=90,in=90] (4.6,1);

        \path[name path=probe1] (3,0) -- (3,2.7);
        \path[name intersections={of=arch1 and probe1,by=Q1}];
        \draw[thick] (3,.1) -- (Q1);

        \path[name path=probe2] (3.3,1) -- (3.3,2.7);
        \path[name intersections={of=arch1 and probe2,by=Q2}];
        \draw[thick] (3.3,1) -- (Q2);

        \path[name path=probe3] (3.8,1) -- (3.8,2.7);
        \path[name intersections={of=arch1 and probe3,by=Q3}];
        \draw[thick] (3.8,1) -- (Q3);

        \path[name path=probe4] (4.1,1) -- (4.1,2.7);
        \path[name intersections={of=arch1 and probe4,by=Q4}];
        \draw[thick] (4.1,1) -- (Q4);

        \draw[thick, bend right=90] (3.3,1) to (6.3,1);
        \draw[thick, bend right=90] (3.8,1) to (5.8,1);
        \draw[thick, bend right=90] (4.1,1) to (5.5,1);
        \draw[thick] (4.6,1) arc (-180:0:.2cm);

        \draw[thick, name path=arch2]
            (5.8,1) -- (5.8,1.5)
            to[out=90,in=90] (7.1,1.5);

        \path[name path=probe5] (6.3,1) -- (6.3,2.7);
        \path[name intersections={of=arch2 and probe5,by=Q5}];
        \draw[thick] (6.3,1) -- (Q5);

        \path[name path=probe6] (6.6,1) -- (6.6,2.7);
        \path[name intersections={of=arch2 and probe6,by=Q6}];
        \draw[thick] (6.6,1) -- (Q6);

        \draw[thick, bend right=90]
            (6.6,1) to (8,1) -- (8,2.5);
        \draw[thick]
            (7.1,1.5) -- (7.1,1.2)
            arc (-180:0:.2cm) -- (7.5,2.5);

        \foreach \x in {0,.5,.8,1.3,1.7,2.2}
            \filldraw[r] (\x,.1) circle (1.5pt);
        \foreach \x in {2.5}
            \filldraw[b] (\x,.1) circle (1.5pt);
        \foreach \x in {3}
            \filldraw[g] (\x,.1) circle (1.5pt);

        \filldraw[r] (7.5,2.5) circle (1.5pt);
        \filldraw[g] (8,2.5) circle (1.5pt);

        \foreach \x in {4.1,4.6}
            \filldraw[r] (\x,1) circle (1.5pt);
        \foreach \x in {3.8,5.8}
            \filldraw[b] (\x,1) circle (1.5pt);

        \filldraw[p] (3.3,1) circle (1.5pt);
        \filldraw[p] (6.3,1) circle (1.5pt);

        \foreach \x in {.25,1.05,2.75}
            \node at (\x,.2) {$\scriptstyle\cdots$};
        \foreach \x in {3.55,4.35,6.05}
            \node at (\x,1.1) {$\scriptstyle\cdots$};
        \node at (7.75,2.4) {$\scriptstyle\cdots$};

        \draw[decorate, decoration={brace,mirror}]
            (-.1,-.1) -- (.6,-.1)
            node[midway, below=3pt] {$g$};
        \draw[decorate, decoration={brace,mirror}]
            (.7,-.1) -- (1.4,-.1)
            node[midway, below=3pt] {$h$};
        \draw[decorate, decoration={brace,mirror}]
            (1.6,-.1) -- (3.1,-.1)
            node[midway, below=3pt] {$k$};
    \end{tikzpicture}
    \end{adjustbox}
\end{enumerate}
\end{minipage}
\par

\begin{lemma}
    When $|g|$ and $|h|$  are odd, we have that
    \begin{align*}
        \omega(g,h,k) = \begin{cases}
            \epsilon_r^{\lceil |h|/2 \rceil}\epsilon_b^{\lfloor |h|/2 \rfloor} \xi^{\lfloor |g|/2\rfloor}, \text{ if } |g|\geq |h|, |gh|+|k|\geq 2n, |h|< |k|, \text{ and } |g|\geq |hk|,\\ 
           \epsilon_r^{\lceil |k|/2\rceil} \epsilon_b^{\lfloor |k|/2\rfloor}, \text{ if } |g| \geq |h|, |gh|+|k|<2n, |h|\geq |k|, \text{ and }|g|\geq |hk|,\\\
            \epsilon_r^{\lceil |kg|/2 \rceil} \epsilon_b^{\lfloor |kg|/2 \rfloor} \xi^{\lfloor |g|/2\rfloor},\text{ if } |g| \geq |h|, |gh|+|k|<2n, |h|<|k|, \text{ and } |g|<|hk|,\\
            \epsilon_r^{\lceil |ghk|/2 \rceil} \epsilon_b^{\lfloor |ghk|/2\rfloor}, \text{ if } |g|<|h|, |gh|+|k|\geq 2n, |h|\geq|k|, \text{ and } |g|\geq |hk|,\\
            \epsilon_r^{\lceil |kh|/2\rceil}\epsilon_b^{\lfloor  |kh|/2\rfloor}\xi^{\lfloor |g|/2\rfloor} \text{ if } |g| < |h|, |gh|+|k|\geq 2n, |h|< |k|, \text{ and } |g|<|hk|, \text{ and }\\
            1 \text{ otherwise}.
        \end{cases}
    \end{align*}
\end{lemma}

\begin{proof}
 Here are the 6 possible combinations of left trees and right trees.
\begin{center}
\renewcommand{\arraystretch}{1.2}
\begin{tabular}{lll}
(1) \ref{L1oo} and \ref{R3oo}
&
(2) \ref{L2oo} and \ref{R1oo}
&
(3) \ref{L2oo} and \ref{R4oo}
\\
(4) \ref{L3oo} and \ref{R1oo}
&
(5) \ref{L3oo} and \ref{R4oo}
&
(6) \ref{L4oo} and \ref{R2oo}
\end{tabular}
\end{center}
\noindent \underline{Case (1).}  Starting with \ref{R3oo}, we first apply $\text{Definition \ref{def: blue jelly}}$ to the jellyfish. Up to straightening 2 zigzags using  $\ref{rel: zigzag}$, we obtain

\[\vcenter{\hbox{\begin{tikzpicture}[scale=.8]
    \draw[thick] (0,-.3) -- (0,2);
    \draw[thick] (0.5,-.3) -- (0.5,2);
    \draw[thick] (.8,-.3) -- (.8,.5);
    \draw[thick] (1.3,-.3) -- (1.3,.3);
    \draw[thick,name path=arch] (1.6,-.3) -- (1.6, .3) to [out=90, in=90] (4.9,.3);
    \draw[thick, bend left =90] (1.9,-.3) to (3.3,-.3) ;
    \draw[thick, bend left =90] (2.4,-.3) to (2.8,-.3) ;
    \path[name path=probe0] (3.6,0) -- (3.6,2);
    \path[name intersections={of=arch and probe0, by=Q}];
    \draw[thick] (3.6,-.3) -- (Q); 
    \path[name path=probe1] (4.1,0) -- (4.1,1.2);
    \path[name intersections={of=arch and probe1, by=Q1}];
    \draw[thick] (4.1,-.3) -- (Q1); 
    \path[name path=probe2] (4.4,0) -- (4.4,1.2);
    \path[name intersections={of=arch and probe2, by=Q2}];
    \draw[thick] (4.4,0) -- (Q2); 
    \draw[thick] (4.9,.3) arc (-180:0:.2cm);
    \draw[thick, bend left =90] (1.3,0.3) to (5.3,0.3) ;
    \draw[thick, out = -90,in=-90] (4.4,0) to (5.8,0) ;
    \draw[thick] (5.8,0) -- (5.8,.5);
    \draw[thick, out=90,in=90] (.8,0.5) to (5.8,0.5) ;
    \foreach \x in {0, 1.6,1.9,2.4,2.8,3.3}
        \filldraw[r] (\x,-.3) circle (1.5pt);
    \foreach \x in {0.5, 4.1}
        \filldraw[g] (\x,-.3) circle (1.5pt);
    \foreach \x in {1.3, 3.6}
        \filldraw[b] (\x,-.3) circle (1.5pt);
    \foreach \x in {.8}
        \filldraw[p] (\x,-.3) circle (1.5pt);
    \filldraw[r] (0,2) circle (1.5pt);
    \filldraw[g] (.5,2) circle (1.5pt);
    \foreach \x in {0.25, 1.05, 2.15,3.85}
    \node at (\x,-.2) {$\scriptstyle\cdots$};
    \node at (4.65,.2) {$\scriptstyle\cdots$};
    \draw[decorate, decoration={brace,mirror}] (-.1,-0.5) -- (1.7,-0.5) node[midway, below=3pt] {$g$};
    \draw[decorate, decoration={brace,mirror}] (1.8,-0.5) -- (2.5,-0.5) node[midway, below=3pt] {$h$};
    \draw[decorate, decoration={brace,mirror}] (2.7, -0.5) -- (4.2, -0.5) node[midway, below=3pt] {$k$};
\end{tikzpicture}}}\ \ . \]

Apply $\text{\ref{rel:redbluejelly}}$ to the jellyfish and the $|g|-1$ vertical strands to the left of it for $\frac{|g|-1}{2}$ times to obtain

\[
\xi^{\frac{|g|-1}{2}}
\vcenter{\hbox{\begin{tikzpicture}
\draw[thick](0,0) -- (0,.3);
\draw[thick](3.1,0) -- (3.1,.3);
\draw[thick,name path=arch] (0,.3) to [out=90, in=90] (3.1,.3);
\draw[thick, bend left =90] (.8,0) to (2.3,0) ;
\draw[thick, bend left =90] (1.3,0) to (1.8,0);
\draw[thick](3.4,0) -- (3.4,1.2);
\draw[thick](3.9,0) -- (3.9,1.2);
\path[name path=probe1] (.5,0) -- (.5,1.5);
    \path[name intersections={of=arch and probe1, by=Q1}];
    \draw[thick] (.5,0) -- (Q1); 
    \path[name path=probe2] (2.6,0) -- (2.6,1.5);
    \path[name intersections={of=arch and probe2, by=Q2}];
    \draw[thick] (2.6,0) -- (Q2); 
    \foreach \x in {0, .5,.8, 1.3,1.8,2.3,3.4}
        \filldraw[r] (\x,0) circle (1.5pt);
    \foreach \x in {2.6, 3.1}
        \filldraw[b] (\x,0) circle (1.5pt);
    \filldraw[g] (3.9,0) circle (1.5pt); 
    \filldraw[r] (3.4,1.2) circle (1.5pt);
    \filldraw[g] (3.9,1.2) circle (1.5pt);
    \foreach \x in {0.25, 1.05, 2.85,3.65}
    \node at (\x,.1) {$\scriptstyle\cdots$};
\draw[decorate, decoration={brace,mirror}] (-.1,-0.2) -- (.6,-0.2) node[midway, below=3pt] {$g$};
    \draw[decorate, decoration={brace,mirror}] (.7,-0.2) -- (1.4,-0.2) node[midway, below=3pt] {$h$};
    \draw[decorate, decoration={brace,mirror}] (1.7, -0.2) -- (4, -0.2) node[midway, below=3pt] {$k$};
\end{tikzpicture}}}\]
Finally, apply $\text{\ref{rel:redstrandcap}}$ and $\text{\ref{rel:bluestrandcap}}$ to the red and blue caps together with vertical strands to the left of them for a total of $|h|$ times, to obtain $\xi^{\frac{|g|-1}{2}} \epsilon_r^{\lceil \frac{|h|}{2} \rceil}\epsilon_b^{\lfloor \frac{|h|}{2} \rfloor} $ times \ref{L1oo}.


\noindent\underline{Case (2)}.
We start with \ref{L2oo}. Using \ref{rel: redstrandcap} and \ref{rel: bluestrandcap} on the rightmost $|k|$ strands, we obtain $\epsilon_r^{\lceil |k|/2\rceil} \epsilon_b^{\lfloor |k|/2\rfloor}$ times the right tree. Thus $\omega(g,h,k)=\epsilon_r^{\lceil |k|/2\rceil} \epsilon_b^{\lfloor |k|/2\rfloor}$.

\noindent\underline{Case (3)}. We start with \ref{R4oo}. First, use relation \ref{rel: redbluejelly} to move the rightmost $|g|-1$ strands of the left hand jellyfish to the left. Using then Definition \ref{def: blue jelly} this results in $\xi^{(|g|-1)/2}=\xi^{\lfloor |g|/2\rfloor}$ multiplied by 
\begin{align*}
    \twojellyswithredcaps
\end{align*}
Next, move the leftmost $|ghk|$ strands using relations \ref{rel: bluestrandcap} and \ref{rel: redstrandcap} to obtain $\xi^{\lfloor|g|/2\rfloor}\epsilon_r^{\lceil|ghk|/2\rceil}\epsilon_b^{\lfloor|ghk|/2\rfloor}$ times 
\begin{align*}
    \twojellysoverredcupswithstrands \lowcomma
\end{align*}
where there are $|g|$ straight strands on the left and $|h|$ caps in the middle. Next, evaluate the two connected jellyfish to 1. Finally, to obtain the left tree, use relations \ref{rel: bluestrandcap} and \ref{rel: redstrandcap} on the $|h|$ caps in the middle to get $\xi^{\lfloor |g|/2\rfloor}\epsilon_r^{\lceil|ghk|/2\rceil}\epsilon_b^{\lfloor |ghk|/2\rfloor} \epsilon_r^{\lceil |h|/2\rceil}\epsilon_b^{\lfloor |h|/2\rfloor}=\epsilon_r^{\lceil |kg|/2 \rceil} \epsilon_b^{\lfloor |kg|/2 \rfloor} \xi^{\lfloor |g|/2\rfloor}$ times the left tree. Therefore, $\omega(g,h,k)=\epsilon_r^{\lceil |kg|/2 \rceil} \epsilon_b^{\lfloor |kg|/2 \rfloor} \xi^{\lfloor |g|/2\rfloor}$. 

\noindent\underline{Case (4)}.
Starting with \ref{R1oo}, apply $\ref{rel: redstrandcap}, \ref{rel: bluestrandcap}$ to the straight strands on the left and the group of archs on the far right for a total number of $|ghk|$ times, to get

\[
\epsilon_r^{\lceil \frac{|ghk|}{2} \rceil}\epsilon_b^{\lfloor \frac{|ghk|}{2} \rfloor}
\vcenter{\hbox{\begin{tikzpicture}
\draw[thick] (3.6,0) -- (3.6,.5);
    \draw[thick] (4.1,0) -- (4.1,.5);
    \draw[thick, bend left =90] (0,0) to (1.5,0) ;
    \draw[thick, bend left =90] (.5,0) to (1,0) ;
    \draw[thick, bend left =90] (1.8,0) to (3.3,0) ;
    \draw[thick, bend left =90] (2.3,0) to (2.8,0) ;
    \foreach \x in {0, .5,1,1.5,2.3,2.8,3.6}
        \filldraw[r] (\x,0) circle (1.5pt);
    \foreach \x in {1.8,3.3}
        \filldraw[b] (\x,0) circle (1.5pt);
    \filldraw[r] (3.6,0.5) circle (1.5pt);
    \filldraw[g] (4.1,0) circle (1.5pt);
    \filldraw[g] (4.1,.5) circle (1.5pt);
    \foreach \x in {0.25, 2.05,3.85}
        \node at (\x,.1) {$\scriptstyle\cdots$};
    \draw[decorate, decoration={brace,mirror}] (-.1,-0.2) -- (.6,-0.2) node[midway, below=3pt] {$g$};
    \draw[decorate, decoration={brace,mirror}] (.9,-0.2) -- (2.4,-0.2) node[midway, below=3pt] {$h$};
    \draw[decorate, decoration={brace,mirror}] (2.7, -0.2) -- (4.2, -0.2) node[midway, below=3pt] {$k$};
\end{tikzpicture}}} 
\]
By Lemma \ref{lem: brtwojellys}, we may add a closed diagram without changing the scalar:
\[
\epsilon_r^{\lceil \frac{|ghk|}{2} \rceil}\epsilon_b^{\lfloor \frac{|ghk|}{2} \rfloor}
\vcenter{\hbox{\begin{tikzpicture}
    \draw[thick](5.2,0) -- (5.2,2.5);
    \draw[thick](5.7,0) -- (5.7,2.5);
    \draw[thick, bend left =90] (1.8,0) to (4.9,0) ;
    \draw[thick, bend left =90] (2.3,0) to (4.4,0);
    \draw[thick, bend left =90] (0,0) to (1.5,0) ;
    \draw[thick, bend left =90] (.5,0) to (1,0);
\draw[thick, bend right =90] (1.8,2) to (4.9,2) ;
\draw[thick, bend right =90] (2.3,2) to (4.4,2);
\draw[thick, bend right =90] (2.6,2) to (4.1,2) ;
\draw[thick] (3.1,2) arc (-180:0:.25cm);
\draw[thick] (1.8,2) -- (1.8,2.2);
\draw[thick] (3.1,2) -- (3.1,2.2);
\draw[thick] (3.6,2) -- (3.6,2.2);
\draw[thick] (4.9,2) -- (4.9,2.2);
\draw[thick,name path=arch1] (1.8,2.2) to [out=90, in=90] (3.1,2.2);
\draw[thick,name path=arch2] (3.6,2.2) to [out=90, in=90] (4.9,2.2);
\path[name path=probe1] (2.3,2) -- (2.3,2.7);
    \path[name intersections={of=arch1 and probe1, by=Q1}];
    \draw[thick] (2.3,2) -- (Q1); 
    \path[name path=probe2] (2.6,2) -- (2.6,2.7);
    \path[name intersections={of=arch1 and probe2, by=Q2}];
    \draw[thick] (2.6,2) -- (Q2); 
\path[name path=probe3] (4.1,2) -- (4.1,2.7);
    \path[name intersections={of=arch2 and probe3, by=Q3}];
    \draw[thick] (4.1,2) -- (Q3); 
    \path[name path=probe4] (4.4,2) -- (4.4,2.7);
    \path[name intersections={of=arch2 and probe4, by=Q4}];
    \draw[thick] (4.4,2) -- (Q4); 
    \foreach \x in {0, .5,1,1.5,2.3,4.4,5.2}
        \filldraw[r] (\x,0) circle (1.5pt);
    \foreach \x in {1.8,4.9}
        \filldraw[b] (\x,0) circle (1.5pt);
    \foreach \x in {2.3,3.1,3.6, 4.4}
        \filldraw[r] (\x,2.1) circle (1.5pt);
    \foreach \x in {1.8,2.6,4.1,4.9}
        \filldraw[b] (\x,2.1) circle (1.5pt);
    \filldraw[r] (5.2,2.5) circle (1.5pt);
    \filldraw[g] (5.7,0) circle (1.5pt);
    \filldraw[g] (5.7,2.5) circle (1.5pt);
    \foreach \x in {0.25, 2.05,5.45}
        \node at (\x,.1) {$\scriptstyle\cdots$};
    \foreach \x in {2.05,2.85}
        \node at (\x,2) {$\scriptstyle\cdots$};
    \draw[decorate, decoration={brace,mirror}] (-.1,-0.2) -- (.6,-0.2) node[midway, below=3pt] {$g$};
    \draw[decorate, decoration={brace,mirror}] (.9,-0.2) -- (2.4,-0.2) node[midway, below=3pt] {$h$};
    \draw[decorate, decoration={brace,mirror}] (4.3, -0.2) -- (5.8, -0.2) node[midway, below=3pt] {$k$};
\end{tikzpicture}}}
\]
Apply \ref{rel: rcupcap} and \ref{rel: bcupcap} to the closed diagram and the archs below it, to get $\epsilon_r^{\lceil \frac{|ghk|}{2} \rceil}\epsilon_b^{\lfloor \frac{|ghk|}{2} \rfloor} $ times \ref{L3oo}.

\noindent\underline{Case (5)}.  We start with \ref{R4oo}. Apply \ref{rel: blueredjelly} for $\frac{|g|-1}{2}$ times to move the left jellyfish to the far left, and use \ref{eq: red up} to make this jellyfish start with a red strand; we get
\[
\xi^{\frac{|g|-1}{2}}
\vcenter{\hbox{\begin{tikzpicture}
\draw[thick,name path=arch1] (0,0) -- (0,.5) to [out=90, in=90] (3.8,.5);
\draw[thick, bend left =90] (.8,0) to (2.2,0) ;
\draw[thick, bend left =90] (1.3,0) to (1.7,0);
    \path[name path=probe1] (.5,0) -- (.5,1.5);
    \path[name intersections={of=arch1 and probe1, by=Q1}];
    \draw[thick] (.5,0) -- (Q1); 
    \path[name path=probe2] (2.5,0) -- (2.5,1.8);
    \path[name intersections={of=arch1 and probe2, by=Q2}];
    \draw[thick] (2.5,0) -- (Q2); 
    \path[name path=probe3] (3,0) -- (3,1.5);
    \path[name intersections={of=arch1 and probe3, by=Q3}];
    \draw[thick] (3,0) -- (Q3); 
    \path[name path=probe4] (3.3,0.5) -- (3.3,1.5);
    \path[name intersections={of=arch1 and probe4, by=Q4}];
    \draw[thick] (3.3,0.5) -- (Q4); 
\draw[thick, bend right =90] (3.3,.5) to (4.7,.5);
\draw[thick, bend right =90] (3.8,.5) to (4.2,.5);
    \draw[thick] (4.2,.5) -- (4.2,1);
    \draw[thick] (5.5,.5) -- (5.5,1);
\draw[thick,name path=arch2] (4.2,1) to [out=90, in=90] (5.5,1);
    \path[name path=probe5] (4.7,.5) -- (4.7,2);
    \path[name intersections={of=arch2 and probe5, by=Q5}];
    \draw[thick] (4.7,.5) -- (Q5); 
    \path[name path=probe6] (5,.5) -- (5,2);
    \path[name intersections={of=arch2 and probe6, by=Q6}];
    \draw[thick] (5,.5) -- (Q6); 
\draw[thick, bend right =90] (5,.5) to (6.4,.5) -- (6.4,1.6);
\draw[thick, bend right =90] (5.5,.5) to (5.9,.5) -- (5.9,1.6);
    \foreach \x in {0, .5,.8, 1.3,1.7,2.2}
        \filldraw[r] (\x,0) circle (1.5pt);
    \foreach \x in {2.5}
        \filldraw[b] (\x,0) circle (1.5pt);
    \filldraw[g] (3,0) circle (1.5pt); 
    \filldraw[r] (5.9,1.6) circle (1.5pt);
    \filldraw[g] (6.4,1.6) circle (1.5pt);
    \foreach \x in {3.8,4.2}
        \filldraw[b] (\x,0.5) circle (1.5pt);
    \foreach \x in {3.3,4.7}
        \filldraw[p] (\x,0.5) circle (1.5pt);
    \foreach \x in {0.25,1.05, 2.75}
    \node at (\x,.1) {$\scriptstyle\cdots$};
    \node at (3.55,.6) {$\scriptstyle\cdots$};
    \node at (4.45,.6) {$\scriptstyle\cdots$};
    \node at (6.15,1.5) {$\scriptstyle\cdots$};
\draw[decorate, decoration={brace,mirror}] (-.1,-0.2) -- (.6,-0.2) node[midway, below=3pt] {$g$};
    \draw[decorate, decoration={brace,mirror}] (.7,-0.2) -- (1.4,-0.2) node[midway, below=3pt] {$h$};
    \draw[decorate, decoration={brace,mirror}] (1.6, -0.2) -- (3.1, -0.2) node[midway, below=3pt] {$k$};
\end{tikzpicture}}}\]
Next, apply \ref{rel: redstrandcap} and \ref{rel: bluestrandcap} for a total number of $|g|$ times to the archs nested with the left jellyfish and strands to the left of them, we get
\[
\xi^{\frac{|g|-1}{2}} \epsilon_r^{\lceil \frac{|g|}{2} \rceil}\epsilon_b^{\lfloor \frac{|g|}{2} \rfloor}
\vcenter{\hbox{\begin{tikzpicture}
    \draw[thick](3.4,0) -- (3.4,.5);
    \draw[thick, bend left =90] (0,0) to (1.4,0) ;
    \draw[thick, bend left =90] (.5,0) to (.9,0);
    \draw[thick, bend left =90] (1.7,0) to (3.1,0) ;
    \draw[thick, bend left =90] (2.2,0) to (2.6,0);
\draw[thick, bend right =90] (4.2,.5) to (5.6,.5) ;
\draw[thick] (4.7,.5) arc (-180:0:.2cm);
\draw[thick,name path=arch1] (3.4,.5) to [out=90, in=90] (4.7,.5);
\draw[thick,name path=arch2] (5.1,.5) to [out=90, in=90] (6.4,.5);
    \path[name path=probe1] (3.9,0) -- (3.9,1.2);
    \path[name intersections={of=arch1 and probe1, by=Q1}];
    \draw[thick] (3.9,0) -- (Q1);
    \path[name path=probe2] (4.2,.5) -- (4.2,1.2);
    \path[name intersections={of=arch1 and probe2, by=Q2}];
    \draw[thick] (4.2,.5) -- (Q2); 
    \path[name path=probe3] (5.6,.5) -- (5.6,1.2);
    \path[name intersections={of=arch2 and probe3, by=Q3}];
    \draw[thick] (5.6,.5) -- (Q3); 
    \path[name path=probe4] (5.9,.5) -- (5.9,1.2);
    \path[name intersections={of=arch2 and probe4, by=Q4}];
    \draw[thick] (5.9,.5) -- (Q4); 
\draw[thick, bend right =90] (5.9,.5) to (7.3,.5) -- (7.3,1) ;
\draw[thick] (6.4,.5) arc (-180:0:.2cm) -- (6.8,1);
    \foreach \x in {0, .5,.9,1.4,2.2,2.6,3.4}
        \filldraw[r] (\x,0) circle (1.5pt);
    \foreach \x in {1.7,3.1}
        \filldraw[b] (\x,0) circle (1.5pt);
    \filldraw[r] (6.8,1) circle (1.5pt);
    \filldraw[g] (7.3,1) circle (1.5pt);
    \filldraw[g] (3.9,0) circle (1.5pt);
    \foreach \x in {4.7,5.1}
        \filldraw[b] (\x,0.5) circle (1.5pt);
    \foreach \x in {4.2,5.6}
        \filldraw[p] (\x,0.5) circle (1.5pt);
    \foreach \x in {0.25, 1.95,3.65}
        \node at (\x,.1) {$\scriptstyle\cdots$};
    \foreach \x in {4.45,5.35,6.15}
        \node at (\x,.4) {$\scriptstyle\cdots$};
    \node at (7.05,.9) {$\scriptstyle\cdots$};
    \draw[decorate, decoration={brace,mirror}] (-.1,-0.2) -- (.6,-0.2) node[midway, below=3pt] {$g$};
    \draw[decorate, decoration={brace,mirror}] (.8,-0.2) -- (2.3,-0.2) node[midway, below=3pt] {$h$};
    \draw[decorate, decoration={brace,mirror}] (2.5, -0.2) -- (4, -0.2) node[midway, below=3pt] {$k$};
\end{tikzpicture}
}}
\]
Apply \ref{rel: redstrandcup} and \ref{rel: bluestrandcup} to the right-most group of cups and the unbent strands in the left jellyfish for a total number of $|ghk|$ times, we get
\[
\xi^{\frac{|g|-1}{2}} \epsilon_r^{\lceil \frac{|g|}{2} \rceil+ \lceil \frac{|ghk|}{2} \rceil}\epsilon_b^{\lfloor \frac{|g|}{2} \rfloor+\lfloor \frac{|ghk|}{2} \rfloor}
\vcenter{\hbox{\begin{tikzpicture}
    \draw[thick, bend left =90] (-1.6,1.1) to (-.2,1.1) ;
    \draw[thick, bend left =90] (-1.1,1.1) to (-.7,1.1) ;
    \draw[thick, bend left =90] (0.1,1.1) to (1.5,1.1) ;
    \draw[thick, bend left =90] (.6,1.1) to (1,1.1) ;
\draw[thick] (5.2,1.1) -- (5.2,2.6);
    \draw[thick] (5.7,1.1) -- (5.7,2.6);
\draw[thick, bend right =90] (1.8,2) to (4.9,2) ;
\draw[thick, bend right =90] (2.3,2) to (4.4,2);
\draw[thick, bend right =90] (2.6,2) to (4.1,2) ;
\draw[thick] (3.1,2) arc (-180:0:.25cm);
\draw[thick] (1.8,2) -- (1.8,2.2);
\draw[thick] (3.1,2) -- (3.1,2.2);
\draw[thick] (3.6,2) -- (3.6,2.2);
\draw[thick] (4.9,2) -- (4.9,2.2);
\draw[thick,name path=arch1] (1.8,2.2) to [out=90, in=90] (3.1,2.2);
\draw[thick,name path=arch2] (3.6,2.2) to [out=90, in=90] (4.9,2.2);
\path[name path=probe1] (2.3,2) -- (2.3,2.7);
    \path[name intersections={of=arch1 and probe1, by=Q1}];
    \draw[thick] (2.3,2) -- (Q1); 
    \path[name path=probe2] (2.6,2) -- (2.6,2.7);
    \path[name intersections={of=arch1 and probe2, by=Q2}];
    \draw[thick] (2.6,2) -- (Q2); 
\path[name path=probe3] (4.1,2) -- (4.1,2.7);
    \path[name intersections={of=arch2 and probe3, by=Q3}];
    \draw[thick] (4.1,2) -- (Q3); 
    \path[name path=probe4] (4.4,2) -- (4.4,2.7);
    \path[name intersections={of=arch2 and probe4, by=Q4}];
    \draw[thick] (4.4,2) -- (Q4); 
    \foreach \x in {-1.6,-1.1, -.7,-.2,.6,1,5.2}
        \filldraw[r] (\x,1.1) circle (1.5pt);
    \foreach \x in {.1,1.5}
        \filldraw[b] (\x,1.1) circle (1.5pt);
    \foreach \x in {2.3,3.1,3.6, 4.4}
        \filldraw[b] (\x,2.1) circle (1.5pt);
    \foreach \x in {1.8,2.6,4.1,4.9}
        \filldraw[r] (\x,2.1) circle (1.5pt);
    \filldraw[r] (5.2,2.6) circle (1.5pt);
    \filldraw[g] (5.7,2.6) circle (1.5pt);
    \filldraw[g] (5.7,1.1) circle (1.5pt);
    \foreach \x in {-1.35,.35,5.45}
        \node at (\x,1.2) {$\scriptstyle\cdots$};
    \foreach \x in {2.05,2.85}
        \node at (\x,2) {$\scriptstyle\cdots$};
    \draw[decorate, decoration={brace,mirror}] (-1.7,.9) -- (-1,.9) node[midway, below=3pt] {$g$};
    \draw[decorate, decoration={brace,mirror}] (-.8,.9) -- (.7,.9) node[midway, below=3pt] {$h$};
    \draw[decorate, decoration={brace,mirror}] (.9,.9) -- (5.8,.9) node[midway, below=3pt] {$k$};
\end{tikzpicture}}}\]
By Remark \ref{rk:normalization} and Lemma \ref{lem: brtwojellys}, we can change the closed diagram into the one where the left jellyfish starts with blue, and stack it on top of the second group of caps. We get:
\[
\xi^{\frac{|g|-1}{2}} \epsilon_r^{\lceil \frac{|g|}{2} \rceil+ \lceil \frac{|ghk|}{2} \rceil}\epsilon_b^{\lfloor \frac{|g|}{2} \rfloor+\lfloor \frac{|ghk|}{2} \rfloor}
\vcenter{\hbox{\begin{tikzpicture}
    \draw[thick](5.2,0) -- (5.2,2.5);
    \draw[thick](5.7,0) -- (5.7,2.5);
    \draw[thick, bend left =90] (1.8,0) to (4.9,0) ;
    \draw[thick, bend left =90] (2.3,0) to (4.4,0);
    \draw[thick, bend left =90] (0,0) to (1.5,0) ;
    \draw[thick, bend left =90] (.5,0) to (1,0);
\draw[thick, bend right =90] (1.8,2) to (4.9,2) ;
\draw[thick, bend right =90] (2.3,2) to (4.4,2);
\draw[thick, bend right =90] (2.6,2) to (4.1,2) ;
\draw[thick] (3.1,2) arc (-180:0:.25cm);
\draw[thick] (1.8,2) -- (1.8,2.2);
\draw[thick] (3.1,2) -- (3.1,2.2);
\draw[thick] (3.6,2) -- (3.6,2.2);
\draw[thick] (4.9,2) -- (4.9,2.2);
\draw[thick,name path=arch1] (1.8,2.2) to [out=90, in=90] (3.1,2.2);
\draw[thick,name path=arch2] (3.6,2.2) to [out=90, in=90] (4.9,2.2);
\path[name path=probe1] (2.3,2) -- (2.3,2.7);
    \path[name intersections={of=arch1 and probe1, by=Q1}];
    \draw[thick] (2.3,2) -- (Q1); 
    \path[name path=probe2] (2.6,2) -- (2.6,2.7);
    \path[name intersections={of=arch1 and probe2, by=Q2}];
    \draw[thick] (2.6,2) -- (Q2); 
\path[name path=probe3] (4.1,2) -- (4.1,2.7);
    \path[name intersections={of=arch2 and probe3, by=Q3}];
    \draw[thick] (4.1,2) -- (Q3); 
    \path[name path=probe4] (4.4,2) -- (4.4,2.7);
    \path[name intersections={of=arch2 and probe4, by=Q4}];
    \draw[thick] (4.4,2) -- (Q4); 
    \foreach \x in {0, .5,1,1.5,2.3,4.4,5.2}
        \filldraw[r] (\x,0) circle (1.5pt);
    \foreach \x in {1.8,4.9}
        \filldraw[b] (\x,0) circle (1.5pt);
    \foreach \x in {2.3,3.1,3.6, 4.4}
        \filldraw[r] (\x,2.1) circle (1.5pt);
    \foreach \x in {1.8,2.6,4.1,4.9}
        \filldraw[b] (\x,2.1) circle (1.5pt);
    \filldraw[r] (5.2,2.5) circle (1.5pt);
    \filldraw[g] (5.7,0) circle (1.5pt);
    \filldraw[g] (5.7,2.5) circle (1.5pt);
    \foreach \x in {0.25, 2.05,5.45}
        \node at (\x,.1) {$\scriptstyle\cdots$};
    \foreach \x in {2.05,2.85}
        \node at (\x,2) {$\scriptstyle\cdots$};
    \draw[decorate, decoration={brace,mirror}] (-.1,-0.2) -- (.6,-0.2) node[midway, below=3pt] {$g$};
    \draw[decorate, decoration={brace,mirror}] (.9,-0.2) -- (2.4,-0.2) node[midway, below=3pt] {$h$};
    \draw[decorate, decoration={brace,mirror}] (4.3, -0.2) -- (5.8, -0.2) node[midway, below=3pt] {$k$};
\end{tikzpicture}}}\]
Finally, apply \ref{rel: rcupcap} and \ref{rel: bcupcap} to the stacked part to get $\xi^{\frac{|g|-1}{2}} \epsilon_r^{\lceil \frac{|g|}{2} \rceil+ \lceil \frac{|ghk|}{2} \rceil}\epsilon_b^{\lfloor \frac{|g|}{2} \rfloor+\lfloor \frac{|ghk|}{2} \rfloor} $ times \ref{L3oo}.

\noindent\underline{Case (6)}. We start with \ref{R2oo}. Use relation \ref{rel: bcupcap} and \ref{rel: rcupcap} on the rightmost $|k|$ cups and caps to obtain the left tree. Thus, $\omega(g,h,k)=1$.
\end{proof}

\section{Equivalence}\label{sec: equiv}

We denote by
$\VecDnx$
the pointed fusion category corresponding to the cocycle
$\omega_{\epsilon_r,\epsilon_b,\xi}$. 
In this Section we prove a presentation for $\VecDnx$.

\subsection{\texorpdfstring{Diagrammatics for $\VecDnx$}{Diagrammatics for $\VecDnx$}}\label{sec: diagrammatics}

Let $(\epsilon_r,\epsilon_b,\xi)$ be an admissible triple. 

\begin{definition}\label{tensor presentaion}
   Let $\Dx$ denote the $\mathsf{k}$-linear monoidal category generated by
 \begin{itemize}
     \item Object-generators:
	\(
	\reddot
	\qquad \text{and} \qquad
	 \bluedot.
	\)
    \item Morphism-generators: 
      \begin{align*}\
	\redstrand\lowcomma\;
	\bluestrand\lowcomma\;
	\redcap\lowcomma\;
	\bluecap\lowcomma\;
	\redcup\lowcomma\;
	\bluecup\lowcomma\;
	\rbjellyup\lowcomma\;
	\rbjellydown\lowperiod
\end{align*}
subject to the relations
\setcounter{savedequation}{\value{equation}}

\begingroup
\renewcommand{\theequation}{(\textnormal{R}\arabic{equation})}
\renewcommand{\theHequation}{R\arabic{equation}}
\setcounter{equation}{0}

\noindent\begin{tabular}{@{}p{0.29\linewidth}p{0.29\linewidth}p{0.40\linewidth}@{}}
\refstepcounter{equation}\label{rel:redstrandcap}%
$\theequation\quad \redstrand \redcap = \epsilon_r\redcap \redstrand$
&
\refstepcounter{equation}\label{rel:bluestrandcap}%
$\theequation\quad \bluestrand \bluecap = \epsilon_b\bluecap \bluestrand$
&
\refstepcounter{equation}\label{rel:redbluejelly}%
$\theequation\quad \redstrand \bluestrand \rbjellydown = \xi\rbjellydown \redstrand \bluestrand$
\\[0.8em]

\refstepcounter{equation}\label{rel:redcircle}%
$\theequation\quad \redcircle = \mathrm{id}_{\mathbbm 1}$
&
\refstepcounter{equation}\label{rel:bluecircle}%
$\theequation\quad \bluecircle = \mathrm{id}_{\mathbbm 1}$
&
\refstepcounter{equation}\label{rel:twojellys}%
$\theequation\quad \rbtwojellys = \mathrm{id}_{\mathbbm 1}$
\\[0.8em]

\refstepcounter{equation}\label{rel:rcupcap}%
$\theequation\quad \rcupcap = \redstrand\;\redstrand$
&
\refstepcounter{equation}\label{rel:bcupcap}%
$\theequation\quad \bcupcap = \bluestrand\;\bluestrand$
&
\refstepcounter{equation}\label{rel:jellysmultiplied}%
$\theequation\quad \rbtwojellysout = \redstrand\;\bluestrand\cdots\redstrand\;\bluestrand$
\\[0.8em]

&
\refstepcounter{equation}\label{rel:closedjellys}%
\makebox[\linewidth][c]{%
  $\theequation\quad \redredbubble=\mathrm{id}_{\mathbbm 1}$%
}
&
\end{tabular}
\endgroup
\setcounter{equation}{\value{savedequation}}
 \end{itemize}  
\end{definition}

As in Definition \ref{def: blue jelly}, we will also use the jellyfish beginning in blue. The proof of Lemma \ref{lem: brtwojellys} applies verbatim and gives the corresponding blue analogue of \ref{rel:closedjellys}.

We now state the main theorem of this section.

\begin{theorem}\label{thm diag category}
 The category $\Dx$ is a diagrammatic presentation of $\VecDnx.$
\end{theorem}

\subsection{Proof of Theorem \ref{thm diag category}}\label{sec: proof of thm}

By Section \ref{sec: presentation}, all the defining relations
of \(\Dx\) hold in \(\FGR\). 
Hence there is a canonical monoidal functor
$\Phi:
\Dx
\longrightarrow
\FGR$
sending each generating diagram in $\Dx$ to the
corresponding generator of
$\FGR$. We will show this is a monoidal equivalence.

The proof of Lemma \ref{lemma: red blue sequence} uses only the defining relations of $\Dx$. Hence every object $x$ in $\Dx$ is isomorphic to $\bar x$ in standard form.  We get the following preliminary results.

\begin{lemma}\label{lemma:zero hom} 
If $\bar x\ne \bar y,$ then $\Hom_{\Dx}(x,y)=0$ .
\end{lemma}
\begin{proof}
The source and target of each generating morphism of \(\Dx\) have the
same standard form.  This property is preserved under tensor products
and composition, hence every morphism in $\Dx$ preserves standard form, and the statement follows. 
\end{proof}

We show that $\dim(\Hom_{\Dx}(x,y))=1$ whenever $\bar x = \bar y.$ 

\begin{remark}\label{rk: std form}
Since
$\dim(\Hom_{\Dx}(x,y))=\dim(\Hom_{\Dx}(\bar x, \bar y))$, we may assume objects are in standard form. 
\end{remark}

Note that the proof of Lemma \ref{lemma: extra rels} uses only the defining relations of $\Dx$. Hence the auxiliary relations \ref{rel: redstrandcup}-\ref{eq: red up} also hold in $\Dx$.

\begin{lemma}\label{lemma: jellysareconnected}
    Let $w$ be a word in standard form and $D$ be a nonzero diagram in $\text{End}_{\Dx}(w)$. Then every jellyfish in $D$ has at least one strand connecting it to another jellyfish.
\end{lemma}

\begin{proof}
    Let $J$ be a jellyfish in $D$. Each of the $2n$ strands has four options of where they end: on $J$, on the top boundary, on the bottom boundary, or on another jellyfish. It is clear due to the alternating coloring of the strands of the jellyfish that the strands of $J$ cannot end on itself. Further, not all of the strands of $J$ can be on the top boundary (respectively, the bottom boundary) since $|w|<2n$. 

    Thus, the only case to consider is when there is at least one strand of $J$ ending at the top boundary and at least one strand ending at the bottom boundary. In this case, $J$ partitions $D$ with a left-hand side and a right-hand side. If the leftmost strand ending on the top boundary is red, then the leftmost strand ending on the bottom boundary is blue. Therefore, the left-hand side of $D$ can be considered as a morphism between two words $w_1, w_2$ in $\Dn$ in standard form, so by Lemma \ref{lemma:zero hom}, the morphism is zero. Hence, the entire diagram, $D$, is zero.
\end{proof}

For the rest of this section, each orange, green or purple dot in a diagram may represent either a red or a blue dot. Purple and green dots appearing in the same diagram always represent different color. 
We say that a jellyfish has \emph{winding} if one or more strands wind around it, as pictured below:
\begin{align*}
    \cwjellywind \qquad \ccwjellywind \lowperiod
\end{align*}

\begin{lemma}\label{jellys: no winding}
 Every diagram in $\Dx$ is, up to scalar, a diagram in which
no jellyfish has winding. 
\end{lemma}

\begin{proof}
    Suppose that \(k\) strands wind around a given jellyfish, where $k$ is even. By using relation \ref{rel:redbluejelly} $k/2$ times, we have the following equalities:

    \begin{align*}
        \ccwjellywind = \xi^{k/2} \pjellydown \qquad \text{ and } \qquad
        \cwjellywind = \xi^{-k/2} \pjellydown.
    \end{align*}
    When $k$ is odd, one can unwind the first $k-1$ strands using the same strategy as the $k$ even case. The remaining strand can then be unwound using \ref{eq: blue up} or \ref{eq: red up}, according to its color. Repeating this procedure for every jellyfish proves the result. \end{proof}

\begin{remark}\label{jellys: orientation change}
    It follows from Relation \ref{rel:closedjellys} and Lemma \ref{lem: brtwojellys} that $$\downtoup=\pjellyup \qquad \text{and} \qquad \uptodown = \pjellydown.$$
   Consequently, by adding the appropriate  caps or cups via Relations \ref{rel:rcupcap} and \ref{rel:bcupcap}, we can reverse the orientation of any jellyfish from upward-facing to downward-facing, or vice-versa.
\end{remark}

We thank Chumeng Di for the proof to the next lemma. 

    \begin{lemma}\label{lemma: twooppositefacingjellysreduce}
        Let $D$ be a diagram in $\Dx$. Let $J_1,\ J_2$ be any pair of jellyfish in $D$ that are connected by a strand, $L.$ Then $D$ is equal, up to a scalar, to a diagram where $L$ does not have any bending (i.e., not a composition of cup and caps), as shown in the picture below,
    \begin{equation}\label{eq:connectedjelly}\vcenter{\hbox{\begin{tikzpicture}[scale=1]
    \draw[thick,name path=arch] (0,0) arc (180:0:.5cm);
    \path[name path=probe1] (.4,0) -- (.4,.7);
    \path[name intersections={of=arch and probe1, by=Q1}];
    \draw[very thick,blue] (.4,-.2) -- (Q1);
    \draw[thick,name path=arch2] (.8,-.2) arc (0:-180:.5cm);
    \path[name path=probe2] (.4,-.2) -- (.4,-.7);
    \path[name intersections={of=arch2 and probe2, by=Q2}];
    \draw[very thick,blue] (Q2)-- (.4,-.2);
    \foreach \z in {.2,.7}
    \node at (\z,0.1) {$\scriptstyle\cdots$};
    \foreach \z in {.1,.6}
    \node at (\z,-.3) {$\scriptstyle\cdots$};
\end{tikzpicture}
}} \end{equation}
    where $L$ is denoted by the blue strand. We omit red and blue dots in our diagram because each jellyfish could start with either color.
    \end{lemma}

    \begin{proof}
        Due to Remark~\ref{jellys: orientation change}, we may assume that $J_1$ is the top one facing down and $J_2$ is the bottom one facing up. We may also assume that $L$ does not have any zigzags. Then the following (or its reflection across the vertical axis) is the only possible class of configurations of the local neighborhood that exactly contains both jellyfish and $L$:
    
    \[\vcenter{\hbox{\begin{tikzpicture}[scale=0.8]
    \draw[thick,name path=arch] (0,0) arc (180:0:.5cm);
    \path[name path=probe1] (.4,0) -- (.4,.6);
    \path[name intersections={of=arch and probe1, by=Q1}];
    \draw[very thick,blue] (.4,0) -- (Q1); 
    \foreach \z in {.2, .7}
    \node at (\z,.1) {$\scriptstyle\cdots$};
    \draw[very thick,blue] (.4,0) arc(0:-180:.4cm) arc(180:0:.9cm) arc(0:-180:1.1cm) arc(180:0:1.3cm) -- (1.8,-2) arc(0:-180:.9cm) arc(180:0:.5cm);
    \draw[thick,name path=arch] (0.4,-2) arc (-180:0:.5cm);
    \path[name path=probe2] (1,-2) -- (1,-2.6);
    \path[name intersections={of=arch and probe2, by=Q2}];
    \draw[very thick,blue] (1,-2) -- (Q2); 
    \foreach \z in {.7,1.2}
    \node at (\z,-2.1) {$\scriptstyle\cdots$};
    \node at (1.6,0) {\color{blue}$\cdots$};
    \node at (.8,-.2) {$\scriptstyle J_1$};
    \node at (1.4,-1.8) {$\scriptstyle J_2$};
    \node at (2,-1) {$\scriptstyle \textcolor{blue}{L}$};
    \draw[->, thick] (-1.5,0) node[left]{$\scriptstyle m$ loops around $J_1$} -- (-1.2,0);
    \draw[->, thick] (-.7,-2) node[left]{$\scriptstyle n$ loops around $J_2$} -- (-.4,-2);
\end{tikzpicture}
}} .\]
Thus, we have
\[\vcenter{\hbox{\begin{tikzpicture}[scale=0.6]
    \draw[thick,name path=arch] (0,0) arc (180:0:.65cm) -- (1.3,-.6);
    \path[name path=probe1] (.5,0) -- (.5,.7);
    \path[name intersections={of=arch and probe1, by=Q1}];
    \draw[thick] (.5,0) -- (Q1);
    \path[name path=probe2] (.8,0) -- (.8,.7);
    \path[name intersections={of=arch and probe2, by=Q2}];
    \draw[very thick,blue] (Q2)-- (.8,0) arc(0:-180:.6cm) arc(180:0:1.05cm) --  (1.7,-.6); 
    \foreach \z in {.25,1.05}
    \node at (\z,0) {$\scriptstyle\cdots$};
    \fill[orange, fill opacity=0.15, rounded corners=3pt]
    (-.2,-.2) rectangle (.6,.2);
\end{tikzpicture}
}} 
=
\vcenter{\hbox{
\begin{tikzpicture}[scale=.5]
    \fill[orange, fill opacity=0.15, rounded corners=3pt]
        (-1.6,-.2) rectangle (-.9,.2);
    \draw[thick,name path=arch] (0,0) arc (180:0:.65cm) -- (1.3,-1.1);
    \draw[thick] (0,0) arc (0:-180:.225cm) arc (180:0:1.1cm) arc (180:360:.475cm) arc (0:180:2.05cm);
    \path[name path=probe1] (.25,0) -- (.25,.7);
    \path[name intersections={of=arch and probe1, by=Q1}];
    \draw[thick] (.25,0) -- (Q1);
    \draw[thick] (.25,0) arc (0:-180:.475cm) arc (180:0:1.35cm) arc (180:360:.225cm) arc (0:180:1.8cm);
    \path[name path=probe2] (.5,0) -- (.5,.7);
    \path[name intersections={of=arch and probe2, by=Q2}];
    \draw[very thick,blue] (Q2) -- (.5,0) arc (0:-180:1.15cm) arc (180:0:2.45cm) -- (3.1,-1.1);
    \foreach \z in {-1.275,-.575,.125,.9}
        \node at (\z,0) {$\scriptscriptstyle\cdots$};
\end{tikzpicture}%
}}
=
\vcenter{\hbox{%
\begin{tikzpicture}[scale=.5]
    \fill[orange, fill opacity=0.15, rounded corners=3pt]
        (2.55,-.2) rectangle (3.2,.2);
    \draw[thick,name path=arch] (0,0) arc (180:0:.65cm) -- (1.3,-1);
    \draw[thick] (0,0) arc (0:-180:.25cm) arc (180:0:1.15cm) arc (180:360:1cm) arc (0:180:.55cm);
    \path[name path=probe1] (.35,0) -- (.35,.7);
    \path[name intersections={of=arch and probe1, by=Q1}];
    \draw[thick] (.35,0) -- (Q1);
    \draw[thick] (.35,0) arc (0:-180:.6cm) arc (180:0:1.5cm) arc (180:360:.65cm) arc (0:180:.2cm);
    \path[name path=probe2] (.6,0) -- (.6,.7);
    \path[name intersections={of=arch and probe2, by=Q2}];
    \draw[very thick,blue] (Q2) -- (.6,0) arc (0:-180:.85cm) arc (180:0:2.6cm) -- (4.1,-1);
    \foreach \z in {-.68,.18,.95}
        \node at (\z,0) {$\scriptscriptstyle\cdots$};
    \node at (2.88,0) {$\scriptscriptstyle\cdots$};
\end{tikzpicture}%
}}
= \lambda\
\vcenter{\hbox{\begin{tikzpicture}[scale=0.7]
    \draw[thick,name path=arch](0,-.8)-- (0,0) arc (180:0:1.3cm);
    \path[name path=probe1] (.4,0) -- (.4,1.5);
    \path[name intersections={of=arch and probe1, by=Q1}];
    \draw[thick] (.4,-.8) -- (Q1);
    \path[name path=probe2] (.65,0) -- (.65,1.5);
    \path[name intersections={of=arch and probe2, by=Q2}];
    \draw[thick] (Q2)-- (.65,0) arc(-180:0:.8cm) arc(0:180:.55cm); 
    \path[name path=probe3] (.95,0) -- (.95,1.5);
    \path[name intersections={of=arch and probe3, by=Q3}];
    \draw[thick] (Q3)-- (.95,0) arc(-180:0:.5cm) arc(0:180:.25cm);
    \foreach \z in {.2,.8,1.3}
    \node at (\z,0) {$\scriptstyle\cdots$};
    \fill[orange, fill opacity=0.15, rounded corners=3pt]
    (1.05,-.15) rectangle (1.55,.15);
    \draw[very thick,blue] (Q3) arc ({acos((-1.3+.95)/1.3)}:0:1.3cm) -- (2.6,-.8);
\end{tikzpicture}
}} 
\]
for some $\lambda$, where the orange shaded region could be connected to anything that is part of the diagram $D$, and the final equality comes from Lemma \ref{jellys: no winding}. We may repeatedly apply this to our diagram until we obtain the desired configuration.
\end{proof}

\begin{lemma}\label{jellyfish: eliminated}
Let $D$ be a diagram with two jellyfish connected by a strand inside  a  local neighborhood $N$ with no other jellyfish inside $N$. Then $D$ is equal, up to a nonzero scalar, to a diagram $D'$ such that $D$ and $D'$ agree  outside $N$, while $D'$ contains no jellyfish inside $N$.
\end{lemma}
\begin{proof}
  By Lemma \ref{lemma: twooppositefacingjellysreduce}, it is enough to consider jellyfish $J_1$ and $J_2$ in the form of \eqref{eq:connectedjelly} in $D$, where $J_1$ is the top jellyfish and $J_2$ is the bottom jellyfish. Assume that the $k$th strand of $J_1$ is connected to the $\ell$th strand of $J_2$, where we count from left to right. Without loss of generality, let $k \leq \ell$. All equivalences of diagrams in this proof are understood up to scalar. 
    
   We first transform 	\eqref{eq:connectedjelly} to 
    \begin{equation}\label{eq:connectedjellytwo}
      \twojellysonestrandstep\lowperiod
\end{equation}
Without loss of generality, suppose that the $k$th strand of $J_1$ and the $\ell$th strand of $J_2$ are the rightmost connected strands of $J_1$ and $J_2$. We show by induction that for all $\ell'\in \mathbb{Z}$ such that $\ell\leq \ell'\leq 2n$, $D$ is equivalent to a diagram $D'$ where outside $N$, $D$ and $D'$ are identical, and inside $N$,  the $\ell$ through $\ell'$th strands of $J_2$ are connected to the $k$ through $(\ell'+1-\ell+k)$th strands of $J_1$. The base case is true by hypothesis. Suppose the claim is true up to some $\ell'$. The $(\ell'+1)$th strand of $J_2$ is blocked on the left by the connected strands of $J_1$ and $J_2$. Similarly, the $(\ell'+1-\ell+k)$th strand of $J_1$ is blocked on its left. Thus, in a local neighborhood of $N$, the $(\ell'+1)$st strand of $J_2$ and the $(\ell'-\ell+k)$th strand of $J_1$ looks like 
$        \gcupcap\lowcomma$
   \noindent where the green dot represents either red or blue. Using either relation \ref{rel:rcupcap} or \ref{rel:bcupcap}, $D$ will be equivalent to $D'$ where outside $N$ they are equal, and inside $N$, the $\ell$ through $\ell'+1$ strands of $J_2$ are connected to the $k$ through $(\ell'+1-\ell+k)$th strands of $J_1$. 

    A nearly identical argument to the above proves that $D'$ is equivalent to a diagram $D''$ where in $N$ all the $k-1$ strands to the left of the $k$th strand of $J_1$ are connected to the $(\ell-k+1)$th to $(\ell-1)$th strands of $J_2$, and outside $N$, $D'$ and $D''$ are identical.
   So inside $N$, $D''$ looks like figure \eqref{eq:connectedjellytwo}, where there are $\ell-k$ unconnected strands on both the left and right, light blue and pink dots are opposite colors in red or blue, and the orange dots are either red or blue depending on the parity of $k$. 
    
    Notice that $J_1$ and $J_2$ both begin in the same color dot if and only if $\ell-k$ is even. Using either relation \ref{rel:redbluejelly} or the same relation with opposite colors, we can move any even number of strands from the left of $J_1$ to the right of $J_1$ for a power of $\xi$. Thus, when $J_1$ and $J_2$ begin in the same color dot, $D''$ will equal a diagram $D'''$, where outside $N$ $D''$ and $D'''$ are identical, and inside $N$, $D'''$ has $J_1$ and $J_2$ connected by all their strands, with $\ell-k$ straight strands on the right.  
    
   In the case where $J_1$ and $J_2$ begin in different color dots, $D''$ equals a diagram $D'''$, where outside $N$, $D''$ and $D'''$ are equal, and inside $N$, $D'''$ has $J_1$ and $J_2$ connected by all their strands except the leftmost of $J_2$ and the rightmost of $J_1$ with $\ell-k-1$ straight strands on the left, as
$        \twojellysallconnectedstep \lowperiod$
 
    Then, using Definition \ref{def: blue jelly}, convert the jellyfish beginning in blue to a jellyfish beginning in red. When $J_1$ began in blue, use relation \ref{rel:rcupcap} then \ref{rel:redbluejelly}. When $J_2$ began in blue, use relation \ref{rel:redbluejelly} then \ref{rel:redcircle}. In either case, we obtain that $D'''$ is a scalar multiple of a diagram $D''''$ where outside $N$, $D'''$ and $D''''$ are identical, and inside $N$, $D''''$ has $J_1$ and $J_2$ connected by all their strands with  $\ell-k$ to the right. 
 By using relation \ref{rel:twojellys}, we obtain the desired result.
\end{proof}

We now combine the preceding results into a ``Jellyfish elimination algorithm". It shows that every nonzero endomorphism of a standard form
object is a scalar multiple of the identity.

\begin{algorithm}[Jellyfish Elimination Algorithm]\label{algorithm} Let \(w\) be a word in standard form and let
\(
D\in\End_{\Dx}(w)
\)
be a nonzero diagram between two objects in standard form. 
{\leavevmode
     \begin{enumerate}
        \item If there are no jellyfish in $D$, skip to step (6).
        \item If there is a jellyfish, pick one and label it $J_1$.
        \item By Lemma \ref{lemma: jellysareconnected}, $J_1$ must have at least one of its strands connecting it to another jellyfish in $D$. Pick one such strand and label the other jellyfish $J_2$.
        \item Apply Lemma \ref{jellyfish: eliminated} to eliminate $J_1$ and $J_2$ from $D$.
        \item Repeat steps (2)--(4) until zero jellyfish remain.
        \item Remove all closed red and blue circles using
 \textnormal{ \ref{rel:redcircle} } and \textnormal{ \ref{rel:bluecircle}}. Since $w$ is in standard form, there can be no cups and caps, as adjacent points have different
    colors. The remaining diagram is therefore a  scalar multiple of $\text{id}_w$. 
    \end{enumerate}}
\end{algorithm}

\begin{corollary}\label{cor: dim of Homs}
Let $x,y$ be any objects in $\Dx$. Then
\[
\dim \Hom_{\Dx}(x,y)
=
\begin{cases}
1 & \overline{x}=\overline{y},\\
0 & \overline{x}\neq\overline{y}.
\end{cases}
\] Equivalently, \(\Hom_{\Dx}(x,y)\) is one-dimensional if \(x\cong y\)
and is zero otherwise.
\end{corollary}
\begin{proof}
By Remark \ref{rk: std form}, it is enough to assume that $x$ and $y$ are in standard form. If \(x\neq y\), then Lemma~\ref{lemma:zero hom}  gives
\(
\Hom_{\Dx}(x,y)=0.
\)
Now suppose that \(x=y=w\). By Algorithm~\ref{algorithm}, every
diagram in \(\End_{\Dx}(w)\) is a scalar multiple of
\(\mathrm{id}_w\), and so \(\End_{\Dx}(w)\) is spanned by \(\mathrm{id}_w\).
Moreover, \(\mathrm{id}_w\neq0\), since
\(
\Phi(\mathrm{id}_w)=\mathrm{id}_{\Phi(w)}\neq0.
\)
Therefore,
\(
\dim_k\End_{\Dx}(w)=1.
\)
\end{proof}

\begin{proof}[Proof of Theorem \ref{thm diag category}]
Every object of \(\FGR\) is represented by a word in the generating
objects, so \(\Phi\) is essentially surjective. Moreover, every
morphism of \(\FGR\) is a linear combination of diagrams built from the
generating morphisms. Since all these generators lie in the image of
\(\Phi\), the functor \(\Phi\) is full. 
It remains to prove faithfulness, which follows from Corollary \ref{cor: dim of Homs}. Hence $\Phi$ is an equivalence. 
Applying the additive Karoubi envelope gives
\(
\mathcal K(\Dx)
\simeq
\KFGR
\simeq
\VecDn,
\)
as desired. 
\end{proof}

\section{Applications}\label{sec: applications}

 We give three applications of our diagrammatics. First, we establish diagrammatic presentations for all pivotal structures on $\VecDnx$. Second, we characterize tensor
functors from $\VecDnx $ to an arbitrary tensor category. Third, when $ \epsilon_r=\epsilon_b=1$, we identify $\mathrm{Vec}_{\mathrm{D}_{2n}}^{1,1,\xi}$  with the fusion category
arising from an affine  $A$ subfactor
planar algebra.

\subsection{Pivotal Structures of $\VecDnx$}
\label{sec: pivotal}

Fix an admissible triple
$(\epsilon_r,\epsilon_b,\xi)$.
Recall that every group homomorphism
$\pi\colon \Dn\longrightarrow\mathsf{k}^{\times}$
determines a pivotal structure on $\VecDnx$. We denote the resulting pivotal category by  $(\VecDnx,\pi)$.

\begin{remark}
    Every pivotal structure on $\VecDn$ is spherical. In fact, let $d_R, d_L:\Dn \to \mathsf k^{\times}$ denote the right and left pivotal dimensions, respectively. Then $d_R(r)^2=d_R(r^2)=1$, and $d_R(b)^2=d_R(b^2)=1$, so $d_R(g)=\pm 1$ for all $g\in \Dn$. Hence
$d_L(g)=d_R(g^{\vee})=d_R(g^{-1})=d_R(g),$ 
    for all $g\in \Dn.$ 
\end{remark}

We extend our diagrammatic presentation to record the pivotal structure determined by $\pi$, using the general pivotal presentation of pointed fusion categories established in \cite[Theorem 1.1]{dp}.

Below,  $ \vcenter{\hbox{\begin{tikzpicture}
    \filldraw[r] (0,0) circle (.05cm);
\end{tikzpicture}}}$ 
represents $r \in (\VecDnx,\pi)$ and 
$\vcenter{\hbox{\begin{tikzpicture}
    \draw[thick,r] (0,0) circle (.05cm);
\end{tikzpicture}}}$ represents $r^\vee$; $\vcenter{\hbox{\begin{tikzpicture}
    \filldraw[b] (0,0) circle (.05cm);
\end{tikzpicture}}}$
represents $b$ and
$\vcenter{\hbox{\begin{tikzpicture}
    \draw[thick,b] (0,0) circle (.05cm);
\end{tikzpicture}}}$ 
represents $b^\vee.$

\begin{theorem}\label{thm: pivotal presentation}
    The following is a presentation for $(\VecDnx,\pi)$:

    Generating objects: 
\begin{enumerate}[label=\textup{(O\arabic*)}]
\item
\label{Obj:1}
$ \vcenter{\hbox{\begin{tikzpicture}
    \filldraw[r] (0,0) circle (.05cm);
\end{tikzpicture}}}$,
\hspace{1cm}
$\vcenter{\hbox{\begin{tikzpicture}
    \draw[thick,r] (0,0) circle (.05cm);
\end{tikzpicture}}}$,
\hspace{1cm}
$\vcenter{\hbox{\begin{tikzpicture}
    \filldraw[b] (0,0) circle (.05cm);
\end{tikzpicture}}}$, 
\hspace{1cm}
$\vcenter{\hbox{\begin{tikzpicture}
    \draw[thick,b] (0,0) circle (.05cm);
\end{tikzpicture}}}$
\end{enumerate}

Generating morphisms:

\begin{enumerate}[label=\textup{(M\arabic*)}]
\item
\label{Mor:1}
$\vcenter{\hbox{\begin{tikzpicture}
    \draw[thick] (0,0) -- (0,.6);
    \filldraw[r] (0,0) circle (.05cm) ;
    \filldraw[r] (0,.6) circle (.05cm) ;
\end{tikzpicture}}}$\ ,\hspace{.3cm}
$\vcenter{\hbox{\begin{tikzpicture}
    \draw[thick] (0,0) -- (0,.6);
    \filldraw[b] (0,0) circle (.05cm) ;
    \filldraw[b] (0,.6) circle (.05cm) ;
\end{tikzpicture}}}$\ ,\hspace{.3cm}
$\vcenter{\hbox{\begin{tikzpicture}
    \draw[thick] (0,0) arc (180:0:.4cm);
    \filldraw[r] (0,0) circle (.05cm) ;
    \filldraw[r] (.8,0) circle (.05cm) ;
\end{tikzpicture}}}$\ ,\hspace{.3cm}
$\vcenter{\hbox{\begin{tikzpicture}
    \draw[thick] (0,0) arc (180:0:.4cm);
    \filldraw[b] (0,0) circle (.05cm) ;
    \filldraw[b] (.8,0) circle (.05cm) ;
\end{tikzpicture}}}$\ ,\hspace{.3cm}
$\vcenter{\hbox{\begin{tikzpicture}
    \draw[thick] (0,0) arc (-180:0:.4cm);
    \filldraw[r] (0,0) circle (.05cm) ;
    \filldraw[r] (.8,0) circle (.05cm) ;
\end{tikzpicture}}}$\ ,\hspace{.3cm}
$\vcenter{\hbox{\begin{tikzpicture}
    \draw[thick] (0,0) arc (-180:0:.4cm);
    \filldraw[b] (0,0) circle (.05cm) ;
    \filldraw[b] (.8,0) circle (.05cm) ;
\end{tikzpicture}}}$\ ,\hspace{.3cm}
$\vcenter{\hbox{\begin{tikzpicture}
    \draw[thick,name path = arch] (0,0) arc (180:0:.55cm);
    \path[name path=probe1] (.3,0) -- (.3,.5);
    \path[name intersections={of=arch and probe1, by=Q1}];
    \draw[thick] (.3,0) -- (Q1); 
    \path[name path=probe2] (.8,0) -- (.8,.5);
    \path[name intersections={of=arch and probe2, by=Q2}];
    \draw[thick] (.8,0) -- (Q2); 
    \filldraw[r] (0,0) circle (.05cm) ;
    \filldraw[r] (.8,0) circle (.05cm) ;
    \filldraw[b] (0.3,0) circle (.05cm) ;
    \filldraw[b] (1.1,0) circle (.05cm) ;
    \node at (.55,.1) {$\scriptstyle\cdots$};
\end{tikzpicture}}}$\ ,\hspace{.3cm}
$\vcenter{\hbox{\begin{tikzpicture}
    \draw[thick,name path = arch] (0,0) arc (-180:0:.55cm);
    \path[name path=probe1] (.3,0) -- (.3,-.5);
    \path[name intersections={of=arch and probe1, by=Q1}];
    \draw[thick] (.3,0) -- (Q1); 
    \path[name path=probe2] (.8,0) -- (.8,-.5);
    \path[name intersections={of=arch and probe2, by=Q2}];
    \draw[thick] (.8,0) -- (Q2); 
    \filldraw[r] (0,0) circle (.05cm) ;
    \filldraw[r] (.8,0) circle (.05cm) ;
    \filldraw[b] (0.3,0) circle (.05cm) ;
    \filldraw[b] (1.1,0) circle (.05cm) ;
    \node at (.55,-.1) {$\scriptstyle\cdots$};
\end{tikzpicture}}}$.
\vspace{.2cm}

\item 
\label{Mor:3}
$
 \vcenter{\hbox{\begin{tikzpicture}
    \draw[thick] (0,0) arc (180:0:.4cm);
    \filldraw[thick,r,fill=white] (0,0) circle (.05cm) ;
    \filldraw[r] (.8,0) circle (.05cm) ;
\end{tikzpicture}}}\ , \hspace{.2cm}
\vcenter{\hbox{\begin{tikzpicture}
    \draw[thick] (0,0) arc (180:0:.4cm);
    \filldraw[thick,r,fill=white] (0.8,0) circle (.05cm) ;
    \filldraw[r] (0,0) circle (.05cm) ;
\end{tikzpicture}}}\ , \hspace{.2cm}
\vcenter{\hbox{\begin{tikzpicture}
    \draw[thick] (0,0) arc (-180:0:.4cm);
    \filldraw[r] (0,0) circle (.05cm) ;
    \filldraw[thick,r,fill=white] (.8,0) circle (.05cm) ;
\end{tikzpicture}}}\ , \hspace{.2cm}
\vcenter{\hbox{\begin{tikzpicture}
    \draw[thick] (0,0) arc (-180:0:.4cm);
    \filldraw[thick,r,fill=white] (0,0) circle (.05cm) ;
    \filldraw[r] (.8,0) circle (.05cm) ;
\end{tikzpicture}}},\
$ and
$\vcenter{\hbox{\begin{tikzpicture}
    \draw[thick] (0,0) arc (180:0:.4cm);
    \filldraw[thick,b,fill=white] (0,0) circle (.05cm) ;
    \filldraw[b] (.8,0) circle (.05cm) ;
\end{tikzpicture}}}\ , \hspace{.2cm}
\vcenter{\hbox{\begin{tikzpicture}
    \draw[thick] (0,0) arc (180:0:.4cm);
    \filldraw[thick,b,fill=white] (0.8,0) circle (.05cm) ;
    \filldraw[b] (0,0) circle (.05cm) ;
\end{tikzpicture}}}\ , \hspace{.2cm}
\vcenter{\hbox{\begin{tikzpicture}
    \draw[thick] (0,0) arc (-180:0:.4cm);
    \filldraw[b] (0,0) circle (.05cm) ;
    \filldraw[thick,b,fill=white] (.8,0) circle (.05cm) ;
\end{tikzpicture}}}\ , \hspace{.2cm}
\vcenter{\hbox{\begin{tikzpicture}
    \draw[thick] (0,0) arc (-180:0:.4cm);
    \filldraw[thick,b,fill=white] (0,0) circle (.05cm) ;
    \filldraw[b] (.8,0) circle (.05cm) ;
\end{tikzpicture}}}
$
\end{enumerate}
\vspace{.2cm}

Relations:

\vspace{.2cm}
\noindent \  \textnormal{\ref{rel:redstrandcap} -- \ref{rel:closedjellys}} from Definition \ref{tensor presentaion}, the relations in the presentation of $\VecDnx$.

\begin{enumerate}[label=\textup{(R\arabic*)}]\addtocounter{enumi}{+10}
    \item 
    \label{piv zigzag}
$\vcenter{\hbox{\begin{tikzpicture}
     \draw[thick] (0,-.3) -- (0,0) arc (180:0:.3cm) -- (.6,0) arc(-180:0:.3cm) -- (1.2,.3);
    \filldraw[thick,r,fill=white] (0,-.3) circle (.05cm) ;
    \filldraw[thick,r,fill=white] (1.2,.3) circle (.05cm) ;
    \filldraw[r] (.6,0) circle (.05cm) ;
\end{tikzpicture}}} = 
\vcenter{\hbox{\begin{tikzpicture}
     \draw[thick] (0,-.3) -- (0,.3);
    \filldraw[thick,r,fill=white] (0,-.3) circle (.05cm) ;
    \filldraw[thick,r,fill=white] (0,.3) circle (.05cm) ;
\end{tikzpicture}}}\ , \hspace{.2cm}
$
$\vcenter{\hbox{\begin{tikzpicture}
     \draw[thick] (0,.3) -- (0,0) arc (-180:0:.3cm) -- (.6,0) arc(180:0:.3cm) -- (1.2,-.3);
    \filldraw[r] (1.2,-.3) circle (.05cm) ;
    \filldraw[r] (0,.3) circle (.05cm) ;
    \filldraw[thick,r,fill=white] (.6,0) circle (.05cm) ;
\end{tikzpicture}}} = 
\vcenter{\hbox{\begin{tikzpicture}
     \draw[thick] (0,-.3) -- (0,.3);
    \filldraw[r] (0,-.3) circle (.05cm) ;
    \filldraw[r] (0,.3) circle (.05cm) ;
\end{tikzpicture}}}\ , \hspace{.2cm}
$
$\vcenter{\hbox{\begin{tikzpicture}
     \draw[thick] (0,-.3) -- (0,0) arc (180:0:.3cm) -- (.6,0) arc(-180:0:.3cm) -- (1.2,.3);
    \filldraw[thick,r] (0,-.3) circle (.05cm) ;
    \filldraw[thick,r] (1.2,.3) circle (.05cm) ;
    \filldraw[thick,r,fill=white] (.6,0) circle (.05cm) ;
\end{tikzpicture}}} = 
\vcenter{\hbox{\begin{tikzpicture}
     \draw[thick] (0,-.3) -- (0,.3);
    \filldraw[r] (0,-.3) circle (.05cm) ;
    \filldraw[r] (0,.3) circle (.05cm) ;
\end{tikzpicture}}}\ , \hspace{.2cm}
$
$\vcenter{\hbox{\begin{tikzpicture}
     \draw[thick] (0,.3) -- (0,0) arc (-180:0:.3cm) -- (.6,0) arc(180:0:.3cm) -- (1.2,-.3);
    \filldraw[thick,r,fill=white] (1.2,-.3) circle (.05cm) ;
    \filldraw[thick,r,fill=white] (0,.3) circle (.05cm) ;
    \filldraw[r] (.6,0) circle (.05cm) ;
\end{tikzpicture}}} = 
\vcenter{\hbox{\begin{tikzpicture}
     \draw[thick] (0,-.3) -- (0,.3);
    \filldraw[thick,r,fill=white] (0,-.3) circle (.05cm) ;
    \filldraw[thick,r,fill=white] (0,.3) circle (.05cm) ;
\end{tikzpicture}}}\ , 
$
\vspace{.2cm}

and
$\vcenter{\hbox{\begin{tikzpicture}
     \draw[thick] (0,-.3) -- (0,0) arc (180:0:.3cm) -- (.6,0) arc(-180:0:.3cm) -- (1.2,.3);
    \filldraw[thick,b,fill=white] (0,-.3) circle (.05cm) ;
    \filldraw[thick,b,fill=white] (1.2,.3) circle (.05cm) ;
    \filldraw[b] (.6,0) circle (.05cm) ;
\end{tikzpicture}}} = 
\vcenter{\hbox{\begin{tikzpicture}
     \draw[thick] (0,-.3) -- (0,.3);
    \filldraw[thick,b,fill=white] (0,-.3) circle (.05cm) ;
    \filldraw[thick,b,fill=white] (0,.3) circle (.05cm) ;
\end{tikzpicture}}}\ , \hspace{.2cm}
$
$\vcenter{\hbox{\begin{tikzpicture}
     \draw[thick] (0,.3) -- (0,0) arc (-180:0:.3cm) -- (.6,0) arc(180:0:.3cm) -- (1.2,-.3);
    \filldraw[b] (1.2,-.3) circle (.05cm) ;
    \filldraw[b] (0,.3) circle (.05cm) ;
    \filldraw[thick,b,fill=white] (.6,0) circle (.05cm) ;
\end{tikzpicture}}} = 
\vcenter{\hbox{\begin{tikzpicture}
     \draw[thick] (0,-.3) -- (0,.3);
    \filldraw[b] (0,-.3) circle (.05cm) ;
    \filldraw[b] (0,.3) circle (.05cm) ;
\end{tikzpicture}}}\ , \hspace{.2cm}
$
$\vcenter{\hbox{\begin{tikzpicture}
     \draw[thick] (0,-.3) -- (0,0) arc (180:0:.3cm) -- (.6,0) arc(-180:0:.3cm) -- (1.2,.3);
    \filldraw[thick,b] (0,-.3) circle (.05cm) ;
    \filldraw[thick,b] (1.2,.3) circle (.05cm) ;
    \filldraw[thick,b,fill=white] (.6,0) circle (.05cm) ;
\end{tikzpicture}}} = 
\vcenter{\hbox{\begin{tikzpicture}
     \draw[thick] (0,-.3) -- (0,.3);
    \filldraw[b] (0,-.3) circle (.05cm) ;
    \filldraw[b] (0,.3) circle (.05cm) ;
\end{tikzpicture}}}\ , \hspace{.2cm}
$
$\vcenter{\hbox{\begin{tikzpicture}
     \draw[thick] (0,.3) -- (0,0) arc (-180:0:.3cm) -- (.6,0) arc(180:0:.3cm) -- (1.2,-.3);
    \filldraw[thick,b,fill=white] (1.2,-.3) circle (.05cm) ;
    \filldraw[thick,b,fill=white] (0,.3) circle (.05cm) ;
    \filldraw[b] (.6,0) circle (.05cm) ;
\end{tikzpicture}}} = 
\vcenter{\hbox{\begin{tikzpicture}
     \draw[thick] (0,-.3) -- (0,.3);
    \filldraw[thick,b,fill=white] (0,-.3) circle (.05cm) ;
    \filldraw[thick,b,fill=white] (0,.3) circle (.05cm) ;
\end{tikzpicture}}}\ .
$
\vspace{.2cm}

\item 
\label{quantum dims}
$\vcenter{\hbox{\begin{tikzpicture}
    \draw[thick] (0,0) circle (.4cm);
    \filldraw[thick,r] (-.4,0) circle (.05cm) ;
    \filldraw[thick,r,fill=white] (.4,0) circle (.05cm) ;
\end{tikzpicture}}} = \pi(r)$\ ,  \hspace{.2cm}
$\vcenter{\hbox{\begin{tikzpicture}
    \draw[thick] (0,0) circle (.4cm);
    \filldraw[thick,b] (-.4,0) circle (.05cm) ;
    \filldraw[thick,b,fill=white] (.4,0) circle (.05cm) ;
\end{tikzpicture}}} = \pi(b)$\ .

\end{enumerate}
\end{theorem}

\begin{proof}
We obtain the presentation above from the general pivotal presentation in \cite[Theorem~1.1]{dp}. To distinguish the relations in that presentation from ours, we denote them by \textup{(P1)}--\textup{(P4)}.

We first consider the object-generators. 
Every element in $\Dn$ has a unique standard form $ (
\vcenter{\hbox{\begin{tikzpicture}
    \filldraw[thick,r] (0,0) circle (.05cm);
    \filldraw[thick,b] (.3,0) circle (.05cm);
\end{tikzpicture}}})^m\
\vcenter{\hbox{\begin{tikzpicture}
    \filldraw[thick,r] (0,0) circle (.05cm);
\end{tikzpicture}}}^\ell
$, where $0\leq m \leq n-1$ and $\ell \in \{0,1\};$ its dual is $ 
\vcenter{\hbox{\begin{tikzpicture}
    \draw[thick,r] (0,0) circle (.05cm);
\end{tikzpicture}}}^\ell\ (
\vcenter{\hbox{\begin{tikzpicture}
 \draw[thick,r] (.3,0) circle (.05cm);
    \draw[thick,b] (0,0) circle (.05cm);
\end{tikzpicture}}})^m
$. Thus every object-generator in the general presentation of \cite{dp} is a word in the four objects listed in \ref{Obj:1}.

We next consider the morphism-generators.  
The multiplication isomorphisms $\mu_{g,h}$ and $\mu_{g,h}^{-1}$ in \cite{dp} are represented by our $f_{g,h}^{-1}$ and $f_{g,h}$, respectively, which were constructed in Section \ref{Choices for J}
 from the morphisms in \ref{Mor:1}. The remaining morphism-generators in \cite{dp} are evaluation and coevaluation morphisms. Let  $g= (
\vcenter{\hbox{\begin{tikzpicture}
    \filldraw[thick,r] (0,0) circle (.05cm);
    \filldraw[thick,b] (.3,0) circle (.05cm);
\end{tikzpicture}}})^m\
\vcenter{\hbox{\begin{tikzpicture}
    \filldraw[thick,r] (0,0) circle (.05cm);
\end{tikzpicture}}}^\ell
$ and 
$g^{\vee}=\vcenter{\hbox{\begin{tikzpicture}
    \draw[thick,r] (0,0) circle (.05cm);
\end{tikzpicture}}}^\ell\
(\vcenter{\hbox{\begin{tikzpicture}
    \draw[thick,b] (0,0) circle (.05cm);
    \draw[thick,r] (.3,0) circle (.05cm);
\end{tikzpicture}}})^m
$. Then these are obtained by stacking the corresponding morphisms for the pairs $\{
\vcenter{\hbox{\begin{tikzpicture}
    \filldraw[thick,r] (0,0) circle (.05cm);
\end{tikzpicture}}}\ ,\
\vcenter{\hbox{\begin{tikzpicture}
    \draw[thick,r] (0,0) circle (.05cm);
\end{tikzpicture}}}
\}$ 
and
$\{
\vcenter{\hbox{\begin{tikzpicture}
    \filldraw[thick,b] (0,0) circle (.05cm);
\end{tikzpicture}}}\ ,\
\vcenter{\hbox{\begin{tikzpicture}
    \draw[thick,b] (0,0) circle (.05cm);
\end{tikzpicture}}}
\}$ listed in \textup{(M2)}.

It remains to verify the relations \textup{(P1)}--\textup{(P4)}. In Theorem \ref{thm diag category}, we showed that \textup{(P1)} is generated by relations \ref{rel:redstrandcap} -- \ref{rel:closedjellys}.
Relation \textup{(P2)} is already built into the choice of $J$ made in Section \ref{Choices for J}.

We next consider relation \textup{(P3)}, which consists of the four snake identities for each $g\in\Dn$. We illustrate the argument for the first snake identity when $|g|$ is odd; all other three snake identities follow analogously. After replacing the evaluation and coevaluation morphisms for $g$ by the stacks described above, \textup{(P3)} becomes
\[\vcenter{\hbox{\begin{tikzpicture}[scale =0.8]
   \draw[thick](.1,-1)-- (0.1,0) arc (180:0:.3cm) arc(-180:0:1cm) -- (2.7,1);
   \draw[thick] (-.2,-1)-- (-.2,0) arc (180:0:.6cm) arc(-180:0:.7cm) -- (2.4,1);
   \draw[thick] (-.6,-1) -- (-.6,0) arc (180:0:1cm) arc(-180:0:.3cm) -- (2,1);
   \node at (3.2,0) {$=$};
   \draw[thick] (3.7,-1) -- (3.7,1);
   \draw[thick] (4.1,-1) -- (4.1,1);
   \draw[thick] (4.4,-1) -- (4.4,1);
   \foreach \x in {-.6,.1,3.7,4.4}
       \filldraw[thick,r,fill=white] (\x,-1) circle (.05cm) ;
    \foreach \x in {2,2.7,3.7,4.4}
       \filldraw[thick,r,fill=white] (\x,1) circle (.05cm) ;   
    \foreach \x in {-.2,4.1}
       \filldraw[thick,b,fill=white] (\x,-1) circle (.05cm) ;
    \foreach \x in {2.4,4.1}
       \filldraw[thick,b,fill=white] (\x,1) circle (.05cm) ;  
    \filldraw[r] (.7,0) circle (.05cm) ;
   \filldraw[r] (1.4,0) circle (.05cm) ;
   \filldraw[b] (1,0) circle (.05cm) ;
   \foreach \p in {(-.4,-.8),(2.2,.8),(3.9,-.8)}
       \node at \p {$\scriptstyle\cdots$};
   \draw[decorate, decoration={brace,mirror}] (-.7,-1.1) -- (-.1,-1.1) node[midway, below=3pt] {$2m\ \text{strands}$};
\end{tikzpicture}}}\]
Then \textup{(P3)} holds by \ref{piv zigzag}. 



Finally, for $g= (
\vcenter{\hbox{\begin{tikzpicture}
    \filldraw[thick,r] (0,0) circle (.05cm);
    \filldraw[thick,b] (.3,0) circle (.05cm);
\end{tikzpicture}}})^m\
\vcenter{\hbox{\begin{tikzpicture}
    \filldraw[thick,r] (0,0) circle (.05cm);
\end{tikzpicture}}}^\ell
$, the closed loop corresponding to $g$ separates into $m+\ell$ red loops and $m$ blue loops. By \textup{(R12)}, it therefore evaluates to
$\bigl(\pi(r)\pi(b)\bigr)^m\pi(r)^\ell=\pi(g).$ 
The other closed loop evaluates to $\pi(g^{-1})$. Thus \textup{(P4)} follows from \ref{quantum dims}, completing the proof. 
\end{proof}

\subsection{Universal Property}
Let \(\mathcal C\) be a tensor category; we may assume
that \(\mathcal C\) is strict, see~\cite[Chapter~VII, Theorem~1]{MacLane}.

Let $\operatorname{MonFun}_{\mathsf k}(\Dx,\mathcal C)$ 
denote the category of \(\mathsf k\)-linear monoidal functors
\(\Dx\to\mathcal C\) and monoidal natural transformations. Let $\operatorname{Fun}_{\otimes}(\VecDnx, \mathcal C)$ denote the category  with objects tensor functors $\VecDnx \to \mathcal C$ and morphisms monoidal natural transformations. 
Restriction along
$\Dx\longrightarrow\mathcal K(\Dx)\simeq\VecDnx$
induces an equivalence
\[
\operatorname{Fun}_{\otimes}(\VecDnx,\mathcal C)
\simeq
\operatorname{MonFun}_{\mathsf k}(\Dx,\mathcal C).
\]
A strict \(\mathsf k\)-linear monoidal functor
$\Dx\longrightarrow\mathcal C$ is determined by the images of the 
generating objects and  morphisms.
Hence there is a bijection
\begin{equation}\label{bijection}
\left\{
\begin{array}{c}
\text{strict \(\mathsf k\)-linear monoidal functors}\\
F\colon\Dx\to\mathcal C
\end{array}
\right\}
\longleftrightarrow
\left\{
\text{tuples }(X,Y,\lambda_X,\lambda_Y,\rho)
\right\},
\end{equation}
where
\begin{enumerate}
    \item\label{(1)} $X,Y$ are objects in $\mathcal C$, and
    \item\label{(2)} $\lambda_X:X^{\otimes 2} \to 1_{\mathcal C}$, $\lambda_Y:Y^{\otimes 2} \to 1_{\mathcal C}$ and $\rho :(X\otimes Y)^{\otimes n} \to 1_{\mathcal C}$ are isomorphisms satisfying
    {\small
    \begin{align*}
    \text{id}_X\otimes \lambda_X= \ \epsilon_r \  (\lambda_X\otimes \text{id}_X), \ \ \ \ \ \ 
    \text{id}_Y\otimes \lambda_Y= \ \epsilon_b \ (\lambda_Y\otimes \text{id}_Y), \ \ \ \ \ \ 
    \text{id}_{X\otimes Y}\otimes \rho=\ \xi \ (\rho\otimes \text{id}_{X\otimes Y}),\\
   \lambda_X  (\text{id}_X\otimes \lambda_Y\otimes \text{id}_X)\dots 
   (\text{id}_X\otimes \dots \otimes \text{id}_Y \otimes \lambda_X \otimes \text{id}_Y\otimes \dots \otimes \text{id}_X)( \rho^{-1} \otimes \text{id}_X\otimes \rho^{-1} \otimes \text{id}_X)\lambda_X^{-1}=\text{id}_{1_{\mathcal C}}.
    \end{align*}}
\end{enumerate}
Under this bijection, \(F\) sends the red and blue generating objects
to \(X\) and \(Y\), respectively, and sends the generating morphisms as
follows:
   \begin{align*}
&	\redstrand \longmapsto \text{id}_X,\;
	\bluestrand \longmapsto \text{id}_Y,\;
     \redcap \longmapsto \lambda_X,\;
	\bluecap \longmapsto \lambda_Y,\;
	\redcup \longmapsto \lambda_X^{-1},\;
	\bluecup \longmapsto \lambda_Y^{-1},\;
   \\
	& \  \ \ \ \ \ \ \ \ \ \ \ 
    \ \ \ \ \ \ \ \ \ \ \ 
    \ \ \ \ \ \ \
\rbjellydown \longmapsto \rho,    \;
	\ \ \ \ \ \ \ \ \rbjellyup  \longmapsto \rho^{-1}\;.
\end{align*}
Given such a tuple $(X, Y, \lambda_X, \lambda_Y, \rho),$ we will refer to the corresponding functor by $F_{(X, Y, \lambda_X, \lambda_Y, \rho)}$.  We use throughout the convention that  $F(1)=1_{\mathcal C}$.

We will now upgrade this bijection to an equivalence of categories. 
\begin{definition}
  Let $\mathcal F_{\mathcal C}$ be the category with
  \begin{itemize}
      \item Objects are tuples $(X,Y, \lambda_X, \lambda_Y, \rho)$ as in \eqref{(1)} and \eqref{(2)}.
      \item A morphism $(X, Y, \lambda_X, \lambda_Y, \rho)\to (X', Y', \lambda_{X'}, \lambda_{Y'}, \rho')$ is a pair $(f,g)$ where  $f: X\to X'$ and $g:Y \to Y'$ are morphisms in $\mathcal C$ such that 
      \begin{align*}
          &\lambda_{X'}f^{\otimes 2} = \lambda_X, &&\lambda_{Y'}g^{\otimes 2}= \lambda_Y &&\text{and} &&\rho'(f\otimes g)^{\otimes n} =\rho.
      \end{align*}
  \end{itemize}
\end{definition}

\begin{theorem}[Universal Property for $\VecDnx$]\label{thm: univ prop} The categories $\operatorname{Fun}_{\otimes}(\VecDnx, \mathcal C)$ and $\mathcal F_{\mathcal C}$ are equivalent.     
\end{theorem}

\begin{proof}
Every $\mathsf k$-linear 
monoidal functor \(\Dx\to\mathcal C\) is monoidally naturally isomorphic to a
strict monoidal functor. Let
$\operatorname{MonFun}^{\mathrm{str}}_{\mathsf k}(\Dx,\mathcal C)$ 
denote the full subcategory of
 strict monoidal
functors.  It suffices to prove
that
$\operatorname{MonFun}^{\mathrm{str}}_{\mathsf k}(\Dx,\mathcal C)
\simeq
\mathcal F_{\mathcal C}.$

    We define  $\Psi\colon
\operatorname{MonFun}^{\mathrm{str}}_{\mathsf k}(\Dx,\mathcal C)
\longrightarrow
\mathcal F_{\mathcal C}$ by sending each functor to its tuple under~\eqref{bijection}, and each monoidal natural transformation $\tau$ to $(\tau_r, \tau_b)$. For two functors corresponding to  $(X,Y,\lambda_X,\lambda_Y,\rho)$ and $(X',Y',\lambda_{X'},\lambda_{Y'},\rho')$, 
naturality and monoidality at the caps and downward jellyfish gives
\[
\lambda_{X'}\tau_r^{\otimes2}=\lambda_X,\qquad
\lambda_{Y'}\tau_b^{\otimes2}=\lambda_Y,\qquad
\rho'(\tau_r\otimes \tau_b)^{\otimes n}=\rho.
\]Hence $\Psi$ is well-defined. We show now that $\Psi$ is an equivalence. 
It is essentially surjective by \eqref{bijection}, so all we need to show is that it is fully faithful. 
    For a morphism $(f,g): (X,Y,\lambda_X, \lambda_Y,\rho) \to (X',Y',\lambda_{X'}, \lambda_{Y'},\rho')$ in $\mathcal{F_C}$, let $\tau: F_{(X,Y,\lambda_X, \lambda_Y,\rho)}\to F_{(X',Y',\lambda_{X'}, \lambda_{Y'},\rho')}$ be the natural transformation obtained by setting $\tau_r=f,$ $\tau_b=g,$ and extending monoidally everywhere else. It is a quick check that this gives an inverse, and the statement follows.  
\end{proof}

\subsection{\texorpdfstring{Equivalence to Affine $A$ Planar Algebras}{Equivalence to Affine $A$ Planar Algebras}}

The goal of this section is to discuss how the $\VecDnx$ diagrammatics recover categories induced by affine $A$ unshaded subfactor planar algebras when $\epsilon_r=\epsilon_b=1$. While the monoidal equivalence proven in this section was known to experts, the diagrammatic proof is novel and was the impetus of this paper.

\begin{definition}[From \cite{Mol24}, Theorem 8.1]\label{PA presentaion}
   Fix $n \in \mathbb{N}$. Let $\mathcal{P}^\xi$ denote the subfactor planar algebra generated by   
$$ 	p_1:=\paredstrand \ , \qquad 
	q_1:=\pabluestrand \ , \qquad
	V:= \paV \ , \qquad 
	V^*:=\paVad \ ,$$   
where $V$ and $V^*$ have $n$ strands alternating colors on both top and bottom, subject to the relations
\setcounter{savedequation}{\value{equation}}

\begingroup
\renewcommand{\theequation}{(\textnormal{PA}\arabic{equation})}
\renewcommand{\theHequation}{PA\arabic{equation}}
\setcounter{equation}{0}

\noindent
\begin{tabular}{@{}p{0.333\linewidth}@{}p{0.333\linewidth}@{}p{0.333\linewidth}@{}}

\makebox[\linewidth][c]{%
  \refstepcounter{equation}\label{rel:paredcircle}%
  $\theequation\quad \paredcircle=\emptyset$%
}
&
\makebox[\linewidth][c]{%
  \refstepcounter{equation}\label{rel:pabluecircle}%
  $\theequation\quad \pabluecircle=\emptyset$%
}
&
\makebox[\linewidth][c]{%
  \refstepcounter{equation}\label{rel:paredbluestrand}%
  $\theequation\quad \paredbluestrand=\emptyset$%
}
\\[0.2em]

\makebox[\linewidth][c]{%
  \refstepcounter{equation}\label{rel:paredcupcap}%
  $\theequation\quad \paredcupcap=\patwored$%
}
&
\makebox[\linewidth][c]{%
  \refstepcounter{equation}\label{rel:pabluecupcap}%
  $\theequation\quad \pabluecupcap=\patwoblue$%
}
&
\makebox[\linewidth][c]{%
  \refstepcounter{equation}\label{rel:pablackstrand}%
  $\theequation\quad
  \pablackstrand=\paredstrand+\pabluestrand$%
}
\\[0.2em]

\multicolumn{3}{@{}c@{}}{%
  \makebox[0.5\linewidth][c]{%
    \refstepcounter{equation}\label{rel:paVadV}%
    $\theequation\quad \paVadV=\emptyset$%
  }%
  \makebox[0.5\linewidth][c]{%
    \refstepcounter{equation}\label{rel:paVVad}%
    $\theequation\quad \paVVad=\patworedblue$%
  }%
}
\\[0.2em]

\multicolumn{3}{@{}c@{}}{%
  \refstepcounter{equation}\label{rel:paclick}%
  $\theequation\quad
  \Vinverseclick=\Vadinclickrel=\xi\paVoneclick\lowperiod$%
}
\end{tabular}
\endgroup
\setcounter{equation}{\value{savedequation}}

\end{definition}

 In the diagrams above, green and purple strands indicate strands colored red or blue, depending on the parity of $n$, and $\emptyset$ indicates the empty diagram.

From a finite index unshaded subfactor planar algebra, $\mathcal{P}$, one can create a fusion category. First, create a strict monoidal category, $\mathcal{C}_\mathcal{P}$, whose objects are \textit{projections} in $\mathcal{P}$, i.e., elements $p\in \mathcal{P}$ where $p=p^2=p^*$. A morphism  $f\in \text{Hom}_{\mathcal{C}_\mathcal{P}}(p,q)$ is any diagram in $\mathcal{P}$ whose bottom boundary is $p$ and top boundary is $q$. The tensor product is horizontal concatenation, composition is vertical stacking, and the unit object is the empty diagram, $\emptyset$. Taking the additive envelope, $\text{Add}(\mathcal{C}_\mathcal{P})$, we obtain a fusion category. In this category, every projection is a direct sum of \textit{minimal projections}, i.e., the simple objects.

Fix $n\in \mathbb{N}$ and an $n$th root of unity $\xi$. Then there is an affine $A_{2n-1}$ unshaded subfactor planar algebra, $\mathcal{P}^\xi$, with the diagrammatic presentation given in Definition \ref{PA presentaion} (\cite{Mol24}, Theorem 8.1). The adjoint, $*$, is compatible with flipping over an $x$-axis and switching any $V \leftrightarrow V^*$. The fusion category obtained from this planar algebra, which we will denote $\text{Add}({\mathcal{C}_{\mathcal{P}^\xi}})$, has simple objects consisting of alternating words in $p_1$ and $q_1$ (\cite{Mol24}, Theorem 8.1). Therefore, $\text{Add}({\mathcal{C}_{\mathcal{P}^\xi}})$ is equivalently the additive envelope of the monoidal subcategory $\langle p_1,q_1\rangle^{\xi}$ generated by $p_1$ and $q_1$. 

We can now  define a functor $F_{\xi}$ between $\mathcal{D}^{1,1,\xi}$ and $\langle p_1,q_1\rangle^{\xi}$ as shown in Definition \ref{def: pa-functor}. The goal of this subsection is to show that $F_{\xi}$ is a monoidal equivalence.


\begin{remark}
    In the category $\langle p_1, q_1 \rangle^\xi$, 
    we can view $p_1$ as a morphism in $\text{Hom}_{\langle p_1, q_1 \rangle^\xi}(p_1, p_1)$ since its top and bottom boundary are both red strands. Similarly, $q_1$ can also be viewed as a morphism.
\end{remark}

Choose a square root \(\xi^{1/2}\in\mathsf k^\times\).
\begin{definition}\label{def: pa-functor}
           Define a functor $F_\xi:\mathcal{D}^{1,1,\xi}\to \langle p_1,q_1 \rangle^\xi$ on generating objects and morphisms by
\begingroup
\renewcommand{\arraystretch}{1}
\begin{center}
\begin{adjustbox}{max width=\linewidth}
$ 
\begin{array}{
@{}r@{\;}c@{\;}l@{\quad}
r@{\;}c@{\;}l@{\quad}
r@{\;}c@{\;}l@{\quad}
r@{\;}c@{\;}l@{}
}

\multicolumn{3}{c}{}
&
\reddot & \longmapsto & p_1\ \text{(as an object)}
&
\bluedot & \longmapsto & q_1\ \text{(as an object)}
&
\multicolumn{3}{c}{}
\\[0.8em]

\redstrand & \longmapsto & p_1\ \text{(as a morphism)},
&
\redcup & \longmapsto & \paredcup\lowcomma
&
\redcap & \longmapsto & \paredcap\lowcomma
&
\rbjellydown & \longmapsto & \xi^{1/2}V,
\\[0.8em]

\bluestrand & \longmapsto & q_1\ \text{(as a morphism)},
&
\bluecup & \longmapsto & \pabluecup\lowcomma
&
\bluecap & \longmapsto & \pabluecap\lowcomma
&
\rbjellyup & \longmapsto & \xi^{-1/2}V^*.
\end{array}$ 
\end{adjustbox}
\end{center}
\endgroup
\end{definition}

\begin{lemma}\label{lem: pa-functor-well-defined}
    $F_{\xi}$ is well-defined.
\end{lemma}

\begin{proof}
    It is enough to prove that $F_{\xi}$ respects the defining relations of $\mathcal{D}^{1,1,\xi}$. 
 We show $F_\xi$ respects relation \ref{rel:redstrandcap}; \ref{rel:bluestrandcap} is analogous. The left hand side of \ref{rel:redstrandcap} maps to 
$         \paredstrand \raisebox{-1pt}{\makebox{$\paredcap$}}\lowperiod$
    Then, applying the relation \ref{rel:paredcupcap}, we obtain the image of the right hand side of \ref{rel:redstrandcap}.

    Next, we prove that $F_\xi$ respects relation \ref{rel:redbluejelly}. The functor $F_\xi$ maps the left hand side to $\xi^{1/2} (p_1 \otimes q_1 \otimes V)$. Using relation \ref{rel:pabluecupcap}, this diagram is equivalent to $\xi^{1/2}$ times 
$         \paFwelldefRthreea \lowperiod$ 
    Now we can use the second equality in relation \ref{rel:paclick} to obtain $\xi^{3/2}$ times
   $ 
       \paFwelldefRthreeb \lowperiod$
    Applying \ref{rel:paredcupcap} on the left two red strands, we obtain the image of the right hand side of \ref{rel:redbluejelly}.

    It can be immediately seen that relations \ref{rel:redcircle}, \ref{rel:bluecircle}, \ref{rel:twojellys}, \ref{rel:rcupcap}, \ref{rel:bcupcap}, and 
\ref{rel:jellysmultiplied} are respected by $F_\xi$ since each side of the relations map to the corresponding sides of relations \ref{rel:paredcircle}, \ref{rel:pabluecircle}, \ref{rel:paVadV}, \ref{rel:paredcupcap}, \ref{rel:pabluecupcap}, and \ref{rel:paVVad}, respectively.
    
It remains to show that $F_\xi$ respects relation \ref{rel:closedjellys}. The left hand side of \ref{rel:closedjellys} maps to $\xi$ times              
\vspace{-1.1cm}
\begin{align*}
    \patwoVs \lowperiod
\end{align*}
   Using the second equality in relation \ref{rel:paclick}, the above is equal to 
   \vspace*{-1.2cm}
\begin{align*}
    \paVVadclosed \lowperiod
\end{align*}   
Now we can planar isotope the $V^*$ on top of the $V$. By using relation \ref{rel:paVVad}, this diagram is equivalent to $2n$ circles alternating in red and blue. Relations \ref{rel:paredcircle}, \ref{rel:pabluecircle} give that this diagram evaluates to $\emptyset$, which is the image of the right hand side of relation \ref{rel:closedjellys}. Thus, $F_\xi$ respects all the defining relations of $\mathcal{D}^{1,1,\xi}$, so it is well-defined.
\end{proof}

\begin{lemma}\label{lem: pa-functor-onto}
    $F_{\xi}$ is essentially surjective on objects.
\end{lemma}

\begin{proof}
    The generating objects of $\mathcal{D}^{1,1,\xi}$ map to the generating objects of $\langle p_1,q_1\rangle^{\xi}$ so the proof is immediate.
\end{proof}

To show that $F_{\xi}$ is full and faithful, we need to check that the induced map, $\tilde{F_{\xi}}: \text{Hom}_{\mathcal{D}^{1,1,\xi}}(g,h)\to \text{Hom}_{\langle p_1,q_1\rangle^{\xi}}(w,w')$ is bijective, where $w, w'$ are words in $p,q$ and $g,h$ are words in $r,b$. Duality allows a diagram in $\text{Hom}_{\mathcal{D}^{1,1,\xi}}(g,h)$ to be turned into a diagram $\text{Hom}_{\mathcal{D}^{1,1,\xi}}(gh^{-1},e)$. Further, bending diagrams is linear and bijective, so equivalently, we can prove that the induced map, $\tilde{F_{\xi}}:\text{Hom}_{\mathcal{D}^{1,1,\xi}}(g,e)\to \text{Hom}_{\langle p_1,q_1\rangle^{\xi}}(w,\emptyset)$ is bijective for all words $g$ in $r,b$. 

  In \cite{Mol24} Theorem 8.1, the author shows that $p_1\otimes p_1\cong  q_1 \otimes q_1 \cong (p_1 \otimes q_1)^{n}\cong \emptyset$. Therefore, just like objects in $\mathcal{D}^{1,1,\xi}$, every object in $\langle p_1,q_1\rangle^{\xi}$ is isomorphic to a unique one in standard form, $w\cong (p\otimes q)^m \otimes p^\ell$, $0\leq m <n$, $\ell\in \{0,1\}$. We write $|w|:= 2m+\ell$ for the length of $w$ in its standard form. It is clear that $|g|=|F_{\xi}(g)|$. 
  
  Notice that Corollary \ref{cor: dim of Homs} implies that 
$\text{dim Hom}_{\mathcal{D}^{1,1,\xi}}(g,e)=\begin{cases}1, & g = \overline{e}\\ 0 &\text{ otherwise}\end{cases}$. If these match the dimensions of the Hom spaces in the image of $\tilde{F_{\xi}}$, then we only need to show $\tilde{F_{\xi}}$ is full.
  
\begin{lemma}
    Let $w$ be a word in $p_1,q_1$. Then $\text{dim Hom}_{\langle p_1,q_1 \rangle^{\xi}}(w,\emptyset)=\begin{cases} 1, & |w|=0 \\ 0 & \text{ otherwise.}\end{cases}$
\end{lemma}

\begin{proof}
    In \cite{Mol24}, the author proved that the isomorphism classes of simple objects in $\langle p_1,q_1\rangle^{\xi}$ are the words alternating in $p_1$ and $q_1$, starting in $p_1$ or $q_1$, of length 0 to $n$. (This is a result of Theorem 8.1 in \cite{Mol24} which states the principal graph is given by (19). The vertices of the principal graph give the isomorphism classes of the simple objects in the category).
    
    Since $p_1\otimes p_1\cong q_1\otimes q_1 \cong (p_1\otimes q_1)^{n}\cong \emptyset$, any word $w$ with $|w|>0$ must be isomorphic to a word alternating in $p_1$ and $q_1$, starting in $p_1$ or $q_1$, of length between 1 and $n$. Since these words are simple objects and not isomorphic to $\emptyset$, $\text{dim Hom}_{\langle p_1,q_1 \rangle^{\xi}}(w,\emptyset)=0$. On the other hand, if $w$ is a word with $|w|=0$, then $w\cong \emptyset$, which is a simple object, so $\text{dim Hom}_{\langle p_1,q_1\rangle^{\xi}}(w,\emptyset)=1$. 
\end{proof}

\begin{lemma}\label{lem: pa-functor-fully-faithful}
    $F_{\xi}$ is full and faithful.
\end{lemma}

\begin{proof}
    By the preceding lemma, we only need to show that $\tilde{F_{\xi}}$ is full. When $|w|>0$ the result is clear since $\text{dim Hom}_{\langle p_1,q_1 \rangle^{\xi}}(w,\emptyset)=0$. 

    Now let $w$ be a word in $p_1$ and $q_1$ with $|w|=0$. Let $g$ be the object of $\mathcal{D}^{1,1,\xi}$ that maps to $w$. Since $\text{dim Hom}_{\langle p_1,q_1 \rangle^{\xi}}(w,\emptyset)=1$, it is enough to check that the image of the induced map $\tilde{F_{\xi}}:\text{Hom}_{\mathcal{D}^{1,1,\xi}}(g,e)\to \text{Hom}_{\langle p_1,q_1\rangle^{\xi}}(w,\emptyset)$ is nontrivial.
    
    Since $|w|=0$, we also know that $|g|=0$. When $g=e$, the tensor unit maps to $\emptyset$, which has nonzero trace, so is nontrival. If $g\neq e$, then $g$ is isomorphic to $e$ by reduction to standard form. Choose an isomorphism $u_g:g\to e$ in $\mathscr D^{1,1,\xi}$. Then $F_{\xi}(u):w\to\emptyset$ is an isomorphism,  
    and hence is nonzero because
    $\mathrm{id}_{\emptyset}\neq 0$. 
\end{proof}

\begin{theorem}\label{thm: pa-functor-monoidal-equivalence}
    The functor $F_{\xi}:\mathcal{D}^{1,1,\xi}\to \langle p_1,q_1\rangle^{\xi}$ is a monoidal equivalence.
\end{theorem}

\begin{proof}
    As both categories are strict, this follows from Lemmas \ref{lem: pa-functor-well-defined}, \ref{lem: pa-functor-onto} and \ref{lem: pa-functor-fully-faithful}.
\end{proof}

\begin{corollary}
    \label{thm: spa-is-VecD2n}
    The category $\text{Add}(\mathcal{C}_{\mathcal{P}^\xi})$ is monoidally equivalent to $\text{Vec}_{\Dn}^{1,1,\xi}$. 
\end{corollary}

\begin{proof}
  Since $\text{Add}(\mathcal{C}_{\mathcal{P}^\xi})\simeq \text{Add}(\langle p_1,q_1\rangle^{\xi})$, this follows from Theorem \ref{thm: pa-functor-monoidal-equivalence} by taking the additive envelopes of both $\mathcal{D}^{1,1,\xi}$ and $\langle p_1,q_1\rangle^\xi$ .
\end{proof}

Through this equivalence, it can now be seen that the jellyfish elimination algorithm (Algorithm \ref{algorithm}) is the same evaluation algorithm outlined on page 35 of \cite{Mol24}.

\bibliographystyle{alpha}
\bibliography{bib}

\begin{thebibliography}{EGNO15}

\bibitem[BK01]{BK}
Bojko Bakalov and Alexander Kirillov.
\newblock {\em Lectures on Tensor Categories and Modular Functors}, volume~21 of {\em University Lecture Series}.
\newblock American Mathematical Society, 2001.

\bibitem[BPMS12]{BMPS12}
Stephen Bigelow, Emily Peters, Scott Morrison, and Noah Snyder.
\newblock Constructing the extended {H}aagerup planar algebra.
\newblock {\em Acta Math.}, 209(1):29--82, 2012.

\bibitem[Cze24]{Czenky}
Agustina Czenky.
\newblock Diagramatics for cyclic pointed fusion categories.
\newblock {\em Journal of Pure and Applied Algebra}, 228(12):107752, 2024.

\bibitem[DP26]{dp}
Chumeng Di and Anup Poudel.
\newblock A diagrammatic presentation for every pivotal pointed fusion category.
\newblock {\em \texttt{arXiv:2609.27295}}, 2026.

\bibitem[dWP95]{deWildPropitius1995}
Mark de~Wild~Propitius.
\newblock {\em Topological Interactions in Broken Gauge Theories}.
\newblock PhD thesis, Universiteit van Amsterdam, 1995.

\bibitem[EGNO15]{EGNO}
Pavel Etingof, Shlomo Gelaki, Dmitri Nikshych, and Victor Ostrik.
\newblock {\em Tensor categories}, volume 205 of {\em Mathematical Surveys and Monographs}.
\newblock American Mathematical Society, Providence, RI, 2015.

\bibitem[EK98]{EvansKawahigashi}
David~E. Evans and Yasuyuki Kawahigashi.
\newblock {\em Quantum Symmetries on Operator Algebras}.
\newblock Oxford Mathematical Monographs. Oxford University Press, Oxford, 1998.

\bibitem[ENO05]{ENO1}
Pavel Etingof, Dmitri Nikshych, and Victor Ostrik.
\newblock On fusion categories.
\newblock {\em Ann. of Math. (2)}, 162(2):581--642, 2005.

\bibitem[Hua05]{Huang05}
Yi-Zhi Huang.
\newblock Vertex operator algebras, the verlinde conjecture, and modular tensor categories.
\newblock {\em Proceedings of the National Academy of Sciences}, 102(15):5352--5356, 2005.

\bibitem[IMP13]{IMP}
Masaki Izumi, Scott Morrison, and David Penneys.
\newblock Fusion categories between {$\mathcal{C} \boxtimes \mathcal{D}$} and {$\mathcal{C} * \mathcal{D}$}.
\newblock {\em \texttt{arXiv:1308.5723}}, 2013.

\bibitem[Kas95]{Kassel}
Christian Kassel.
\newblock {\em Quantum Groups}, volume 155 of {\em Graduate Texts in Mathematics}.
\newblock Springer-Verlag, New York, 1995.

\bibitem[Lan98]{MacLane}
Saunders~Mac Lane.
\newblock {\em Categories for the Working Mathematician}, volume~5 of {\em Graduate Texts in Mathematics}.
\newblock Springer, New York, 2 edition, 1998.

\bibitem[Mol24]{Mol24}
Melody Molander.
\newblock Skein theory for affine {$A$} subfactor planar algebras, 2024.
\newblock arxiv: 2410.05519.

\bibitem[MPS10]{MPS10}
Scott Morrison, Emily Peters, and Noah Snyder.
\newblock Skein theory for the {$D_{2n}$} planar algebras.
\newblock {\em J. Pure Appl. Algebra}, 214(2):117--139, 2010.

\bibitem[MS89]{Moore}
Gregory~W. Moore and Nathan Seiberg.
\newblock Classical and quantum conformal field theory.
\newblock {\em Communications in Mathematical Physics}, 123:177--254, 1989.

\bibitem[Pen]{Dave}
David Penneys.
\newblock 3-cocycle representatives for the dihedral group?
\newblock MathOverflow.
\newblock URL: https://mathoverflow.net/questions/156985/3-cocycle-representatives-for-the-dihedral-group-d-2n.

\bibitem[Pet10]{Pet10}
Emily Peters.
\newblock A planar algebra construction of the {H}aagerup subfactor.
\newblock {\em Internat. J. Math.}, 21(8):987--1045, 2010.

\bibitem[TH23]{TeoHu2023}
Jeffrey C.~Y. Teo and Yichen Hu.
\newblock Dihedral twist liquid models from emergent {M}ajorana fermions.
\newblock {\em Quantum}, 7:967, 2023.

\bibitem[Tur94]{Turaev}
Vladimir~G. Turaev.
\newblock {\em {Quantum invariants of knots and three manifolds}}, volume~18.
\newblock 1994.

\end{thebibliography}
\end{document}